\documentclass[a4paper,11pt]{amsart}

\usepackage{graphicx}
\usepackage{amsmath}
\usepackage{amssymb}
\usepackage[active]{srcltx}
\usepackage{hyperref}

\newtheorem{thm}{Theorem}[section]
\newtheorem{lem}[thm]{Lemma}%
\newtheorem{prop}[thm]{Proposition}%
\newtheorem{cor}[thm]{Corollary}%

\newenvironment{thmbis}[1]
  {\renewcommand{\thethm}{\ref{#1}$'$}%
   \addtocounter{thm}{-1}%
   \begin{thm}}
  {\end{thm}}

\theoremstyle{definition}
\theoremstyle{remark}
\newtheorem{remark}{Remark}[section] %

\theoremstyle{plain}

\def\arsinh{\operatorname{arsinh}}

\def\vecm{{\text{\boldmath$m$}}}
\def\vecn{{\text{\boldmath$n$}}}

\def\vecv{{\text{\boldmath$v$}}}

\def\vecw{{\text{\boldmath$w$}}}

\def\vecy{{\text{\boldmath$y$}}}

\def\vecalf{{\text{\boldmath$\alpha$}}}

\def\veceta{{\text{\boldmath$\eta$}}}

\def\vecxi{{\text{\boldmath$\xi$}}}

\def\scrR{{\mathcal R}}

\def\scrU{{\mathcal U}}

\def\scrY{{\mathcal Y}}

\def\fb{{\mathfrak b}}
\def\fb{{b}}

\def\fM{{\mathfrak M}}

\def\fR{{\mathfrak R}}

\def\Ad{\operatorname{Ad}}

\def\C{\operatorname{C{}}}
\def\b{\operatorname{b{}}}

\def\L{\operatorname{L{}}}

\def\GL{\operatorname{GL}}

\def\SL{\operatorname{SL}}

\def\SO{\operatorname{SO}}

\def\PSL{\operatorname{PSL}}

\def\T{\mathbb T}

\def\sgn{\operatorname{sgn}}

\def\supp{\operatorname{supp}}

\def\W{\operatorname{W{}}}

\def\ord{\operatorname{ord}}

\def\mod{\operatorname{mod}}

\def\trans{\,^\mathrm{t}\!}

\def\Onder#1#2#3#4#5{#1 \setbox0=\hbox{$#1$}\setbox1=\hbox{$#2$}
       \dimen0=.5\wd0 \dimen1=\dimen0 \dimen2=\dp0 \dimen3=\dimen2
       \advance\dimen0 by .5\wd1 \advance\dimen0 by -#4
       \advance\dimen1 by -.5\wd1 \advance\dimen1 by -#4
       \advance\dimen2 by -#3 \advance\dimen2 by \ht1
       \advance\dimen2 by 0.3ex \advance\dimen3 by #5
        \kern-\dimen0\raisebox{-\dimen2}[0ex][\dimen3]{\box1}
       \kern\dimen1}

\newcommand{\dn}{\mathbf{d}}
\newcommand{\svl}{\text{\boldmath$\ell$}}
\newcommand{\lsl}{\mathfrak{sl}}

\newcommand{\ig}{\mathfrak{g}}

\newcommand{\GaG}{\Gamma\backslash G}
\newcommand{\wh}{\widehat}
\newcommand{\tf}{\widetilde f}

\newcommand{\Q}{\mathbb{Q}}
\newcommand{\R}{\mathbb{R}}
\newcommand{\Z}{\mathbb{Z}}

\newcommand{\kk}{\mathbf{k}}

\newcommand{\col}{\: : \:}

\newcommand{\bn}{\mathbf{0}}

\newcommand{\ve}{\varepsilon}

\newcommand{\matr}[4]{\left( \begin{matrix} #1 & #2 \\ #3 & #4 \end{matrix} \right) }
\newcommand{\smatr}[4]{\left( \begin{smallmatrix} #1 & #2 \\ #3 & #4 \end{smallmatrix} \right) }

\title{An effective Ratner equidistribution result for $\SL(2,\R)\ltimes\R^2$}
\author{Andreas Str\"ombergsson}
\address{Department of Mathematics, Box 480, Uppsala University,
SE-75106 Uppsala, Sweden\newline
\rule[0ex]{0ex}{0ex} \hspace{8pt}{\tt astrombe@math.uu.se}}
\thanks{The work was conducted while Str\"ombergsson was a Royal Swedish Academy of Sciences Research Fellow supported by
a grant from the Knut and Alice Wallenberg Foundation.}

\begin{document}

\begin{abstract}
Let $G=\SL(2,\R)\ltimes\R^2$ %
be the affine special linear
group of the plane, and set $\Gamma=\SL(2,\Z)\ltimes\Z^2$. 
We prove a polynomially
effective asymptotic equidistribution result for
the orbits of a 1-dimensional, non-horospherical unipotent flow on $\GaG$. %
\end{abstract}

\maketitle

\section{Introduction}\label{INTROSEC}

In the theory of unipotent flows on homogeneous spaces,
a fundamental role is played by 
the theorems by M.\ Ratner on measure rigidity, topological rigidity, 
and orbit equidistribution, \cite{Ratner91}, \cite{mR91};
these results also appear %
as a crucial ingredient in numerous, and surprisingly diverse, applications.
See \cite{WitteMorris} and \cite{KSS} %
for expositions and references;
some %
more recent works where important use is made of Ratner's theorems are
\cite{EO}, \cite{GO}, \cite{Shah09}, \cite{Shah10}, \cite{partI}, \cite{quasicrystals},
to just mention a few.

In the last decade %
there has been an increased interest in obtaining %
\textit{effective} versions of 
Ratner's results, that is, to provide an explicit rate of density or equidistribution
for the orbits of a unipotent flow.
This problem was raised for example in \cite[Probl.\ 7]{Margulis}.
There are two general cases where it has been known for a fairly long time that effective results may be proved,
namely when the group generating the flow %
is either horospherical or ``large'' in an appropriate sense
(cf.\ \cite[\S 1.5.2]{EMV} for a discussion; compare also p.\ \pageref{MIXINGTECHNIQUE} below).
Recently, however, some new important cases have been established:
Green and Tao \cite{GreenTao} have proved effective equidistribution
of polynomial orbits on nilmanifolds;
this is an important input %
in their work on linear equations in primes
\cite{GreenTao3}, \cite{GreenTao2}.
Moreover, Einsiedler, Margulis and Venkatesh \cite{EMV} have 
proved effective equidistribution
for large closed orbits of semisimple groups 
on homogeneous spaces;
see also Mohammadi \cite{mohammadi2012} for a more explicit result in the special case of
closed $\SO(2,1)$-orbits in $\SL(3,\Z)\backslash\SL(3,\R)$.
Recently also Lindenstrauss and Margulis \cite{LM}
have obtained an effective density-type result for \textit{arbitrary} 
$\SO(2,1)$-orbits in $\SL(3,\Z)\backslash\SL(3,\R)$,  %
and used this to 
give an effective proof of a theorem of Dani and Margulis regarding the values of 
indefinite ternary quadratic forms at primitive integer vectors.

\vspace{5pt}

Our purpose in the present paper is to establish effective Ratner equidistribution in a new particular setting:
We let $G$ be the semidirect product group $G=\SL(2,\R)\ltimes\R^2$ with multiplication law
\begin{align*}
(M,\vecv)(M',\vecv')=(MM',\vecv M'+\vecv').
\end{align*}
Let $\Gamma=\SL(2,\Z)\ltimes\Z^2$ and $X=\GaG$,
and consider the flow on $X$ which is generated by right multiplication
by the (Ad-)unipotent 1-parameter subgroup $U^\R=\{U^t\col t\in\R\}$, where 
\begin{align*}
U^t=\biggl(\matr1t01,(0,0)\biggr).
\end{align*}
The Ratner measure rigidity and equidistribution for this particular flow,
and closely related ones, %
have %
found several applications in number theory and in mathematical physics;
cf.\
\cite[Remark 4]{SV},
\cite{MarklofpaircorrI}, \cite{MarklofpaircorrII}, \cite{Elkies04}, \cite{Marklofmeanssquare}, \cite{partI}, 
\cite[Thm.\ 1.10]{MM},
\cite{elbazmv1}, \cite{elbazmv2}; we discuss this further in Section~\ref{APPLSEC}. %
Note that $\{U^t\}$ on $X$ is a 
\textit{1-dimensional, non-horospherical} unipotent flow on a non-solvable homogeneous space.
As far as we are aware, there is only one previous setting of this kind where effective equidistribution has been established;
namely, the results by Venkatesh \cite[\S 3.1]{Vsparse} and Sarnak and Ubis \cite[Thm.\ 4.11]{pSaU2011} 
for orbits of the discrete horocycle flow %
can be viewed as giving effective equidistribution for the flow generated by
$U^t=(\smatr 1t01,s^{-1}t)$ (any fixed $s>0$) in $(\Gamma'\times\Z)\backslash(\SL(2,\R)\times\R)$,
with either $\Gamma'=\SL(2,\Z)$ or $\Gamma'$ a cocompact subgroup of $\SL(2,\R)$.

\vspace{5pt}

The group $G=\SL(2,\R)\ltimes\R^2$ can be viewed as the group of area and orientation preserving affine maps 
of the plane $\R^2$, with the action given by
\begin{align*}
\vecy(M,\vecv):=\vecy M+\vecv,\qquad
\forall (M,\vecv)\in G,\:\vecy\in\R^2,
\end{align*}
and a central property of $X=\GaG$ is that it can be naturally identified with the space of translates of unimodular lattices 
in $\R^2$, through $\Gamma g\mapsto \Z^2g=\{\vecm g\col\vecm\in\Z^2\}$.
Then the subspace of (non-translated) %
lattices becomes identified with 
$X'=\Gamma'\backslash G'$, where $G'=\SL(2,\R)$,
which we always view as a subgroup of $G$ through $M\mapsto(M,\bn)$, and $\Gamma'=\Gamma\cap G'=\SL(2,\Z)$.
Note that $X'$ is an embedded submanifold of $X$.
Furthermore, %
$U^\R$ is contained in $G'$,
and the flow $U^\R$ on $X'$ is the standard horocycle flow.
There is also a natural projection $D:G\to G'$ sending $(M,\vecv)$ to $M$, 
which makes $X$ into a torus fiber bundle over $X'$.
We write $D$ also for the projection map $X\to X'$.
Note that the embeddings $G'\subset G$ and $X'\subset X$ are sections of $D$.
In the language of lattice translates, the fiber %
over a lattice $L\in X'$ equals the torus $\R^2/L$ consisting of all translates of $L$.

Let $\mu$ be the (left and right invariant) Haar measure on $G$,
normalized so as to induce a probability measure on $X$,
which we also denote by $\mu$.
Then $\mu':=D_*\mu$ is the Haar measure on $G'$ which induces a probability measure on $X'$.

We will start by discussing the case of $U^\R$-orbits in $X$ which project to %
\textit{closed} orbits in $X'$;
we then turn to the case of general $U^\R$-orbits is Section \ref{INTROGENERALORBITS}.

\subsection{Lifts of pieces of closed horocycles}
\label{INTROCLOSED}

Set 
\begin{align*}
\Phi^t=\matr{e^{-t/2}}00{e^{t/2}}\in G'\subset G\qquad(t\in\R).
\end{align*}
We also write $1_2=\smatr1001$.
For given $\vecxi\in\R^2$ and $t\in\R$, we consider
pieces of the $U^\R$-orbit through the point $\Gamma\,(1_2,\vecxi)\,\Phi^t\in X$.
These are %
exactly those $U^\R$-orbits in $X$ which project to closed orbits, i.e.\ closed horocycles, in $X'$.
From the relation 
\begin{align}\label{UPHICOMM}
U^x\Phi^t=\Phi^t U^{e^tx}
\end{align}
we see that $\Gamma\,(1_2,\vecxi)\,\Phi^tU^\R=\Gamma\,(1_2,\vecxi)\,U^\R\Phi^t$,
that is, the $U^\R$-orbit through $\Gamma\,(1_2,\vecxi)\,\Phi^t$ is obtained as 
the $\Phi^t$-push-forward of the $U^\R$-orbit through $\Gamma\,(1_2,\vecxi)\,$.
It also follows from \eqref{UPHICOMM}
that the projected orbit, $x\mapsto\Gamma'\Phi^tU^x\in X'$, %
has period $e^t$ with respect to $x$.
It is well-known that these closed horocycles,
and more generally the $\Phi^t$-push-forwards of any fixed segment
$\{\Gamma' U^x\col x\in[\alpha,\beta]\}$, %
become asymptotically equidistributed in $(X',\mu')$ as $t\to\infty$.
These facts are also known with precise rates;
cf.\ \cite{pS80}, \cite{Hejhal}, \cite{Str}, \cite{FF}.
As to the orbits in $X$,
it turns out that the $\Phi^t$-push-forwards of a fixed segment
$\{\Gamma(1_2,\vecxi)U^x\col x\in[\alpha,\beta]\}$ become asymptotically equidistributed 
in $(X,\mu)$ as $t\to\infty$
\textit{if and only if $\vecxi$ is irrational}. %
We state the non-trivial direction of this implication as Theorem \ref{NONEFFECTIVETHM} below;
it is a special case of a theorem of Shah, \cite[Thm.\ 1.4]{Shah}
(cf.\ \cite[proof of Thm.\ 5.2]{partI}),
and also a special case of Elkies and McMullen, \cite[Thm.\ 2.2]{Elkies04}.
Both proofs depend crucially on Ratner's classification of invariant measures.
(See \cite[\S 3]{Eskinlecturenotes} for a discussion of the proof of Ratner's theorem in exactly 
our setting with $G=\SL(2,\R)\ltimes\R^2$, $\Gamma=\SL(2,\Z)\ltimes\Z^2$.)

\begin{thm}\label{NONEFFECTIVETHM}
(\cite{Shah} or \cite{Elkies04})
Fix any $\vecxi\in\R^2$ with at least one irrational coordinate,
i.e.\ $\vecxi\notin\Q^2$.
Then the $\Phi^t$-push-forwards of any fixed portion of the orbit
$\Gamma\,(1_2,\vecxi)\,U^\R$
become asymptotically equidistributed in $(X,\mu)$ as $t\to\infty$.
In other words, for any fixed $\alpha<\beta$ and any
bounded continuous function $f:X\to\R$,
\begin{align}\label{NONEFFECTIVETHMRES}
\lim_{t\to\infty}\frac1{\beta-\alpha}
\int_\alpha^\beta f\Bigl(\Gamma\,(1_2,\vecxi)\,U^x\Phi^t\Bigr)\,dx
=\int_X f\,d\mu. %
\end{align}
\end{thm}

To see that the assumption $\vecxi\notin\Q^2$ in Theorem \ref{NONEFFECTIVETHM} is a necessary condition,
set, %
for any positive integer $q$,
\begin{align*}
X_q:=\bigl\{\Gamma(1_2,\vecv)M\col M\in G',\:
\vecv\in\Q^2,\:\dn(\vecv)=q\bigr\},
\end{align*}
where for any vector $\vecv\in\Q^2$ we write $\dn(\vecv)$ for its denominator, 
i.e.\ the smallest positive integer $d$ such that 
$\vecv\in d^{-1}\Z^2$. %
Then $X_q$ is a closed embedded $3$-dimensional submanifold of $X$;
this is an easy consequence of the fact that %
$\{\vecv\in\Q^2\col \dn(\vecv)=q\}$ is an invariant subset for the action of $\Gamma$ on $\R^2$.
Note in particular that $X_1=X'$.
Now if $\vecxi\in\Q^2$ then $\Gamma\,(1_2,\vecxi)\,U^\R\Phi^t\subset X_{\dn(\vecxi)}$
holds for every $t$; hence the orbit certainly cannot become equidistributed in $(X,\mu)$. %

The map $G'\ni M\mapsto\Gamma\bigl(1_2,(0,q^{-1})\bigr)M\in X_q$
gives an identification of $X_q$ with the homogeneous space
$\Gamma_1(q)\backslash G'$, where $\Gamma_1(q)$ is the 
congruence subgroup
\begin{align*}
\Gamma_1(q)=\bigl\{\smatr abcd\in\Gamma'\col a\equiv d\equiv1\:
(\text{mod } q),\: c\equiv 0\:(\text{mod } q)\bigr\}.
\end{align*}
(To see this, note that 
$\Gamma$ acts transitively on $\{\vecv\in\Q^2\col \dn(\vecv)=q\}$.)
If $\vecxi\in\Q^2$ with $\dn(\vecxi)=q$ then the curves studied in
Theorem \ref{NONEFFECTIVETHM} correspond to pieces of closed horocycles
in $\Gamma_1(q)\backslash G'$, and hence
as $t\to\infty$ they go asymptotically equidistributed \textit{in $X_q$},
i.e.\ in place of \eqref{NONEFFECTIVETHMRES} we have
\begin{align*}
\lim_{t\to\infty}\frac1{\beta-\alpha}
\int_\alpha^\beta f\Bigl(\Gamma\,(1_2,\vecxi)\,U^x\Phi^t\Bigr)\,dx
=\int_{X_q} f\,d\mu_q, %
\end{align*}
where $\mu_q$ is the measure which corresponds to Haar measure on 
$G'$, normalized to give a probability measure on
$\Gamma_1(q)\backslash G'$
(cf., e.g., \cite{FF}).

\vspace{5pt}

The main result of the present paper is  Theorem~\ref{MAINABTHM} below, which is an \textit{effective} version of 
Theorem \ref{NONEFFECTIVETHM}.
It is clear from the preceding discussion that the rate of convergence
in \eqref{NONEFFECTIVETHMRES} is necessarily quite sensitive to the Diophantine properties of the vector $\vecxi$.

One should note %
that the flow $\{\Phi^t\}$ on $X$ is
Anosov, with unstable directions generated by\label{MIXINGTECHNIQUE}
the flows $U^\R$ and $(1_2,(0,\R))$
and stable directions generated by the flows
$\bigl(\smatr10{\R}1,(0,0)\bigr)$ and $(1_2,(\R,0))$.
In fact, for any fixed %
metric on $X$ coming from a left
invariant Riemannian metric on $G$,
the tangent vectors in the direction of $U^\R$ are expanded at
a rate $e^t$ by the flow $\Phi^t$ (cf.\ \eqref{UPHICOMM}),
the tangent vectors in the direction of $(1_2,(0,\R))$
are expanded at a rate $e^{t/2}$,
while %
vectors in the direction of 
$\bigl(\smatr10{\R}1,(0,0)\bigr)$ are contracted at a rate
$e^{-t}$ and those %
in the direction of $(1_2,(\R,0))$
are contracted at a rate $e^{-t/2}$.
If, in place of \textit{1-dimensional}
averages along $U^\R$-orbits, %
we would instead consider \textit{2-dimensional} averages
taken over some bounded open subset of the unstable manifold, %
then there exists a by now standard approach to establishing effective
results by using mixing properties of the flow $\Phi^\R$;
the origin of this technique can be
traced back to the thesis of Margulis, \cite{Margulisthesis},
where it was used in the context of general Anosov flows.
However, it seems %
that %
this technique cannot be carried over to the 1-dimensional averages which we consider;
instead our proof relies on Fourier analysis %
and methods from number theory, in particular Weil's bound on Kloosterman sums.

We now state Theorem \ref{MAINABTHM}.
Let $\C_{\b}^k(X)$ be the space of $k$ times continuously differentiable functions on $X$ 
whose all left invariant derivatives
up to order $k$ are bounded.
Choose, once and for all, a norm $\|\cdot\|_{\C_{\b}^k}$ on $\C_{\b}^k(X)$ involving the supremum norms 
of all these derivatives. %
(For definiteness, we fix a precise choice of $\|\cdot\|_{\C_{\b}^k}$; cf.\ \eqref{CKNORMDEF} below.)
Set
\begin{align*}
a(y)=\Phi^{-\log y}=\matr{\sqrt y}00{1/\sqrt y}\qquad\text{for }\: y>0.
\end{align*}
(As a motivation, note that  %
$U^xa(y)(i)=x+iy$, for the standard action of 
$G'$ on the Poincar\'e upper half plane model of the hyperbolic plane.)
For $x\in\R$ we write $\langle x\rangle$ for the distance %
to the nearest integer;
$\langle x\rangle=\min_{n\in\Z}|x-n|$.
For any $\vecxi=(\xi_1,\xi_2)\in\R^2$, $L>0$ and $y>0$ we set
\begin{align}\label{MYXILDEF}
\fb_{\vecxi,L}(y):=\max_{q\in\Z^+}\min\Bigl(\frac1{q^2},\frac{\sqrt y}{Lq\langle q\xi_1\rangle},
\frac{\sqrt y}{q\langle q\xi_2\rangle}\Bigr).
\end{align}
(Convention: if $\langle q\xi_1\rangle=0$ or $\langle q\xi_2\rangle=0$ then the corresponding entry
is removed from the minimum; in particular if both $\langle q\xi_1\rangle=\langle q\xi_2\rangle=0$ then
the minimum equals $1/q^2$.)
Note that the entry $1/q^2$ ensures that the maximum is attained, and
$0<\fb_{\vecxi,L}(y)\leq1$ for all $y,\vecxi,L$;
furthermore, $\fb_{\vecxi,L}(y)$ depends continuously on $y,\vecxi,L$.
\enlargethispage{15pt}

\begin{thm}\label{MAINABTHM}
Given any $\ve>0$, there exists a constant $C>0$ such that, for 
any $f\in \C_{\b}^8(X)$ and any $\alpha<\beta$, $\vecxi\in\R^2$ and $0<y<1$, 
\begin{align}\label{MAINABTHMRES}
\biggl|\frac1{\beta-\alpha}\int_\alpha^\beta f\Bigl(\Gamma\,(1_2,\vecxi)\,U^xa(y) %
\Bigr)\,dx
-\int_{\GaG} f\,d\mu\biggr|
\leq C\|f\|_{\C_{\b}^{8}}\,\frac{L}{\beta-\alpha}\bigl(\fb_{\vecxi,L}(y)+y^{\frac14}\bigr)^{1-\ve},
\end{align}
where $L=\max(1,|\alpha|,|\beta|)$. 
\end{thm}

Let us make some comments on this result.
First of all, note that for any fixed $\vecxi\in\R^2$ and $L>0$, we have $\lim_{y\to0}\fb_{\vecxi,L}(y)=0$
if (and only if) $\vecxi\notin\Q^2$.
Hence Theorem \ref{MAINABTHM} is indeed an effective version of Theorem \ref{NONEFFECTIVETHM}.

In order to discuss the rate of decay of our bound as $y\to0$, we recall the following definition:
We say that a vector $\vecxi\in\R^2$ is of \textit{(Diophantine) type} $K$ if there is some 
constant $c>0$ such that $\|\vecxi-q^{-1}\vecm\|>cq^{-K}$ for all $q\in\Z^+$ and $\vecm\in\Z^2$.
The smallest possible value for $K$ is $K=\frac32$, and it is known that
Lebesgue-almost all $\vecxi\in\R^2$ are of type $K=\frac32+\ve$ for any $\ve>0$.
In fact, by a result of Jarnik \cite{Jarnik}, for any $K\geq\frac32$,
the set of those $\vecxi\in\R^2$ which are not of type $K$ has
Hausdorff dimension $3/K$.
Now from the definition \eqref{MYXILDEF} one easily verifies that, for any fixed $\vecxi$ and $L$ and any given $\delta>0$,
we have $\fb_{\vecxi,L}(y)\ll y^\delta$ as $y\to0$
if and only if $\delta\leq\frac12$ and $\vecxi$ is of type $K=\delta^{-1}$.
Hence we get:
\begin{cor}\label{MAINABTHMCOR}
For any $\ve>0$, $f\in\C_{\b}^8(X)$, $\alpha<\beta$ and any $\vecxi\in\R^2$ of Diophantine type $K\geq\frac32$,
there is a constant $C=C(\ve,f,\alpha,\beta,\vecxi)>0$ such that
\begin{align}\label{MAINABTHMCORRES}
\biggl|\frac1{\beta-\alpha}\int_\alpha^\beta f\Bigl(\Gamma\,(1_2,\vecxi)\,U^xa(y) %
\Bigr)\,dx
-\int_{\GaG} f\,d\mu\biggr|
<C y^{\min(\frac14,\frac1K)-\ve},\qquad\forall 0<y<1.
\end{align}
\end{cor}
In particular, 
in view of Jarnik's result, we obtain the rate $y^{\frac14-\ve}$ for any fixed $\vecxi\in\R^2$
away from a set of Hausdorff dimension $\frac34$.
It seems that the exponent $\frac14$ in \eqref{MAINABTHMCORRES} is not the best possible,
and that 
optimally one might hope to %
prove that the left hand side of
\eqref{MAINABTHMRES} decays with a rate $y^{\frac12-\ve}$ as $y\to0$,
for any $\vecxi$ satisfying an appropriate Diophantine condition;
cf.\ Remark \ref{MAINTHMWEAKERPROPREM} below.

Regarding the dependence of our bound on $\alpha$, $\beta$, 
we remark that we could have chosen to state Theorem \ref{MAINABTHM} with the extra restriction
$-1\leq\alpha<\beta\leq1$ (viz.,\ $L=1$);
the general case can be deduced aposteriori from %
that case by 
using invariance under $U^n\in\Gamma$, $n\in\Z$,
and splitting $[\alpha,\beta]$ into subintervals of length $\leq1$;
this will be seen in Section~\ref{MAJORANTPROPSEC} 
where we discuss basic properties of the majorant function $\fb_{\vecxi,L}(y)$.
We have not given any special attention to the case of $\beta-\alpha$ becoming small in our proof of Theorem~\ref{MAINABTHM}, 
and there seems to be room for improvement in this direction.
(Cf.\ \cite{Str}, where the case of both $\beta-\alpha$ and $y$ being small is considered for the case 
of pieces of closed horocycles in $X'$ and other homogeneous spaces of $\SL(2,\R)$.) %
Also we have made no effort to optimize the dependence on $f$ in Theorem \ref{MAINABTHM}.

A point to note is that the orbit $\Gamma(1_2,\vecxi)U^\R$
is \textit{closed} in $X$ if and only if $\xi_1\in\Q$, and in this case its period equals %
the denominator of $\xi_1$; a corresponding fact also holds for any $\Phi^t$-push-forward of that orbit.
This is to some extent reflected in the bound \eqref{MAINABTHMRES}:
for fixed $y,\vecxi$, we have $\lim_{L\to\infty} \fb_{\vecxi,L}(y)=0$ if and only if $\xi_1\notin\Q$.

\subsection{General orbits}
\label{INTROGENERALORBITS}

We now turn to the case of arbitrary $U^\R$-orbits. %
According to Ratner's equidistribution theorem \cite{mR91}, %
\textit{every} $U^\R$-orbit in $X$ has a closure which is homogeneous.
Stated in more detail, for any given $x=\Gamma g\in X$ ($g\in G$) there exists a closed connected subgroup $H\subset G$ such
that $U^\R\subset H$, $\Gamma\cap gHg^{-1}$ is a lattice in $gHg^{-1}$, 
and the closure of $xU^\R$ in $X$ equals $xH=\Gamma\backslash\Gamma gH$.
Furthermore the orbit $xU^\R$ is then asymptotically equidistributed in $xH$ with respect to $\nu_H$,
the $H$-invariant Borel probability measure on $X$ supported on $xH$ \cite[Thm.\ B]{mR91}.

For our specific space $X$ it is fairly easy to list explicitly those subgroups $H$ which can occur, and in particular
to give a precise criterion for when %
$xU^\R=\Gamma gU^\R$ is asymptotically equidistributed in $(X,\mu)$.
Clearly a necessary condition for the latter %
is that the projected orbit $D(xU^\R)$ should be equidistributed in $X'$.
By a theorem of Dani \cite{Dani82} (a very special case of Ratner's \cite{mR91}), 
$D(xU^\R)$ is equidistributed in $X'$ unless $D(xU^\R)$ is a closed horocycle,
viz., unless the lattice $\Z^2 D(g)$ contains some point along the line $(0,\R):=\{0\}\times\R$ other than the origin.
Assuming that $D(xU^\R)$ is equidistributed in $X'$, one finds
(cf.\ the discussion in \cite[\S 2.6]{Elkies04} applied to the measure $\nu_H$;
see in particular \cite[Cor.\ 2.11 and Cor.\ 2.12, corrected]{Elkies04})
that either $H=G$ and $xH=X$, or else 
there is some $\beta\in\R$ such that $(0,\beta)g^{-1}\in\Q^2$, and then
$H=\bigl(1_2,-(0,\beta)\bigr)G'\bigl(1_2,(0,\beta)\bigr)$ and $xH=X_q\bigl(1_2,(0,\beta)\bigr)$,
where $q=\dn\bigl((0,\beta)g^{-1}\bigr)$.
(For clarity, note that in the second case, $\beta$ is uniquely determined. %
Indeed, if the point set $\Q^2g$ intersects the line $(0,\R)$ in more than one point then
by subtraction $\Q^2D(g)$ contains a non-zero point on $(0,\R)$; hence so does the lattice
$\Z^2D(g)$, contradicting our assumption that $D(xU^\R)$ is equidistributed in $X'$.)

In particular we have:
\begin{thm}\label{NONCLOSEDINEFFTHM} 
(Special case of Ratner, \cite{mR91}.)
Fix any $g\in G$ satisfying $\Z^2D(g)\cap(0,\R)=\{\bn\}$ %
and $(0,\beta)g^{-1}\notin\Q^2$ for all $\beta\in\R$.
Then the orbit $\Gamma gU^\R$ is asymptotically equidistributed in $(X,\mu)$.
In other words, for any bounded continuous function $f$ on $X$,  %
$\frac1T\int_0^T f(\Gamma gU^t)\,dt\to\int_X f\,d\mu$ as $T\to\infty$.
\end{thm}

As an application of our main result, Theorem \ref{MAINABTHM}, 
and using the technique of approximating nonclosed horocycles by pieces of closed horocycles
(cf.\ %
\cite{pSaU2011}),
we will prove an effective version of 
Theorem \ref{NONCLOSEDINEFFTHM}; see Theorem \ref{NONCLOSEDTHM2} below.
Before stating it, it is useful to recall the effective equidistribution result for horocycles in $X'$ proved in \cite{iha}
(viz., an effective version of Dani's theorem \cite{Dani82});
cf.\ also \cite{Burger}, \cite{FF}, \cite{pSaU2011}.
For $g\in G'$ we write $\svl(g)>0$ for the Euclidean length of the shortest non-zero vector 
in the lattice $\Z^2g$.
Note that $\svl(\gamma g)=\svl(g)$ for all $\gamma\in\Gamma'$, i.e.\ $\svl$ is a function on $X'$;
in fact $\svl(g)$ equals the inverse square root of the invariant height function $\scrY_{\Gamma'}(g)$ 
used in \cite{iha}.
More generally for $g\in G$ we set $\svl(g)=\svl(D(g))$.
Finally for $g\in G$ and $T>0$ we set
\begin{align}\label{YGTDEF}
y_g(T):=T^{-1}\svl(g\,a(T))^{-2}.
\end{align}
\begin{thm}\label{XPEFFTHM}
(\cite[Thm.\ 1]{iha}; cf.\ also \cite{pSaU2011})
There exists an absolute constant $C>0$ such that, for any $g\in G'$, $T\geq1$, and any
$f\in\C_{\b}^4(X')$:
\begin{align}\label{XPEFFTHMRES}
\biggl|\frac1T\int_0^T f\bigl(\Gamma' gU^t\bigr)\,dt-\int_{X'}f\,d\mu'\biggr|
\leq C\|f\|_{\C_{\b}^4}\, y_g(T)^{\frac12}\log^3(2+y_g(T)^{-1}).
\end{align}
\end{thm}
Note that for given $g\in G'$, $\lim_{T\to\infty}y_g(T)=0$ holds if and only if the horocycle $\Gamma' gU^\R$ is not closed;
hence Theorem \ref{XPEFFTHM} is indeed an effective version of Dani's equidistribution result.
For given $g=\smatr abcd\in G'$, the rate of decay of $y_g(T)$ as $T\to\infty$
is directly related to the Diophantine properties of the number $\frac ac$ (assuming $c\neq0$):
If $\frac ac$ is of Diophantine type $K\geq2$ 
(viz., $\inf_{q\in\Z^+} q^{K-1}\langle q\frac ac\rangle>0$),
then there is $C=C(g,K)>0$ such that $y_g(T)\leq CT^{-2/K}$ for all $T\geq1$.
In particular, for (Haar-)almost all $g\in G'$, the right hand side of \eqref{XPEFFTHMRES} decays
more rapidly than $T^{\ve-\frac12}$ as $T\to\infty$ ($\forall \ve>0$).  \label{YGTGENERICDECAY}
The rate of decay of the right hand side in \eqref{XPEFFTHMRES} is in fact essentially optimal,
for any given $g\in G'$; cf.\ \cite[Thm.\ 2 and \S\S 4-5]{iha}.
We also remark that \cite[Thm.\ 1]{iha} is more general in that it holds for an arbitrary cofinite subgroup of $\PSL(2,\R)$
in place of $\Gamma'$ (the bound then depends on the small eigenvalues of the Laplace-Beltrami operator
on the corresponding hyperbolic surface); also the bound holds with a weaker function space norm
than the $\|\cdot\|_{\C_{\b}^4}$ used above.

\vspace{5pt}

We are now ready to state our effective version of Theorem \ref{NONCLOSEDINEFFTHM}.
For $T>0$, let $\fR_T$ be the closed rectangle
$\fR_T:=[-T^{-1},T^{-1}]\times[-1,1]\subset\R^2$.
We also use the shorthand notation $\Z^+_{\leq a}=(0,a]\cap\Z$.
Set, for $g\in G$ and $T>0$,
\begin{align}\label{BGTDEF}
\fb_g(T)=\inf\Bigl\{\delta>0\col\Bigl[\forall q\in\Z_{\leq\delta^{-1/2}}^+:\:
(q^{-1}\Z^2)g\cap \frac{1}{\delta q^2}\fR_T=\emptyset\Bigr]\Bigr\}.
\end{align}
(This can be viewed as a generalization of the notation $\fb_{\vecxi,L}(y)$ introduced previously;
cf.\ equation 
\eqref{EXILALTFORMULA2} on p.\ \pageref{EXILALTFORMULA2}.)
\begin{thm}\label{NONCLOSEDTHM2}
Given any $\ve>0$, there exists a constant $C>0$ such that, for 
any $g\in G$, $T\geq2$ and $f\in\C_{\b}^8(\GaG)$, we have
\begin{align}\label{NONCLOSEDTHM2RES}
&\biggl|\frac1T\int_0^T f(\Gamma gU^t)\,dt-\int_{\GaG} f\,d\mu\biggr|
\leq C\|f\|_{\C_{\b}^8}\bigl(y_g(T)^{\frac14}+\fb_g(T)\bigr)^{\frac12-\ve}.
\end{align}
\end{thm}
Note that for any given $g\in G$ we have $\lim_{T\to\infty}(y_g(T)^{\frac14}+b_g(T))=0$ if (and only if)
$D(\Gamma gU^\R)$ is not a closed horocycle in $X'$ and
$\Q^2g\cap(0,\R)=\emptyset$,
viz.\ $(0,\beta)g^{-1}\notin\Q^2$ for all $\beta\in\R$.
Hence Theorem \ref{NONCLOSEDTHM2} is indeed an effective version of Theorem \ref{NONCLOSEDINEFFTHM}.
We will also see that for $\mu$-almost all $g\in G$, we have $\lim_{T\to\infty}b_g(T)T^{\delta}=0$ for all $\delta<\frac12$
(cf.\ Proposition \ref{BGTGENDECAYPROP}); 
hence, recalling the earlier discussion about $y_g(T)$,
we see that for $\mu$-almost all $g\in G$, the right hand side in \eqref{NONCLOSEDTHM2RES} decays more rapidly than
$T^{\ve-\frac18}$ as $T\to\infty$ ($\forall\ve>0$).
As we discuss in Remark \ref{NONCLOSEDOPTEXPREM} below, %
optimally one might hope to improve Theorem \ref{NONCLOSEDTHM2} so as to yield 
a rate of decay $T^{\ve-\frac12}$ for any $g$ satisfying appropriate Diophantine conditions.

\subsection{Applications and extensions} %
\label{APPLSEC}

As we have mentioned, cases of Ratner equidistribution in settings closely related to that of the present paper
have played a crucial role in the solution of several problems in number theory and in mathematical physics.
We discuss some of these here.

In \cite{MarklofpaircorrI}, \cite{MarklofpaircorrII}, Marklof proved that
the limit local pair correlation density of the sequence $\|\vecm-\vecalf\|^k$, $\vecm\in\Z^k$ ($k\geq2$)
is that of a Poisson process, under Diophantine conditions on the fixed vector $\vecalf\in\R^k$.
In particular for $k=2$ this gives a quantitative Oppenheim type statement for the inhomogeneous quadratic form
$(x_1-\alpha)^2+(x_2-\beta)^2-(x_3-\alpha)^2-(x_4-\beta)^2$.
The proof makes use of an analogue of Theorem \ref{NONEFFECTIVETHM} for $G=\SL(2,\R)\ltimes(\R^2)^{\oplus k}$
and $\Gamma$ a congruence subgroup of $\SL(2,\Z)\ltimes(\Z^2)^{\oplus k}$.
In joint work with Pankaj Vishe, \cite{effopp}, we generalize the methods of the present paper to that case,
and apply this to obtain an effective rate of convergence for the pair correlation density of 
$\|\vecm-\vecalf\|^k$.

In particular it is noted in \cite{effopp} that
the methods of the present paper can without serious difficulty be extended
to the case of $\Gamma$ being an arbitrary congruence subgroup of $\SL(2,\Z)\ltimes\Z^2$.
However, already in a case such as 
$\Gamma=\widetilde\Gamma'\ltimes\Z^2$, with $\widetilde\Gamma'$ a noncongruence subgroup of finite index of $\SL(2,\Z)$,
new ideas would be needed to extend the results of the present paper.
(We remark that \textit{every} lattice $\Gamma$ in $G=\SL(2,\R)\ltimes\R^2$
can be conjugated within %
$\GL(2,\R)\ltimes\R^2$ into a position where
$D(\Gamma)$ is a finite index subgroup of $\SL(2,\Z)$ and $(\{1_2\}\ltimes\R^2)\cap\Gamma=\{1_2\}\ltimes\Z^2$;
cf.\ \cite[Cor.\ 8.28]{Raghunathan}.
However it is not always possible to conjugate into a situation where
$\Gamma$ contains $\widetilde\Gamma'\ltimes L$ for some subgroup $\widetilde\Gamma'$ of finite index in $\SL(2,\Z)$
and a lattice $L\subset\Z^2$.
Indeed, consider for example the lattice $\Gamma$ generated by $(\smatr1201,\vecv)$, $(\smatr1021,\vecv')$,
$(1_2,(1,0))$, $(1_2,(0,1))$, for some fixed $\vecv,\vecv'\in\R^2$ such that the first coordinate of $\vecv$ is irrational.
Recall in this connection that $\smatr1201$ and $\smatr1021$ are free generators of the principal congruence subgroup
$\Gamma(2)$ in $\SL(2,\Z)$.)

Quantitative Oppenheim type results for more general inhomogeneous quadratic forms have recently been obtained by
Margulis and Mohammadi \cite{MM}, using a method different from Marklof's.
For the special case of forms of signature (2,1) %
whose homogeneous part is a split rational form
(see \cite[Thm.\ 1.10]{MM}), the proof depends on equidistribution of unipotent orbits in
homogeneous spaces of the group $\SL(2,\R)\ltimes\text{Sym}_2(\R)$.
It seems that it should be possible to extend the methods of the present paper to these
homogeneous spaces, %
and also to more general groups of the form $\SL(2,\R)\ltimes V$ where $V$ is the vector space of a 
finite dimensional linear representation of $\SL(2,\R)$.

Elkies and McMullen \cite{Elkies04}
have shown that the gaps between the fractional parts of $\sqrt n$ for $n=1,\ldots,N$,
have a limit distribution as $N$ tends to infinity, and they compute this limit distribution explicitly.
In a recent paper, El-Baz, Marklof and Vinogradov \cite{elbazmv2} also prove convergence of the
local pair-correlation and more general mixed moments.
The proofs make crucial use of an analogue of Theorem \ref{NONEFFECTIVETHM} for the flow $U_1^\R$,
with $U_1^x=\bigl(\smatr1x01,-(x/2,x^2/4)\bigr)$. Since $U_1^\R$ is not conjugate to $U^\R$, 
Theorem \ref{MAINABTHM} does not apply to this setting.
In fact \textit{any} $\Ad$-unipotent $1$-parameter subgroup in $G$ with nontrivial image in $G'$
is conjugate to either $U^\R$ or $U_1^\R$.
Recently, Browning and Vinogradov \cite{BV} have extended the methods of the present paper
so as to yield an effective equidistribution result for certain orbits of the flow $U_1^\R$,
and applied this to establish an effective rate for the convergence of the gap distribution of $\sqrt n\,\mod1$.
(Note also that Sinai \cite{Sinai} has proposed an alternative approach to the statistics of $\sqrt n\,\mod 1$.)

Another application concerns the local statistics of directions to lattice points:
Consider a fixed lattice translate $L$ in $\R^2$ and record the directions of all lattice
vectors of length at most $T$. In joint work with Marklof we proved in \cite[Thm.\ 1.3; see also Thm.\ 2.1]{partI} 
that the distribution of gaps between the lattice directions has a limit as $T$ tends to infinity;
see also El-Baz, Marklof and Vinogradov \cite{elbazmv1} regarding convergence of the 
local pair-correlation and more general mixed moments.
Assuming that $L$ is an 'irrational' translate, the limit distribution is universal and in fact coincides with
the limiting gap distribution for $\sqrt n\,\mod1$ found by Elkies and McMullen.
The proofs of these facts make use of equidistribution of expanding translates of $\SO(2)$-orbits in the same space $X=\GaG$
as we consider here.  %
By a standard approximation argument this is reduced to the equidistribution of pieces of $U^\R$-orbits 
(cf. the proof of Cor.~5.4 in \cite{partI}),
and thus using our Theorem \ref{NONCLOSEDTHM2} it should be possible to prove an effective rate of convergence in
\cite[Thm.\ 1.3]{partI}, for 'irrational' lattice translates.
However several technicalities remain to be worked out to carry this through.  %

As a final example, in \cite[Remark 4 ($n=2$)]{SV} it is noted that
the number of values modulo one of a random linear form $\omega n$ for $n=1,\ldots,N$
which fall inside a given small interval of length $c/N$ centered at a fixed irrational point $\xi\in\R/\Z$,
has a limit distribution as $N\to\infty$, which is independent of $\xi$.
The proof is an application of Theorem \ref{NONEFFECTIVETHM} in the special case $\vecxi=(0,\xi)$,
and thus using our Theorem \ref{MAINABTHM} it would be possible to prove an effective rate for the convergence
to the limit distribution, depending on the Diophantine properties of $\xi$.

We hope to return to several of the above-mentioned questions %
in later work. 

\subsection{Outline of the paper}
Sections \ref{NOTATIONSEC}--\ref{LEADINGTERMSEC} 
lay down the setup of our approach:
In Section \ref{NOTATIONSEC} we set some basic notation;
in Section \ref{SMOOTHEDERGAV} we smooth the $(\alpha,\beta)$-integral %
appearing in Theorem~\ref{MAINABTHM};
in Section \ref{FOURIERDECSEC} we discuss the Fourier decomposition of the given
test function on $X=\GaG$ with respect to the torus fiber variable;
and in Section \ref{LEADINGTERMSEC} we handle the contribution from the zeroth Fourier term;
this reduces to a known result on the effective equidistribution of horocycle orbits in $X'$.

The basic idea of our approach appears in Sections \ref{INITIALDISCSEC}--\ref{CANCELLATIONSEC};
we first rewrite the remaining terms of the Fourier decomposition in an appropriate format,
and then prove a lemma (Lemma \ref{EXPSUM1LEM})
which can be used to establish cancellation in the sum;
this lemma is nothing but a standard application of the classical Weil's bound on Kloosterman sums.

The proof of Theorem \ref{MAINABTHM} is given in Sections \ref{WEAKERVERSIONPFSEC}--\ref{NEWBOUNDSEC}:
In Section \ref{WEAKERVERSIONPFSEC} we carry out those steps %
which utilize only the irrationality properties of $\xi_1$ and not those of $\xi_2$;
the outcome is a weaker version of the theorem, Proposition \ref{MAINTHMWEAKERPROP},
which is strong enough to imply the equidistribution in Theorem \ref{NONEFFECTIVETHM} whenever $\xi_1$ is irrational, 
with the error bound decaying as a power of $y$ whenever $\xi_1$ is of Diophantine type;
however for $\xi_1$ rational it does not imply any equidistribution whatsoever.
To complete the proof of Theorem \ref{MAINABTHM}, in Section~\ref{NEWBOUNDSEC} (the longest section of the paper),
we consider more carefully those terms in the Fourier decomposition which give the largest contribution
in the treatment of Section \ref{WEAKERVERSIONPFSEC}; these correspond to good rational approximations of $\xi_1$;
we collect these terms in a way
which allows us to utilize also the irrationality properties of $\xi_2$ to establish cancellation.
The error bound which we finally arrive at in Theorem \ref{MAINABTHM}
incorporates the Diophantine properties of both $\xi_1$ and $\xi_2$, the bound being 
far from zero only if
$\xi_1$ and $\xi_2$ are well approximable by rational numbers with a \textit{common} small denominator $q$;
cf.\ the definition of the error majorant $\fb_{\vecxi,L}(y)$ in \eqref{MYXILDEF}.

The precise format of this bound plays a %
crucial role when we %
apply Theorem \ref{MAINABTHM} to
deduce the effective equidistribution of general $U^\R$-orbits, Theorem~\ref{NONCLOSEDTHM2}.
To illustrate this point, note that 
to establish a result which could be called ``an effective version of Theorem \ref{NONEFFECTIVETHM}'',
it would suffice to complement Proposition \ref{MAINTHMWEAKERPROP} with an effective equidistribution result
for $\xi_1$ \textit{rational} and $\xi_2$ irrational.
This would be quite a bit easier than what we do in Section \ref{NEWBOUNDSEC};
however it would not be sufficient for our goal of deriving a 
satisfactory effective equidistribution for general $U^\R$-orbits,
basically since our proof of %
Theorem~\ref{NONCLOSEDTHM2}
for a given $g=(1_2,\vecxi)M$ generally involves
applying Theorem \ref{MAINABTHM} with $\vecxi\gamma$ in place of $\vecxi$,
where $\gamma$ varies through more and more elements of %
$\Gamma'$ as $T\to\infty$.

In Section \ref{MAJORANTPROPSEC} we establish some important basic properties of the error majorant 
$\fb_{\vecxi,L}(y)$.
Finally in Section \ref{NONCLOSEDSECTION} we prove 
Theorem~\ref{NONCLOSEDTHM2}, by approximating 
the given $U^\R$-orbit by one or several lifts of pieces of closed horocycles in $X'$
and applying Theorem~\ref{MAINABTHM} to each of these.

\subsection{Acknowledgments}

I am grateful to 
Livio Flaminio, 
Giovanni Forni,
Han Li, 
Jens Marklof, 
Amir Mohammadi,
Hee Oh,
Wolfgang Staubach, Akshay Venkatesh and Pankaj Vishe for helpful and inspiring discussions.
I would also like to thank the referees for their valuable comments; in particular Remark \ref{METHODREMARK} below
is based on a suggestion by one of the referees.

\section{Some notation}\label{NOTATIONSEC}

We shall use the standard notation $A=O(B)$ or $A\ll B$ meaning $|A|\leq CB$ for some constant $C>0$.
We shall also write $A\asymp B$ as a substitute for $A\ll B\ll A$.
To indicate that the implicit constant $C$ may depend on some quantities or functions $f,g,h$ we will use the notation
$A\ll_{f,g,h} B$ or $A=O_{f,g,h}(B)$.
The constant $C$ will not depend on any other variable, except in a statement that contains an implication of the
kind ``if $A_1=O(B_1)$ then $A_2=O(B_2)$''; in that case the constant implicit in $O(B_2)$ may also depend on the one in
$O(B_1)$. (We will use the last convention only in Remarks \ref{NEEDALFBETANEARZEROREM} and \ref{MAYSPLITREMARK}.)

Recall from Section \ref{INTROSEC} that $G=\SL(2,\R)\ltimes\R^2$, $\Gamma=\SL(2,\Z)\ltimes\Z^2$,
$G'=\SL(2,\R)$ and $\Gamma'=\SL(2,\Z)$. We will also write
$\Gamma'_\infty:=\{\smatr 1x01\col x\in\Z\}$.

Let $\ig$ be the Lie algebra of $G$. 
We may identify $\ig$ in a natural way with the space $\lsl(2,\R)\oplus\R^2$,
with Lie bracket %
$[(X,\vecv),(Y,\vecw)]=(XY-YX,\vecv Y-\vecw X)$
(cf., e.g., \cite[Prop.\ 1.124]{Knapp}).
Using this notation, we fix the following basis of $\ig$:
\begin{align}\label{FIXEDBASIS}
X_1=\bigl(\smatr 0100,\bn\bigr);\quad
X_2=\bigl(\smatr 0010,\bn\bigr);\quad
X_3=\bigl(\smatr 100{-1},\bn\bigr);\quad
X_4=\bigl(0_2,(1,0)\bigr);\quad
X_5=\bigl(0_2,(0,1)\bigr).
\end{align}
To each $Y\in\ig$ corresponds a left invariant differential operator
on functions on $G$, and thus also a differential operator on $\GaG$,
which we will also denote by $Y$.
We let $\C_{\b}^k(\GaG)$ be the space of $k$ times continuously differentiable functions on $\GaG$
such that $\|Df\|_{\L^\infty}<\infty$ for every left invariant differential operator $D$ on $G$ of order $\leq k$.
For $f\in\C_{\b}^k(\GaG)$ we set
\begin{align}\label{CKNORMDEF}
\|f\|_{\C_{\b}^k}:=\sum_{\ord(D)\leq k}\|Df\|_{\L^\infty},
\end{align}
the sum being over all monomials in $X_1,\ldots,X_5$ of degree $\leq k$.
Note in particular that $\C_{\b}^0(\GaG)$ is the space of bounded continuous 
functions on $\GaG$, and %
$\|\cdot\|_{\C_{\b}^0}$ is the supremum norm.

We will also have occasion to use Sobolev $\L^1$-norms on functions on $\R$:
For $1\leq p<\infty$, $k$ a positive integer and $\nu\in\C^k(\R)$ we set
\begin{align*}
\|\nu\|_{\W^{k,p}}=\sum_{j=0}^k\|\nu\|_{\L^p}=\sum_{j=0}^k\Bigl(\int_\R|\nu^{(j)}(x)|^p\,dx\Bigr)^{1/p}.
\end{align*}
We will only use these for $p=1$.

We will use the standard notation $e(x)=e^{2\pi i x}$.
We write $\gcd(c,d)$, or just $(c,d)$, for the greatest common divisor of two integers $c,d$.
For $n$ a positive integer, we write $\sigma(n)$ for the number of (positive) divisors of $n$, and $\sigma_1(n)$ for
their sum: $\sigma(n)=\sum_{d\mid n}1$ and $\sigma_1(n)=\sum_{d\mid n}d$.

\section{Smoothed ergodic averages}
\label{SMOOTHEDERGAV}

As a first step in our proof of Theorem \ref{MAINABTHM} we replace the sharp cutoff ``$\int_\alpha^\beta$'' 
by a compactly supported cutoff function $\nu(x)$ satisfying a mild regularity assumption.
Basically we need control on the $\L^1$-norm of ``$1+\ve$'' derivatives of $\nu$;
in order to avoid a technical overhead we formulate the bound using a crude
interpolation between the Sobolev norms $\|\nu\|_{\W^{1,1}}$ and $\|\nu\|_{\W^{2,1}}$
(cf., e.g., \cite[Sec.\ 2]{SV}).
We will prove the following theorem.

\begin{thm}\label{MAINTHM}
Let $0<\eta<1$ and $\ve>0$ be fixed. Then for any $f\in \C_{\b}^8(\GaG)$,
any $\nu\in\C^2(\R)$ with compact support,
and any $\vecxi\in\R^2$, $0<y<1$,
\begin{align}\notag
\int_\R f\Bigl(\Gamma\,(1_2,\vecxi)\,U^xa(y)\Bigr)\,\nu(x)\,dx
=\int_{\GaG} f\,d\mu\int_\R\nu\,dx
\hspace{180pt}
\\\label{MAINTHMRES}
+O_{\eta,\ve}\biggl\{\|f\|_{\C_{\b}^{8}}\|\nu\|_{\W^{1,1}}^{1-\eta}\|\nu\|_{\W^{2,1}}^\eta\,y^{\frac14}\log(1+y^{-1})
+\|f\|_{\C_{\b}^4}L\|\nu\|_{\L^\infty}\bigl(\fb_{\vecxi,L}(y)+y^{\frac14}\bigr)^{1-\ve}\biggr\},
\end{align}
where $L$ is the smallest real number $\geq1$ such that $\supp(\nu)\subset[-L,L]$.
\end{thm}

\begin{proof}[Proof that Theorem \ref{MAINTHM} implies Theorem \ref{MAINABTHM}]
This is a standard approximation argument.
Fix $g\in\C_c^\infty(\R)$ satisfying $g\geq0$, $\int_\R g=1$ and $\supp(g)\subset[-1,1]$.
Set $g_\delta(x)=\delta^{-1}g(\delta^{-1}x)$ for $0<\delta\leq1$;
then $\supp(g_\delta)\subset[-\delta,\delta]$ and $\int_\R g_\delta=1$.
Let $\alpha<\beta$ be given, and set $L'=\beta-\alpha$.
We apply Theorem \ref{MAINTHM} with $\nu=\chi_{[\alpha,\beta]}*g_\delta$.
Then $\|\nu\|_{\W^{1,1}}\ll L'+1$ and $\|\nu\|_{\W^{2,1}}\ll L'+\delta^{-1}$; thus
$\|\nu\|_{\W^{1,1}}^{1-\eta}\|\nu\|_{\W^{2,1}}^\eta\ll (L'+1)\delta^{-\eta}$,
and so the error term in Theorem \ref{MAINTHM} is
\begin{align*}
O_{\eta,\ve}\biggl\{\|f\|_{\C_{\b}^{8}}(L'+1)\delta^{-\eta}\,y^{\frac14(1-\ve)}
+\|f\|_{\C_{\b}^4}L\bigl(\fb_{\vecxi,L}(y)+y^{\frac14}\bigr)^{1-\ve}\biggr\},
\end{align*}
with $L=\max(1,|\alpha|+\delta,|\beta|+\delta)$.
Furthermore, using $0\leq\nu\leq1$ and $\nu(x)=\chi_{[\alpha,\beta]}(x)$ whenever $|x-\alpha|\geq\delta$ and
$|x-\beta|\geq\delta$, we see that the difference between the left hand side of \eqref{MAINTHMRES} and
$\int_\alpha^\beta f(\Gamma\,(1_2,\vecxi)\,U^xa(y))\,dx$
is $\ll\|f\|_{\C_{\b}^0}\,\delta$.
Hence, choosing $\delta=y^{\frac14}$, we obtain
\begin{align*}
\int_\alpha^\beta f\Bigl(\Gamma\,(1_2,\vecxi)\,U^x a(y)\Bigr)\,dx
=(\beta-\alpha)\int_{\GaG} f\,d\mu
\hspace{180pt}
\\ %
+O_{\eta,\ve}\biggl\{\|f\|_{\C_{\b}^{8}}(L'+1)y^{\frac14(1-\ve-\eta)}
+\|f\|_{\C_{\b}^4}L\bigl(\fb_{\vecxi,L}(y)+y^{\frac14}\bigr)^{1-\ve}\biggr\}.
\end{align*}
This implies Theorem \ref{MAINABTHM} with $\ve+\eta$ in place of $\ve$
(cf.\ also Lemma \ref{EORDERCHANGELEM} below).
\end{proof}

\begin{remark}
The proof shows that the bound in Theorem \ref{MAINABTHM} may be improved to
\begin{align*}
C\biggl(\|f\|_{\C_{\b}^{8}}\Bigl(1+\frac1{\beta-\alpha}\Bigr)y^{\frac14(1-\ve)}
+\|f\|_{\C_{\b}^4}\frac L{\beta-\alpha}\bigl(\fb_{\vecxi,L}(y)+y^{\frac14}\bigr)^{1-\ve}\biggr).
\end{align*}
\end{remark}

\section{Fourier decomposition in the torus variable}
\label{FOURIERDECSEC}

We now start with the proof of Theorem \ref{MAINTHM}.
In this section we consider the Fourier decomposition of the given
test function with respect to the torus variable,
and prove bounds on the Fourier coefficients appearing in this decomposition.

Assume that $f\in \C_{\b}^2(\GaG)$.
We view $f$ as a function on $G$ which is $\Gamma$-left invariant.
In particular we have 
$f((1_2,\vecxi)M)=f((1_2,\vecxi+\vecn)M)$ for all $\vecn\in\Z^2$,
and hence for any fixed $M\in G'$, the function
$\vecxi\mapsto f((1_2,\vecxi)M)$ is a $\C^2$-function on the torus $\T^2=\R^2/\Z^2$.
Decomposing this function as a Fourier series we have
\begin{align}\label{FOURIERSERIES}
f((1_2,\vecxi)M)=\sum_{\vecm\in\Z^2}\widehat f(M,\vecm)e(\vecm\cdot\vecxi),
\end{align}
where the Fourier coefficients $\widehat f(M,\vecm)$ are given by
\begin{align}\label{WHFDEF}
\widehat f(M,\vecm)=\int_{\T^2} f((1_2,\vecxi)M)e(-\vecm\cdot\vecxi)
\,d\vecxi.
\end{align}
Here $d\vecxi$ denotes Lebesgue measure on $\R^2$.
Note that the sum in \eqref{FOURIERSERIES} is absolutely convergent,
uniformly\footnote{This is for any fixed exhaustion of $\Z^2$ by an 
increasing sequence of finite subsets.}
over $(M,\vecxi)$ in any compact subset of $G$,
since $f\in\C_{\b}^2(\GaG)$ implies that
the function $\vecxi\mapsto f(M,\vecxi)$ is in $\C_{\b}^2(\T^2)$,
with $\|f(M,\cdot)\|_{\C_{\b}^2(\T^2)}$ depending continuously on $M\in G'$.

Now the fact that $f$ is also $\Gamma'$-left invariant 
leads to an invariance relation for $\widehat f(M,\vecm)$,
which allows us to group together terms in \eqref{FOURIERSERIES}
in a convenient way.
Let us write $\wh\Z^2$ for the set of primitive lattice points in $\Z^2$,
i.e.\ the set of integer vectors $(c,d)$ %
with $\gcd(c,d)=1$.
Recall that $\Gamma'_\infty:=\{\smatr 1x01\col x\in\Z\}$.
\begin{lem}\label{BASICFOURIERLEM}
In the above situation we have
\begin{align}\label{BASICFOURIERLEMRES1}
\wh f(TM,\vecm)=\wh f(M,\vecm\trans T^{-1}),\qquad\forall T\in\Gamma',
\: M\in G',\:\vecm\in\Z^2.
\end{align}
In particular, for each $n\in\Z_{\geq0}$, the function
\begin{align}\label{TFNDEF}
\tf_n(M):=\wh f(M,(n,0))\qquad
(M\in G')%
\end{align}
is left $\Gamma'_\infty$-invariant, and $\tf_0(M)$ is even left 
$\Gamma'$-invariant.
We have
\begin{align}\label{BASICFOURIERLEMRES2}
f((1_2,\vecxi)M)=\tf_0(M)+\sum_{n=1}^\infty
\sum_{(c,d)\in\wh\Z^2}
\tf_n\left(\matr **cd M\right)\cdot e(n(d\xi_1-c\xi_2)), \qquad\forall (M,\vecxi)\in G,
\end{align}
where $\smatr **cd$ denotes any matrix in $\Gamma'$ having lower entries 
$c$ and $d$.
The sum in \eqref{BASICFOURIERLEMRES2} is absolutely convergent, 
uniformly over $(M,\vecxi)$ in any compact subset of $G$.
\end{lem}
(To see that the sum in \eqref{BASICFOURIERLEMRES2} is well-defined,
note that for any $(c,d)\in\wh\Z^2$, 
the set of matrices $\smatr **cd\in\Gamma'=\SL(2,\Z)$ is a coset of the form
$\Gamma'_\infty\smatr abcd$, and since $\tf_n$ is left
$\Gamma'_\infty$-invariant, $\tf_n(TM)$ takes the same value for
every matrix $T$ in this coset.)
\begin{proof}
For any $T\in\Gamma'$ we have,
\begin{align*}
\widehat f(TM,\vecm)
&=\int_{\Z^2\backslash\R^2} f((1_2,\vecxi)TM)e(-\vecm\cdot\vecxi)\,d\vecxi
=\int_{\Z^2\backslash\R^2} f(T(1_2,\vecxi T)M)e(-\vecm\cdot\vecxi)\,d\vecxi
\\
&=\int_{\Z^2\backslash\R^2} f(T(1_2,\vecxi)M)e(-\vecm\cdot\vecxi T^{-1})\,d\vecxi
=\int_{\Z^2\backslash\R^2} f((1_2,\vecxi)M)e(-\vecm\cdot\vecxi T^{-1})\,d\vecxi,
\end{align*}
where in the third identity we used the fact that
$\vecxi\mapsto \vecxi T$ is a diffeomorphism of
$\Z^2\backslash\R^2$ preserving the area measure $d\vecxi$,
and in the last identity we used the fact that $f$ is left $\Gamma$-invariant.
Using now $\vecm\cdot\vecxi T^{-1}=\vecm \trans T^{-1}\cdot\vecxi$ we obtain \eqref{BASICFOURIERLEMRES1}.

Next, note that every non-zero vector $\vecm\in\Z^2$ can be uniquely expressed
as $n(d,-c)$ with $n\in\Z^+$ and $(c,d)\in\wh\Z^2$.
Hence the Fourier series \eqref{FOURIERSERIES} can be expressed as
\begin{align*}
f((1_2,\vecxi)M)=\wh f(M,\bn)
+\sum_{n=1}^\infty\sum_{(c,d)\in\wh\Z^2}
\wh f\bigl(M,n(d,-c)\bigr)e(n(d,-c)\cdot\vecxi).
\end{align*}
However if $T$ is any matrix of the form $\smatr **cd\in\Gamma'$ then
$n(d,-c)=(n,0)\trans T^{-1}$, and now 
by using \eqref{BASICFOURIERLEMRES1} we obtain \eqref{BASICFOURIERLEMRES2}.
The uniform absolute convergence on compacta holds since it holds in 
\eqref{FOURIERSERIES}.
\end{proof}

Note that the functions $\tf_n$ %
are well-defined for any 
$f\in\C(\GaG)$, through \eqref{TFNDEF}, %
\eqref{WHFDEF}.

\begin{lem}\label{DECAYTFNFROMCMNORMLEM}
For any $m\in\Z_{\geq0}$, $n\in\Z^+$ and $f\in\C_{\b}^m(\GaG)$, we have
\begin{align}\label{DECAYTFNFROMCMNORMLEMRES}
\left|\tf_n\left(\smatr abcd\right)\right|
\ll_m\frac{\|f\|_{\C_{\b}^m}}{n^m(c^2+d^2)^{\frac m2}},
\qquad\forall\smatr abcd\in G'.
\end{align}
\end{lem}
\begin{proof}
The left invariant differential operator corresponding to $X\in\ig$ is 
given by $Xf(g)=\lim_{t\to0} (f(g\exp(tX))-f(g))/t$.
In particular, since $\exp(tX_4)=(1_2,(t,0))$ and
$\exp(tX_5)=(1_2,(0,t))$ (cf.\ \eqref{FIXEDBASIS}), we find that
if we parametrize $G$ by
$(1_2,(x_1,x_2))\smatr abcd$
then 
\begin{align*}
X_4=d\frac{\partial}{\partial x_1}-b\frac{\partial}{\partial x_2}
\qquad\text{and}\qquad
X_5=-c\frac{\partial}{\partial x_1}+a\frac{\partial}{\partial x_2}.
\end{align*}
Now 
\begin{align*}
\tf_n\left(\smatr abcd\right)=\int_{\T^2} %
f\left((1_2,(x_1,x_2))\smatr abcd\right)e(-nx_1)\,dx_2\,dx_1,
\end{align*}
and hence by repeated integration by parts we have
\begin{align*}
(2\pi i nd)^m\cdot\tf_n\left(\smatr abcd\right)
=\int_{\T^2} %
[X_4^mf]\left((1_2,(x_1,x_2))\smatr abcd\right)e(-nx_1)\,dx_2\,dx_1,
\end{align*}
and
\begin{align*}
(-2\pi i nc)^m\cdot\tf_n\left(\smatr abcd\right)
=\int_{\T^2} %
[X_5^mf]\left((1_2,(x_1,x_2))\smatr abcd\right)e(-nx_1)\,dx_2\,dx_1.
\end{align*}
Hence
\begin{align*}
\max(|c|^m,|d|^m)\cdot\bigl|\tf_n\left(\smatr abcd\right)\bigr|
\leq(2\pi n)^{-m}\|f\|_{\C_{\b}^m},
\end{align*}
and this implies that \eqref{DECAYTFNFROMCMNORMLEMRES} holds.
\end{proof}

Using Lemma \ref{DECAYTFNFROMCMNORMLEM} we immediately also obtain bounds
on \textit{derivatives} of $\tf_n$.
To make this explicit, 
let us embed $\lsl(2,\R)$ as a subalgebra of $\ig$ through
$X\mapsto(X,\bn)$ (using our notation $\ig\cong\lsl(2,\R)\oplus\R^2$).
Then each $X\in\lsl(2,\R)$, and more generally any element
$D$ in the universal enveloping algebra $\scrU(\lsl(2,\R))$,
gives rise to a left invariant differential
operator both on $G'$ and on $G$.

\begin{lem}\label{DERDECAYTFNFROMCMNORMLEM}
For any $m\in\Z_{\geq0}$, $n\in\Z^+$, 
any $D\in\scrU(\lsl(2,\R))$ of order $\leq k$,
and any $f\in\C_{\b}^{m+k}(\GaG)$, we have
\begin{align}\label{DERDECAYTFNFROMCMNORMLEMRES}
\left|D\tf_n\left(\smatr abcd\right)\right|
\ll_m\frac{\|Df\|_{\C_{\b}^m}}{n^m(c^2+d^2)^{\frac m2}},
\qquad\forall\smatr abcd\in G'.
\end{align}
\end{lem}
\begin{proof}
Set $g:=Df\in\C_{\b}^m(\GaG)$;
then by differentiation under the integration sign in 
\eqref{WHFDEF} we have $D\tf_n=\widetilde g_n$.
Hence the lemma follows from Lemma \ref{DECAYTFNFROMCMNORMLEM} applied to $g$.
\end{proof}

We will often consider the function $\tf_n$ in Iwasawa coordinates,
that is we write (by a slight abuse of notation)
\begin{align}\label{TFNIWASAWA}
\tf_n(u,v,\theta):=\tf_n\left(\matr 1u01\matr{\sqrt v}00{1/\sqrt v}
\matr{\cos\theta}{-\sin\theta}{\sin\theta}{\cos\theta}\right),
\end{align}
for $u\in\R$, $v>0$, $\theta\in\R/2\pi\Z$.

\label{SL2RGLIEALGDISCSEC}

\begin{lem}\label{DERDECAYTFNFROMCMNORMLEM2}
For any $m\in\Z_{\geq0}$, $n\in\Z^+$, $k_1,k_2,k_3\in\Z_{\geq0}$ and $f\in\C_{\b}^{m+k}(\GaG)$, 
where $k=k_1+k_2+k_3$, we have
\begin{align}\label{DERDECAYTFNFROMCMNORMLEM2RES1}
&\biggl|\Bigl(\frac{\partial}{\partial u}\Bigr)^{k_1}\Bigl(\frac{\partial}{\partial v}\Bigr)^{k_2}
\Bigl(\frac{\partial}{\partial \theta}\Bigr)^{k_3}\tf_n(u,v,\theta)\biggr|
\ll_{m,k}\|f\|_{\C_{\b}^{m+k}}n^{-m}v^{\frac m2-k_1-k_2}.
\end{align}
\end{lem}
\begin{proof}
Let $\kk_\theta:=\smatr{\cos\theta}{-\sin\theta}{\sin\theta}{\cos\theta}
\in\SL(2,\R)$,
and write $X_1,X_2,X_3$ for the $\lsl(2,\R)$-elements $\smatr0100$,
$\smatr0010$, $\smatr100{-1}$, respectively. 
This is consistent with \eqref{FIXEDBASIS} and our fixed embedding of
$\lsl(2,\R)$ in $\ig$.
Also let $\Ad:\SL(2,\R)\to\text{Aut}(\lsl(2,\R))$ be the adjoint representation.
Then we compute, in the parametrization \eqref{TFNIWASAWA},
\begin{align*}
&\Ad(\kk_{-\theta})X_1=v\frac{\partial}{\partial u};
\qquad
\Ad(\kk_{-\theta})X_3=2v\frac{\partial}{\partial v};
\qquad
X_2-X_1=\frac{\partial}{\partial \theta}.
\end{align*}
But $\Ad(\kk_{-\theta})X_1$ and $\Ad(\kk_{-\theta})X_3$ belong
to a fixed compact subset of $\lsl(2,\R)$;
in fact one checks by a quick computation that these elements
always lie in $\{c_1X_1+c_2X_2+c_3X_3\col c_1,c_2,c_3\in[-1,1]\}$.
Hence we have, %
at every point $(u,v,\theta)\in\R\times\R_{>0}\times(\R/2\pi\Z)$,
\begin{align*}
\biggl|\Bigl(v\frac{\partial}{\partial u}\Bigr)^{k_1}\Bigl(2v\frac{\partial}{\partial v}\Bigr)^{k_2}
\Bigl(\frac{\partial}{\partial \theta}\Bigr)^{k_3}\tf_n(u,v,\theta)\biggr|
\leq\sum_{\ord(D)=k}\bigl|D\tf_n(u,v,\theta)\bigr|,
\end{align*}
where the sum is taken over all the $3^k$ monomials in $X_1,X_2,X_3$ of
degree $k$.
Now the desired bound follows immediately from 
Lemma \ref{DERDECAYTFNFROMCMNORMLEM} and the preceding discussion,
if we also note that $c^2+d^2=v^{-1}$ holds
whenever $\smatr abcd=\smatr 1u01 a(v)\kk_\theta$
(the matrix in \eqref{TFNIWASAWA}),
and that $v^{k_2}\partial_v^{k_2}$ can be expressed as a linear combination
of $(v\partial_v)^j$ for $j=1,\ldots,{k_2}$.
\end{proof}

\section{The leading term; horocycle equidistribution in $X'=\SL(2,\Z)\backslash\SL(2,\R)$}
\label{LEADINGTERMSEC}

Our task is to study the integral
\begin{align*}
\int_\R f\Bigl(\Gamma\,(1_2,\vecxi)\,U^xa(y)\Bigr)\nu(x)\,dx
=\int_\R f\left(\Gamma\,(1_2,\vecxi)\matr{\sqrt y}{x/\sqrt y}0{1/\sqrt y}\right)\nu(x)\,dx.
\end{align*}
Decomposing $f$ as in Lemma \ref{BASICFOURIERLEM} we get
\begin{align}\label{MAINSTEP1}
=\int_\R \tf_0\left(\matr{\sqrt y}{x/\sqrt y}0{1/\sqrt y}\right)\nu(x)\,dx
\hspace{220pt}
\\\notag
+\sum_{n=1}^\infty
\sum_{(c,d)\in\wh\Z^2}e(n(d\xi_1-c\xi_2))
\int_\R
\tf_n\left(\matr**cd\matr{\sqrt y}{x/\sqrt y}0{1/\sqrt y}\right)\nu(x)\,dx.
\end{align}

Recall that $\tf_0$ is invariant under $\Gamma'$;
hence the first integral in \eqref{MAINSTEP1}
is simply a weighted average along a closed horocycle
in $X'=\Gamma'\backslash G'=\SL(2,\Z)\backslash\SL(2,\R)$,
a case which has been thoroughly studied in the literature.
One can prove, either through a careful study of the cohomological
equation and invariant distributions for the horocycle flow,
as in Flaminio and Forni, \cite{FF},
or more directly from the representation theory of $\SL(2,\R)$
as in Burger \cite{Burger}, that
\begin{align}\label{SL2REQUIDISTR}
\int_\R \tf_0\left(\matr{\sqrt y}{x/\sqrt y}0{1/\sqrt y}\right)\nu(x)\,dx
=\int_{X'}\tf_0\,d\mu'\int_\R\nu\,dx
+O\Bigl(\|\tf_0\|_{\C_{\b}^4}\,\|\nu\|_{\W^{1,1}}\,y^{\frac12}\log^3(2+1/y)\Bigr).
\end{align}
(See \cite{iha} for how to extend \cite{Burger} to the case of
a non-cocompact but cofinite group such as $\SL(2,\Z)$.
In particular \eqref{SL2REQUIDISTR} follows easily from
\cite[Thm.\ 1, Rem.\ 3.4]{iha}.)
In \eqref{SL2REQUIDISTR}, note that
\begin{align}\label{MAINTERMEXPL}
\int_{X'}\tf_0\,d\mu'
=\int_{\Gamma'\backslash G'}\int_{\T^2}f((1_2,\vecxi)M)\,d\vecxi\,d\mu'(M)
=\int_{\GaG} f\,d\mu.
\end{align}
Hence \eqref{SL2REQUIDISTR} accounts for the leading term in
\eqref{MAINTHMRES} in Theorem \ref{MAINTHM}.
We also note that the error term in \eqref{SL2REQUIDISTR} 
is subsumed by the error term in \eqref{MAINTHMRES},
since $\|\tf_0\|_{\C_{\b}^4}\leq\|f\|_{\C_{\b}^4}\leq\|f\|_{\C_{\b}^8}$.

\section{Initial discussion of the main error contribution}
\label{INITIALDISCSEC}

It now remains to treat the sum over $n\in\Z^+$ in \eqref{MAINSTEP1}.

The contribution from the terms with $c=0$ can be bounded easily.
Indeed, for each $n$ there are two such terms,
for which we can take $\smatr**cd$ to be $1_2$ and $-1_2$, respectively,
and by Lemma~\ref{DECAYTFNFROMCMNORMLEM} we have
\begin{align}\label{C0BOUND}
\int_\R\tf_n\left(\pm\matr{\sqrt y}{x/\sqrt y}0{1/\sqrt y}\right)\nu(x)
\,dx
\ll\|\nu\|_{\L^1}\|f\|_{\C_{\b}^2}\frac y{n^2}.
\end{align}
Adding this over all $n\in\Z^+$ we conclude that the
contribution from all terms with $c=0$ in \eqref{MAINSTEP1}
is $O(\|\nu\|_{\L^1}\|f\|_{\C_{\b}^2}y)$, which is clearly subsumed by
the error term in \eqref{MAINTHMRES}.

Hence from now on we focus on the terms with $c\neq0$.
The following lemma expresses the integral appearing in the
second line of \eqref{MAINSTEP1} in the Iwasawa parametrization (cf.\ \eqref{TFNIWASAWA}).
Note that in this notation, 
the fact that $\tf_n$ is left $\Gamma'_\infty$-invariant
(cf.\ Lemma \ref{BASICFOURIERLEM})
means that $\tf_n(u+1,v,\theta)\equiv\tf_n(u,v,\theta)$.
\begin{lem}\label{IWASAWACOORDSLEM1}
For any $\smatr abcd\in\Gamma'$ with $c>0$, 
and any $n\in\Z^+$, $y>0$, we have
\begin{align}\notag
\int_\R
\tf_n\left(\matr abcd\matr{\sqrt y}{x/\sqrt y}0{1/\sqrt y}\right)
\nu(x)\,dx
\hspace{100pt}
\\\label{IWASAWACOORDSLEM1RES}
=\int_0^\pi 
\tf_n\biggl(\frac ac-\frac{\sin2\theta}{2c^2y},\frac{\sin^2\theta}{c^2y},
\theta\biggr)\,
\nu\biggl(-\frac dc+y\cot\theta\biggr)\,
\frac{y\,d\theta}{\sin^2\theta}.
\end{align}
\end{lem}
\begin{remark}\label{IWASAWACOORDSLEMREM2}
In the case $c<0$ one obtains exactly the same formula, except that
$\int_0^\pi$ is replaced by $\int_{-\pi}^0$
in the right hand side of \eqref{IWASAWACOORDSLEM1RES}.
\end{remark}
\begin{proof}
By a quick computation identifying matrix entries, we find that
for any $x\in\R$, the unique $u\in\R$, $v>0$, $\theta\in\R/2\pi\Z$
satisfying
\begin{align*}
\matr abcd\matr{\sqrt y}{x/\sqrt y}0{1/\sqrt y}
=\matr 1u01\matr{\sqrt v}00{1/\sqrt v}
\matr{\cos\theta}{-\sin\theta}{\sin\theta}{\cos\theta},
\end{align*}
are given by
\begin{align*}
u=\frac ac-\frac{cx+d}{c((cx+d)^2+(cy)^2)};
\qquad
v=\frac{y}{(cx+d)^2+(cy)^2};
\qquad
\theta=\arg(cx+d+i(cy)).
\end{align*}
(Thus $\cos\theta=\frac{cx+d}{\sqrt{(cx+d)^2+(cy)^2}}$ and
$\sin\theta=\frac{cy}{\sqrt{(cx+d)^2+(cy)^2}}$.)
In particular $\theta$ is a smooth and strictly decreasing function
of $x\in\R$, with $\theta\to\pi$ as $x\to-\infty$ and $\theta\to0$
as $x\to\infty$.
We may thus take $\theta$ as a new variable of integration.
Then
\begin{align*}
\cot\theta=\frac{cx+d}{cy},\qquad\text{so that }\quad
x=-\frac dc+y\cot\theta,
\end{align*}
and furthermore
\begin{align*}
u=\frac ac-\frac{(\sin\theta)(\cos\theta)}{c^2y}
=\frac ac-\frac{\sin2\theta}{2c^2y},
\qquad
v=\frac{\sin^2\theta}{c^2y}.
\end{align*}
Hence we obtain the stated identity.
\end{proof}

Note that the map $T\mapsto -T$ gives a bijection from
$\{\smatr abcd\in\Gamma'\col c>0\}$ onto $\{\smatr abcd\in\Gamma'\col c<0\}$.
Also note that for any matrix $\smatr abcd\in\Gamma'$ with $c\neq0$
we have $a\equiv d^*\mod c$, where $d^*\in\Z$
denotes a multiplicative inverse of $d$ modulo $c$.
Hence, by Lemma \ref{IWASAWACOORDSLEM1} and Remark \ref{IWASAWACOORDSLEMREM2},
the sum in the second line of \eqref{MAINSTEP1}, excluding all
terms with $c=0$, can be expressed as
\begin{align}\label{CONTRFROMFIXEDN}
\sum_{n=1}^\infty\sum_{c=1}^\infty 
\int_{-\pi}^\pi \sum_{\substack{d\in\Z\\(c,d)=1}}
\nu\biggl(-\frac dc+y\cot\theta\biggr)\,
\tf_n\biggl(\frac{d^*}c-\frac{\sin2\theta}{2c^2y},\frac{\sin^2\theta}{c^2y},
\theta\biggr)
e\bigl(n(\sgn\theta)(d\xi_1-c\xi_2)\bigr)\,\frac{y\,d\theta}{\sin^2\theta}.
\end{align}
Of course, $\tf_n(\frac{d^*}c-\frac{\sin2\theta}{2c^2y},
\frac{\sin^2\theta}{c^2y},\theta)$ is independent of the choice of
$d^*$ since $\tf_n$ is periodic with period $1$ in its first variable.

It is clear from the way in which we have obtained \eqref{CONTRFROMFIXEDN},
and also easy to check directly,
that if we try to bound \eqref{CONTRFROMFIXEDN} by simply inserting
absolute values and using our bounds on $\tf_n$ proved in Section \ref{FOURIERDECSEC}
together with the fact that $\nu$ has compact support and bounded $\L^\infty$-norm,
we obtain that \eqref{CONTRFROMFIXEDN} stays \textit{bounded} as $y\to0$
(for fixed $f,\nu,\vecxi$).
Hence to reach our goal of proving that \eqref{CONTRFROMFIXEDN}
tends to zero as $y\to0$, it suffices to establish \textit{any}
systematic cancellation in this expression.

\begin{remark}\label{METHODREMARK}
Our approach, working with the sum in \eqref{CONTRFROMFIXEDN},
has close similarities to the following method of proving equidistribution of pieces of
closed horocycles in $X'$.   %

Let $f$ be a function on $X'$, which for simplicity we assume to be 
smooth and compactly supported, i.e.\ $f\in\C_c^\infty(X')$.
Any such $f$ can be expressed as 
\begin{align}\label{FASAUTOSUM}
f(\Gamma'g)=\sum_{\gamma\in\Gamma'}\eta(\gamma g)\qquad (\forall g\in G'),
\end{align}
for some $\eta\in\C_c^\infty(G')$. %
We wish to study the weighted average of $f$ along a closed horocycle in $X'$, 
$\int_\R f(\Gamma' U^x a(y))\nu(x)\,dx$, in the limit $y\to0$.
To do so we use \eqref{FASAUTOSUM}, and change order of summation and integration.
The contribution from all $\gamma=\smatr abcd$ with $c=0$ is seen to vanish for $y$ small, since $\eta$ has compact support.
The remaining terms are handled by expressing $\eta$ in Iwasawa coordinates (cf.\ \eqref{TFNIWASAWA}),
applying an analogue of Lemma \ref{IWASAWACOORDSLEM1}, 
and then introducing $\widetilde\eta(u,v,\theta):=\sum_{n\in\Z}\eta(u+n,v,\theta)$,
a function on $(\R/\Z)\times\R_{>0}\times(\R/2\pi\Z)$:
\begin{align}\notag
\int_\R  & f(\Gamma' U^x a(y))\nu(x)\,dx
=\sum_{c=1}^\infty\sum_{\substack{a,d\in\Z\\ad\equiv1\mod c}}
\int_{-\pi}^{\pi}
\nu\Bigl(-\frac dc+y\cot\theta\Bigr)\eta\left(\frac ac-\frac{\sin2\theta}{2c^2y},\frac{\sin^2\theta}{c^2y},\theta\right)
\,\frac{y\,d\theta}{\sin^2\theta}
\\\notag
&=\int_{-\pi}^{\pi}\sum_{c=1}^\infty\sum_{\substack{d\in\Z\\(d,c)=1}}
\nu\Bigl(-\frac dc+y\cot\theta\Bigr)
\widetilde\eta\left(\frac{d^*}c-\frac{\sin2\theta}{2c^2y},\frac{\sin^2\theta}{c^2y},\theta\right)
\,\frac{y\,d\theta}{\sin^2\theta}
\\\notag
&\approx\int_{-\pi}^{\pi}\sum_{c=1}^\infty\varphi(c)\int_{\R}\nu(x)\,dx
\int_{\R/\Z}\widetilde\eta\Bigl(u,\frac{\sin^2\theta}{c^2y},\theta\Bigr)\,du
\,\frac{y\,d\theta}{\sin^2\theta}
\\\notag
&\approx\int_{-\pi}^{\pi}\int_0^\infty\frac6{\pi^2}r
\int_{\R/\Z}\widetilde\eta\Bigl(u,\frac{\sin^2\theta}{r^2y},\theta\Bigr)\,du
\,dr\,\frac{y\,d\theta}{\sin^2\theta}
\:\int_{\R}\nu\,dx
\\\notag %
&=\frac3{\pi^2}\int_{-\pi}^{\pi}\int_\R\int_0^\infty 
\eta(u,v,\theta)\,\frac{dv\,du\,d\theta}{v^2}\:\int_{\R}\nu\,dx
=\int_{G'}\eta\,d\mu'\int_{\R}\nu\,dx
=\int_{X'}f\,d\mu'\int_{\R}\nu\,dx.
\end{align}
Here the approximate equality between the second and third lines holds since, for any large $c$,
as $d$ varies through a not too small interval of integers,
the multiplicative inverse $d^*$ becomes approximately equidistributed $\Z/c\Z$.
The next approximate equality holds since $\varphi(c)$ for large $c$ behaves like $\frac6{\pi^2}c$ on average.
The last equality in the above computation %
follows using \eqref{FASAUTOSUM} and standard unfolding.
The errors in the %
approximations %
can be bounded using Lemma \ref{EXPSUM1LEM} below
and \cite[Thm.\ 330]{HW} (together with summation by parts);
in this way one obtains, with some work, that the total 
difference between the horocycle average
$\int_\R f(\Gamma' U^x a(y))\nu(x)\,dx$ and the volume average $\int_{X'}f\,d\mu'\int_{\R}\nu\,dx$ is bounded by
$O_{f,\nu,\ve}(y^{\frac14-\ve})$ as $y\to0$.
This falls short of the optimal error bound $y^{\frac12-\ve}$ which we pointed out in Section \ref{LEADINGTERMSEC};
but it is comparable with %
the ``non-Diophantine'' part of the error bound in Theorem \ref{MAINTHM}.

The main difference between the above computation and our proof of Theorem \ref{MAINTHM}
is that we will establish \textit{cancellation} in \eqref{CONTRFROMFIXEDN}, 
caused by the oscillating factor $e(n(\sgn\theta)(d\xi_1-c\xi_2))$.
For $\xi_1$ nicely Diophantine, cancellation can be established already in the inner sum over $d$
(cf.\ Section \ref{WEAKERVERSIONPFSEC});
however when $\xi_1$ is well-approximable by rational numbers we will collect certain main contributions
from the inner sum and establish cancellation when these are added over $c$
(cf.\ Section \ref{NEWBOUNDSEC}).
\end{remark}

\section{Cancellation in an exponential sum}
\label{CANCELLATIONSEC}

The following lemma is a standard application of Weil's bound on Kloosterman sums.

\begin{lem}\label{EXPSUM1LEM}
Let $0<\eta<1$, $\alpha\in\R$, $c\in\Z^+$,
let $g_1\in\C^2(\R)$ with compact support and $g_2\in \C^2(\R/\Z)$,
and let $N$ be an arbitrary subset of $\Z$. Then
\begin{align}\notag
\sum_{\substack{d\in\Z\\(c,d)=1}}g_1\Bigl(\frac dc\Bigr)e(d\alpha)
g_2\Bigl(\frac {d^*}c\Bigr)
\hspace{300pt}
\\\label{EXPSUM1LEMRES}
=\sum_{k\in N}\biggl(\int_{\R} g_1(x) e\bigl((c\alpha-k)x\bigr)\,dx\biggr)
\Bigl(\int_{\R/\Z} g_2\Bigr)
\mu\Bigl(\frac c{(c,k)}\Bigr)
\frac{\varphi(c)}{\varphi(c/(c,k))}
\hspace{100pt}
\\\notag
+O_\eta\biggl(\|g_1\|_{\W^{1,1}}^{1-\eta}\|g_1\|_{\W^{2,1}}^\eta\biggr)
\biggl\{\|g_2\|_{\L^1(\R/\Z)}\,\sum_{k\in\Z\setminus N}
\frac{(c,k)}{1+|k-c\alpha|^{1+\eta}}
+\|g_2''\|_{\L^1(\R/\Z)}\,\sigma(c)\sqrt c\biggr\}.
\end{align}
\end{lem}

\begin{proof}
Set
\begin{align*}
h(x)=\sum_{m\in\Z}g_1(x+m)e(c\alpha(x+m)).
\end{align*}
Then $h\in\C^2(\R)$ and $h(x+n)=h(x)$ for all $n\in\Z$; hence
we may view $h$ as a function in $\C^2(\R/\Z)$.
Let the Fourier expansions of $h(x)$ and $g_2(x)$ be
\begin{align*}
h(x)=\sum_{n\in\Z}a_ne(nx)
\qquad\text{and}\qquad
g_2(x)=\sum_{m\in\Z}b_me(mx).
\end{align*}
Here 
\begin{align}
a_n=\int_{\R/\Z}h(x)e(-nx)\,dx
&=\int_\R g_1(x)e((c\alpha-n)x)\,dx,
\end{align}
and thus, by integration by parts,
$|a_n|\leq\|g_1^{(j)}\|_{\L^1}(2\pi|c\alpha-n|)^{-j}$ for $j=0,1,2$.
Hence, making use of the general inequality
$\min(A,By)\leq A^{1-\eta}B^\eta y^\eta$ (true for all $A,B,y\geq0$)
with $A=\|g_1'\|_{\L^1}$, $B=\|g_1''\|_{\L^1}$ and $y=(2\pi|c\alpha-n|)^{-1}$,
we conclude:
\begin{align*}
|a_n|\ll\frac{\|g_1\|_{\W^{1,1}}^{1-\eta}\|g_1\|_{\W^{2,1}}^\eta}{1+|c\alpha-n|^{1+\eta}},
\qquad\forall n\in\Z.
\end{align*}
Similarly, using $b_m=\int_{\R/\Z}g_2(x)e(-mx)\,dx$, %
we have $|b_0|\leq\|g_2\|_{\L^1(\R/\Z)}$ and $|b_m|\leq \|g_2''\|_{\L^1(\R/\Z)}\,|m|^{-2}$ for $m\neq0$.
Now the sum in the left hand side of \eqref{EXPSUM1LEMRES}
can be expressed as
\begin{align}\label{IWASAWACOORDSLEM1PF1}
\sum_{d\in(\Z/c\Z)^\times}h\Bigl(\frac dc\Bigr)g_2\Bigl(\frac {d^*}c\Bigr)
=\sum_{n\in\Z}\sum_{m\in\Z}a_nb_m S(n,m;c),
\end{align}
where we use standard notation for Kloosterman sums;
$S(n,m;c):=\sum_{d\in(\Z/c\Z)^\times}e(n\frac dc+m\frac{d^*}c)$.

For $m=0$, $S(n,m;c)$ is a Ramanujan sum;
\begin{align*}
S(n,0;c)=\sum_{d\in(\Z/c\Z)^\times}e\Bigl(n\frac dc\Bigr)
=\mu\Bigl(\frac c{(c,n)}\Bigr)\frac{\varphi(c)}{\varphi(c/(c,n))},
\end{align*}
and in particular $|S(n,0;c)|\leq(c,n)$
(cf., e.g., \cite[Sec.\ 3.2]{IK}).
Hence the contribution from all terms with $m=0$ in 
\eqref{IWASAWACOORDSLEM1PF1} is
\begin{align*}
&=\sum_{n\in N}a_{n}b_0S(n,0;c)+
O\biggl(\|g_1\|_{\W^{1,1}}^{1-\eta}\|g_1\|_{\W^{2,1}}^\eta
\|g_2\|_{\L^1}\sum_{n\in\Z\setminus N}\frac{(c,n)}{1+|c\alpha-n|^{1+\eta}}\biggr),
\end{align*}
and the sum over $n\in N$ expands to give the first line in the right hand side of \eqref{EXPSUM1LEMRES}.

Next for $m\neq0$ we use Weil's bound, 
$|S(n,m;c)|\leq \sigma(c)\gcd(n,m,c)^{1/2}\sqrt c$
(cf.\ \cite{Weil}, and \cite[Ch.\ 11.7]{IK}),
and $\gcd(n,m,c)^{1/2}\leq |m|^{1/2}$, %
to see that the contribution from all terms with $m\neq0$ in 
\eqref{IWASAWACOORDSLEM1PF1} is
\begin{align*}
\ll \|g_1\|_{\W^{1,1}}^{1-\eta}\|g_1\|_{\W^{2,1}}^\eta\|g_2''\|_{\L^1}\sigma(c)\sqrt c\sum_{n\in\Z}\frac1{1+|c\alpha-n|^{1+\eta}}
\sum_{m\in\Z\setminus\{0\}}|m|^{-3/2}
\\
\ll_\eta  \|g_1\|_{\W^{1,1}}^{1-\eta}\|g_1\|_{\W^{2,1}}^\eta\|g_2''\|_{\L^1}\sigma(c)\sqrt c.
\end{align*}
This completes the proof of the lemma.
\end{proof}

\section{Proof of a weaker version of Theorem \ref*{MAINTHM}}
\label{WEAKERVERSIONPFSEC}

In this section we go through the first steps of the proof of Theorem \ref{MAINTHM};
the outcome of this is a 
version of Theorem \ref{MAINTHM} which only involves the Diophantine properties of $\xi_1$
and not those of $\xi_2$; see Proposition \ref{MAINTHMWEAKERPROP} below.
This result is strong enough to imply that the error term in Theorem \ref{MAINTHM}
(as well as the left hand side in Theorem \ref{MAINABTHM})
decays like $y^{\min(\frac14,\frac1K)-\ve}$ in the case of $\xi_1$ irrational of Diophantine type $K$
(see Remark \ref{MAINTHMWEAKERPROPREM});
however for $\vecxi$ with $\xi_1\in\Q$ but $\xi_2\notin\Q$, Proposition \ref{MAINTHMWEAKERPROP} %
does not imply any equidistribution whatsoever.

Recall that our task is to bound the sum in \eqref{CONTRFROMFIXEDN}.
We write $\omega=\sgn(\theta)$, where we always assume $\theta\neq0$ so that $\omega\in\{-1,1\}$.
Applying Lemma \ref{EXPSUM1LEM} and replacing $k$ by $\omega k$ we get the following
estimate valid for any $\theta\in(-\pi,\pi)\setminus\{0\}$ and $n,c\in\Z^+$:
\begin{align}\notag
\sum_{\substack{d\in\Z\\(c,d)=1}}
\nu\biggl(-\frac dc+y\cot\theta\biggr)\,
\tf_n\biggl(\frac{d^*}c-\frac{\sin2\theta}{2c^2y},\frac{\sin^2\theta}{c^2y},
\theta\biggr)
e\bigl(n\omega %
(d\xi_1-c\xi_2)\bigr)
\hspace{140pt}
\\\notag
=\sum_{k\in N}\frac{\mu(\frac c{(c,k)})\varphi(c)}{\varphi(\frac c{(c,k)})}e\bigl(-n\omega c\xi_2\bigr)
\int_\R\nu\bigl(-x+y\cot\theta\bigr)e\bigl((cn\xi_1-k)\omega x\bigr)\,dx
\int_{\R/\Z}\tf_n\Bigl(u,\frac{\sin^2\theta}{c^2y},\theta\Bigr)\,du 
\hspace{20pt}
\\\label{PROOF1STEP0new}
+O_\eta\Bigl(\|\nu\|_{\W^{1,1}}^{1-\eta}\|\nu\|_{\W^{2,1}}^\eta\Bigr)
\Biggl\{\biggl(\int_{\R/\Z}
\biggl|\tf_n\Bigl(u,\frac{\sin^2\theta}{c^2y},\theta\Bigr)\biggr|\,du\biggr)
\sum_{k\in\Z\setminus N}\frac{(c,k)}{1+|k-cn\xi_1|^{1+\eta}}
\hspace{20pt}
\\\notag
+\biggl(\int_{\R/\Z}
\biggl|\frac{\partial^2}{\partial u^2}
\tf_n\Bigl(u,\frac{\sin^2\theta}{c^2y},\theta\Bigr)\biggr|\,du\biggr)
\sigma(c)\sqrt c\Biggr\}. 
\hspace{10pt}
\end{align}
Here $N$ is a subset of $\Z$ which we are free to choose (it may depend on $n$, $c$, $\theta$).
In the present section, we will in fact make the simple choice $N=\emptyset$!
Thus the first row in the right hand side of \eqref{PROOF1STEP0new} \textit{vanishes.}
In order to bound the remaining expressions,  note that by Lemma \ref{DECAYTFNFROMCMNORMLEM}, for any $m\in\Z_{\geq0}$ we have
\begin{align}\label{DECAYTFNFROMCMNORMLEMAPPL1}
\int_{\R/\Z}
\biggl|\tf_n\Bigl(u,\frac{\sin^2\theta}{c^2y},\theta\Bigr)\biggr|\,du
\ll_m\|f\|_{\C_{\b}^m}\Bigl(\frac{|\sin\theta|}{nc\sqrt y}\Bigr)^{m}.
\end{align}
Using this bound for both $m=0$ and a general $m\in\Z_{\geq0}$ gives 
\begin{align}
\int_{\R/\Z}
\biggl|\tf_n\Bigl(u,\frac{\sin^2\theta}{c^2y},\theta\Bigr)\biggr|\,du
\ll_m\|f\|_{\C_{\b}^m}\min\Bigl(1,\Bigl(\frac{|\sin\theta|}{nc\sqrt y}\Bigr)^m
\Bigr).
\end{align}
Similarly, by Lemma \ref{DERDECAYTFNFROMCMNORMLEM2},
we have for any $\ell\in\Z_{\geq4}$:
\begin{align}\label{DECAYTFNFROMCMNORMLEMAPPL2}
\int_{\R/\Z}\biggl|\frac{\partial^2}{\partial u^2}
\tf_n\Bigl(u,\frac{\sin^2\theta}{c^2y},\theta\Bigr)\biggr|\,du
\ll_\ell\|f\|_{\C_{\b}^{\ell+2}}\,n^{-4}
\min\Bigl(1,  %
\Bigl(\frac{|\sin\theta|}{nc\sqrt y}\Bigr)^{\ell-4}\Bigr).
\end{align}
We also compute that, for $m\geq2$ and any real $a>0$,
\begin{align}\label{PROOF1STEP2new}
\int_{-\pi}^\pi\min\Bigl(1,\bigl(a^{-1}|\sin\theta|\bigr)^m\Bigr)\,
\frac{d\theta}{\sin^2\theta}\asymp_m a^{-1}(1+a)^{1-m}.
\end{align}
Using the bounds \eqref{DECAYTFNFROMCMNORMLEMAPPL1}--\eqref{DECAYTFNFROMCMNORMLEMAPPL2} in \eqref{PROOF1STEP0new}
and then applying \eqref{PROOF1STEP2new} with $a=nc\sqrt y$,
\textit{assuming from now on that $m\geq2$ and $\ell\geq6$,}
we conclude that \eqref{CONTRFROMFIXEDN} is 
\begin{align}\notag
\ll_{m,\ell,\eta}
\|\nu\|_{\W^{1,1}}^{1-\eta}\|\nu\|_{\W^{2,1}}^\eta
\biggl\{\|f\|_{\C_{\b}^m}\,y^{1-\frac m2}\sum_{n=1}^{\infty}n^{-m}
\sum_{c=1}^{\infty} c^{-1}\bigl((n\sqrt y)^{-1}+c\bigr)^{1-m}
\sum_{k\in\Z}\frac{(c,k)}{1+|k-cn\xi_1|^{1+\eta}}
\hspace{5pt}
\\\label{PROOF1STEP3new}
+\|f\|_{\C_{\b}^{\ell+2}}\,y^{3-\frac\ell2}\sum_{n=1}^{\infty}n^{-\ell}\sum_{c=1}^{\infty}
\bigl((n\sqrt y)^{-1}+c\bigr)^{5-\ell}\,\frac{\sigma(c)}{\sqrt c}
\biggr\}.
\end{align}

\begin{lem}\label{SIGMASUMLEM}
For any $X>0$ and $m\in\Z^+$ we have
\begin{align*}
\sum_{c=1}^\infty(X+c)^{-m}\,\frac{\sigma(c)}{\sqrt c}\ll_m
\begin{cases}
X^{\frac12-m}\log(1+X)&\text{if }X\geq1
\\
1&\text{if }\: X<1.
\end{cases}
\end{align*}
\end{lem}
\begin{proof}
This follows by using $\sum_{1\leq c\leq x}\sigma(c)\ll x\log(1+x)$,
$\forall x\geq1$ (cf., e.g., \cite[(1.75)]{IK}),
and integration by parts.
\end{proof}

It follows from Lemma \ref{SIGMASUMLEM} and a simple summation over $n$
that the expression in the second line of \eqref{PROOF1STEP3new} is
$\ll \|f\|_{\C_{\b}^{\ell+2}}y^{\frac14}\log(1+y^{-1})$.

When bounding the double sum over $c$ and $k$ appearing in the first line of \eqref{PROOF1STEP3new},
it is natural to introduce the following majorant function.
\begin{align}\label{FMDEF}
\fM_\alpha(X):=\sum_{\ell=1}^\infty\min\Bigl(\frac 1{\ell^2},\frac1{X\ell\langle\ell\alpha\rangle }\Bigr)
\qquad (X>0, \:\alpha\in\R).
\end{align}
Clearly this is a decreasing function of $X$ for fixed $\alpha$, but it is never very rapidly decreasing;
in short we have
\begin{align}\label{FMDECAY}
\fM_\alpha(X_2)\leq\fM_\alpha(X_1)\leq\min\Bigl(\frac{X_2}{X_1}\fM_\alpha(X_2),\zeta(2)\Bigr),
\qquad\forall 0<X_1\leq X_2.
\end{align}
The proof is immediate, using \eqref{FMDEF}.

\begin{lem}\label{PROOF1LEM1}
Fix $\eta>0$ and $m\in\Z_{\geq3}$. Then for any $\alpha\in\R$ and $X>0$ we have
\begin{align}\label{PROOF1LEM1RES}
\sum_{c=1}^{\infty} c^{-1}(X+c)^{1-m}
\sum_{k\in\Z}\frac{(c,k)}{1+|k-c\alpha|^{1+\eta}}
\ll_{\eta,m}  \begin{cases} 
X^{2-m}\,\fM_\alpha(X)&\text{if }\: X\geq1
\\
1&\text{if }\:X<1.
\end{cases}
\end{align}
\end{lem}
\begin{proof}
For given positive integers $d$ and $c={\ell}d$ (${\ell}\in\Z^+$),
we will bound the sum over $k$ in \eqref{PROOF1LEM1RES} when further restricted by the condition
$(c,k)=d$.
Denote by $k_0$ the unique integer in the interval
$c\alpha-\frac12d<k_0\leq c\alpha+\frac12d$ which is divisible by $d$.
Then $|k_0-c\alpha|$ equals the distance from $c\alpha$ to the point set $d\Z$, viz.\ 
$|k_0-c\alpha|=d\langle\frac cd\alpha\rangle$.
Note that the set $\{k\in\Z\col (c,k)=d\}$ is contained in $d\Z=k_0+d\Z$.
This gives
\begin{align*} %
\sum_{\substack{k\in\Z\\((c,k)=d)}}
\frac{(c,k)}{1+|k-c\alpha|^{1+\eta}}
\leq\sum_{k\in k_0+d\Z}\frac d{1+|k-c\alpha|^{1+\eta}}
\ll_\eta\frac d{1+(d\langle\frac cd\alpha\rangle)^{1+\eta}},
\end{align*}
since the contribution from all terms with $k\neq k_0$ is $\ll d^{-\eta}\sum_{v=1}^\infty v^{-1-\eta}\ll_\eta d^{-\eta}$.
Hence the left hand side of \eqref{PROOF1LEM1RES} is
\begin{align*}
\ll_\eta\sum_{\ell=1}^\infty \ell^{-1}\sum_{d=1}^\infty
\frac{(X+\ell d)^{1-m}}{1+(d\langle\ell\alpha\rangle)^{1+\eta}}
\asymp_{m,\eta}\sum_{\ell=1}^\infty \ell^{-1}\left.\begin{cases}
X^{2-m}\ell^{-1}&\text{if }\:1\leq X/\ell\leq\langle\ell\alpha\rangle^{-1}
\\
X^{1-m}\langle\ell\alpha\rangle^{-1}&\text{if }\:\langle\ell\alpha\rangle^{-1}<X/\ell
\\
\ell^{1-m}&\text{if }\:X/\ell<1
\end{cases}\right\},
\end{align*}
where we used $m\geq3$ in the last step.
If $X<1$ then the above sum is $\sum_{\ell=1}^\infty \ell^{-m}\ll1$.
On the other hand if $X\geq1$ then we get
\begin{align*}
=X^{2-m}\sum_{1\leq \ell\leq X}\min\Bigl(\frac 1{\ell^2},\frac1{X\ell\langle\ell\alpha\rangle}\Bigr)
+\sum_{\ell>X}\ell^{-m}
\ll X^{2-m}\,\fM_\alpha(X).
\end{align*}
\end{proof}

In Lemma \ref{PROOF1LEM1}, of course the bound $X^{2-m}\fM_\alpha(X)$ is valid also when $X<1$, albeit wasteful.
We now get in \eqref{PROOF1STEP3new},
\textit{assuming from now on $m\geq3$,}
\begin{align*}
&y^{1-\frac m2}\sum_{n=1}^{\infty}n^{-m}
\sum_{c=1}^{\infty} c^{-1}\bigl((n\sqrt y)^{-1}+c\bigr)^{1-m}
\sum_{k\in\Z}\frac{(c,k)}{1+|k-cn\xi_1|^{1+\eta}}
\hspace{100pt}
\\
&\ll_{m,\eta}
y^{1-\frac m2}\sum_{n=1}^{\infty}n^{-m}(n\sqrt y)^{m-2}\fM_{n\xi_1}\Bigl(\frac1{n\sqrt y}\Bigr)
=\sum_{n=1}^\infty\sum_{\ell=1}^\infty\min\Bigl(\frac1{(n\ell)^2},\frac{\sqrt y}{n\ell\langle n\ell\xi_1\rangle}\Bigr)
=\widetilde\fM_{\xi_1}(y^{-1/2}),
\end{align*}
where
\begin{align}\label{TILDEFMDEF}
\widetilde\fM_{\xi_1}(X):=\sum_{k=1}^\infty\sigma(k)\min\Bigl(\frac1{k^2},\frac{1}{Xk\langle k\xi_1\rangle}\Bigr).
\end{align}
We have now proved:
\begin{prop}\label{MAINTHMWEAKERPROP}
Let $0<\eta<1$ be fixed.
Then for any $f\in \C_{\b}^8(\GaG)$,
any $\nu\in\C^2(\R)$ with compact support,
and any $\vecxi\in\R^2$, $0<y<1$,
\begin{align}\notag
\int_\R f\Bigl(\Gamma\,(1_2,\vecxi)\,U^xa(y)\Bigr)\,\nu(x)\,dx
=\int_{\GaG} f\,d\mu\int_\R\nu\,dx
\hspace{180pt}
\\\label{MAINTHMWEAKERPROPRES}
+O_\eta\Bigl(\|\nu\|_{\W^{1,1}}^{1-\eta}\|\nu\|_{\W^{2,1}}^\eta\Bigr)
\biggl\{\|f\|_{\C_{\b}^{8}}\,y^{\frac14}\log(1+y^{-1})+\|f\|_{\C_{\b}^3}\widetilde\fM_{\xi_1}(y^{-1/2})\biggr\}.
\end{align}
\end{prop}

\begin{remark}\label{MAINTHMWEAKERPROPREM}
Note that for fixed $\xi_1$, $\lim_{y\to0}\widetilde\fM_{\xi_1}(y^{-1/2})=0$ if and only if $\xi_1\notin\Q$.
Hence Proposition \ref{MAINTHMWEAKERPROP} gives an effective version of Theorem \ref{NONEFFECTIVETHM}
in the special case of $\xi_1$ irrational. 

In order to compare Proposition \ref{MAINTHMWEAKERPROP} and Theorem \ref{MAINABTHM} (or Theorem \ref{MAINTHM}) in the $y$-aspect,
we point out that
\begin{align}\label{FMXI2BND}
\widetilde\fM_{\xi_1}(y^{-1/2})\ll_\ve
\bigl(\fb^1_{\xi_1}(y)+y^{\frac14}\bigr)^{1-\ve},
\qquad\text{where }\:
\fb^1_{\xi_1}(y):=
\max_{q\in\Z^+}\min\Bigl(\frac 1{q^2},\frac{\sqrt y}{q\langle q\xi_1\rangle}\Bigr).
\end{align}
This can be proved by following the same argument as we will use later below \eqref{AAA2},
and again below \eqref{AAA4}.
Note that $\widetilde\fM_{\xi_1}(y^{-1/2})\geq\fb^1_{\xi_1}(y)$ holds trivially from the definition
\eqref{TILDEFMDEF}, and so the bound \eqref{FMXI2BND} is essentially sharp whenever 
$\fb^1_{\xi_1}(y)\gg y^{1/4}$.

On the other hand for $\xi_1$ Diophantine of type $K\geq2$, 
$\widetilde\fM_{\xi_1}(y^{-1/2})\ll_{\ve,\xi_1} y^{\frac1K-\ve}$ (cf.\ Lemma~\ref{MXIBOUNDLEM});
in particular if $K<4$ then $\widetilde\fM_{\xi_1}(y^{-1/2})$ decays more rapidly than $y^{\frac14}$ as $y\to0$.
We expect that for $\vecxi$ satisfying an appropriate Diophantine condition, the error bound in \eqref{MAINTHMWEAKERPROPRES}
can be improved to $O_{f,\nu,\vecxi,\ve}(y^{\frac12-\ve})$.
This is the rate which one obtains when $f$ is a lift of a function on $X'$ (cf.\ \cite{FF}, \cite{iha});
furthermore the exponent $\frac12$ corresponds to the exponential rate of mixing for the flow $\Phi^\R$ on $X$,
when acting on sufficiently smooth vectors in $\L^2(X)$
(cf.\ \cite[Thm.\ 3.3.10]{HOWETAN} as well as \cite{edwards}).
In our approach we are stuck at the exponent $\frac14$ %
since in \eqref{CONTRFROMFIXEDN}
we bound the absolute value of the sum over $d$ individually for each $c$ using the Weil bound; cf.\ Lemma~\ref{EXPSUM1LEM}.
\end{remark}

\begin{lem}\label{MXIBOUNDLEM}
Let $\ve>0$, $X\geq1$ and $\xi\in\R$, and assume $\langle n\xi\rangle\geq cn^{1-K}$ for all 
$n\in\Z^+$ and some fixed $c>0$ and $K\geq2$. Then
\begin{align}\label{MXIBOUNDLEMRES1}
\widetilde\fM_\xi(X)\ll_\ve c^{-\frac2K}X^{\ve-\frac2K}.
\end{align}
\end{lem}

(This bound is essentially optimal. Indeed,
if $\langle n\xi\rangle\leq cn^{1-K}$ holds for some $n$ then already the single term
$\sigma(n)\min(\frac1{n^2},\frac1{Xn\langle n\xi\rangle})$ equals
$\sigma(n)(cX)^{-\frac2K}$ when $X=n^K/c$.)

\begin{proof}
We assume $cX>1$ since otherwise the stated bound is trivial.
Let $\frac{p_j}{q_j}$ for $j\in\Z_{\geq0}$ be the $j$th convergent of the
(simple) continued fraction expansion of $\xi$
(cf., e.g., \cite[Ch.\ X]{HW}).
Thus $1=q_0\leq q_1<q_2<\cdots$.
Now for any $\ell\geq1$ we have
\begin{align*}
\sum_{1\leq n\leq q_\ell/2}\frac{\sigma(n)}{n\langle n\xi\rangle}
\ll_\ve q_\ell^\ve\sum_{j=1}^\ell\sum_{q_{j-1}/2<n\leq q_j/2}\frac1{n\langle n\xi\rangle}
\ll q_\ell^\ve\sum_{j=1}^\ell q_{j-1}^{-1}\sum_{1\leq n\leq q_j/2}
\frac1{\langle n\xi\rangle}
\ll  q_\ell^\ve\sum_{j=1}^\ell \frac{q_j\log q_j}{q_{j-1}},
\end{align*}
where the last bound follows from \cite[Lemma 4.8]{Nathanson},
since $|\xi-\frac{p_j}{q_j}|<\frac1{q_jq_{j+1}}$ (\cite[Thm.\ 171]{HW}).
But for every $j\geq1$ we have 
$cq_{j-1}^{1-K}\leq\langle q_{j-1}\xi\rangle<q_j^{-1}$, i.e.\ $q_j<c^{-1}q_{j-1}^{K-1}$.
Hence
\begin{align}\label{MXIBOUNDLEMPF4}
\sum_{1\leq n\leq q_\ell/2}\frac{\sigma(n)}{Xn\langle n\xi\rangle}
&\ll_\ve (cX)^{-1}q_\ell^\ve\sum_{j=0}^{\ell-1} q_{j}^{K-2}\log q_{j+1}
\ll_\ve (cX)^{-1}q_\ell^{2\ve}\ell q_{\ell-1}^{K-2}
\ll_\ve (cX)^{-1}q_{\ell}^{3\ve}q_{\ell-1}^{K-2},
\end{align}
where we used the fact that $q_\ell$ is bounded below by the $\ell$th Fibonacci number.

Next note that for any $\ell\geq1$ and $h\geq1$,
by \cite[Lemma 4.9]{Nathanson},
\begin{align}\label{MXIBOUNDLEMPF1}
\sum_{hq_{\ell}+1\leq n\leq (h+1)q_{\ell}}\sigma(n)
\min\Bigl(\frac1{n^2},\frac1{Xn\langle n\xi\rangle}\Bigr)
\ll_\ve X^{-1}(hq_{\ell})^{\ve-1}
\sum_{r=1}^{q_{\ell}}\min\Bigl(\frac X{hq_{\ell}},\frac1{\langle (hq_{\ell}+r)\xi\rangle}\Bigr)
\\\notag
\ll(hq_{\ell})^\ve\Bigl(\frac1{(hq_\ell)^2}+\frac{\log q_{\ell}}{Xh}\Bigr).
\end{align}
Similarly also
\begin{align}\label{MXIBOUNDLEMPF2}
\sum_{q_\ell/2<n\leq q_\ell}\sigma(n)
\min\Bigl(\frac1{n^2},\frac1{Xn\langle n\xi\rangle}\Bigr)
\ll_\ve X^{-1}q_\ell^{\ve-1}\sum_{r=1}^{q_{\ell}}\min\Bigl(\frac X{q_{\ell}},\frac1{\langle r\xi\rangle}\Bigr)
\ll q_\ell^\ve\Bigl(\frac1{q_{\ell}^2}+\frac{\log q_{\ell}}{X}\Bigr).
\end{align}
Adding \eqref{MXIBOUNDLEMPF2} and \eqref{MXIBOUNDLEMPF1} for all $h\leq X/q_\ell$ we obtain
\begin{align}\label{MXIBOUNDLEMPF3}
\sum_{q_\ell/2<n\leq X}\sigma(n)\min\Bigl(\frac1{n^2},\frac{1}{Xn\langle n\xi\rangle}\Bigr)
\ll_\ve X^{\ve}\Bigl(\frac1{q_\ell^2}+\frac{\log q_\ell}X\Bigr).
\end{align}
(This is valid, trivially, also if $X\leq q_\ell/2$.)
Now choose $\ell\geq1$ so that $q_{\ell-1}\leq(cX)^{\frac1K}<q_\ell$.
Then
$q_\ell<c^{-1}q_{\ell-1}^{K-1}\leq (cX)^{-\frac1K}X<X$.
Now \eqref{MXIBOUNDLEMRES1} follows by adding
\eqref{MXIBOUNDLEMPF4}, \eqref{MXIBOUNDLEMPF3}
and the bound $\sum_{n>X}\sigma(n)n^{-2}\ll_\ve X^{\ve-1}$,
replacing $\ve$ by $\frac13\ve$,
and using $(cX)^{-\frac 2K}>X^{-\frac2K}\geq X^{-1}$. %
\end{proof}

\section{Proof of Theorem \ref*{MAINTHM}}
\label{NEWBOUNDSEC}

We will now make a choice of the set $N$ in \eqref{PROOF1STEP0new} which will allow us to reach a reasonable bound also when
$\xi_1$ is rational or well-approximable by rational numbers, provided that $\xi_2$ has good Diophantine properties.
Given any irrational %
number $\alpha$, let $\frac{p_j}{q_j}$ ($j\in\Z_{\geq0}$) be the $j$th convergent of the
(simple) continued fraction expansion of $\alpha$ (cf., e.g., \cite[Ch.\ X]{HW}; thus $1=q_0\leq q_1<q_2<\cdots$),
and set, for each $c\in\Z^+$,
\begin{align}\label{NCALFDEF}
N_c^{(\alpha)}:=\Bigl\{k\in\Z\col\frac kc\in\Bigl\{\frac{p_0}{q_0},\frac{p_1}{q_1},\ldots\Bigr\}\Bigr\}.
\end{align}

We will choose $N=N_c^{(n\xi_1)}$ in \eqref{PROOF1STEP0new}.
In order for this to make sense %
we have to assume that \textit{$\xi_1$ is irrational}.\label{XI2IRRASS}
This assumption is made merely for notational convenience,
to ensure that the continued fraction expansion of $n\xi_1$ is not finite. %
Note that the assumption can be made without loss of generality:  %
if \eqref{MAINTHMRES} holds whenever $\xi_1$ is irrational then it
must also hold when $\xi_1$ is rational, because all expressions involved depend continuously on $\xi_1$.
(There is some flexibility in the possible choices of the set $N$ in \eqref{PROOF1STEP0new} which make the proof work; 
cf.\ Remark~\ref{NCHOICEREM} below; however the choice made here is notationally convenient.)

We will use the following lemma to bound the contribution from the sum over $k\in\Z\setminus N$ in \eqref{PROOF1STEP0new}
to the expression in \eqref{CONTRFROMFIXEDN}.
\begin{lem}\label{PROOF1LEM1TEST1}
Fix $\eta>0$ and $m\in\Z_{\geq3}$.
Then for any irrational $\alpha\in\R$, and any $X>0$,
\begin{align}\label{PROOF1LEM1TEST1RES}
\sum_{c=1}^\infty\sum_{k\in\Z\setminus {N_c^{(\alpha)}}} c^{-1}(X+c)^{1-m}
\frac{(c,k)}{1+|k-c\alpha|^{1+\eta}}
\ll_{m,\eta}  \begin{cases} 
{\displaystyle X^{\frac32-m}}
&\text{if }\: X\geq1
\\
1&\text{if }\:X<1.
\end{cases}
\end{align}
\end{lem}

\begin{proof}
Introduce $d$, $\ell$, $k_0$ as in the proof of Lemma \ref{PROOF1LEM1}.
Note that if $|k_0-c\alpha|<\frac d{2{\ell}}$
then $|\frac{k_0/d}{\ell}-\alpha|<\frac1{2{\ell}^2}$,
which implies that $\frac{k_0}c=\frac{k_0/d}{\ell}$ is a convergent of the continued fraction expansion of $\alpha$
\cite[\href{http://file://T:hardy/hardywright1938.pdf:169}{Thm.\ 184}]{HW}, viz.\  $k_0\in N_c^{(\alpha)}$.
Hence
\begin{align*}
\sum_{\substack{k\in\Z\setminus N_c^{(\alpha)}\\((c,k)=d)}}
\frac{(c,k)}{1+|k-c\alpha|^{1+\eta}}
\leq\sum_{\substack{v\in\Z\\ v\neq0\text{ or }|k_0-c\alpha|\geq d/2\ell}}\frac d{1+|k_0+dv-c\alpha|^{1+\eta}}
\hspace{100pt}
\\
\ll\left.\begin{cases}\frac d{1+|k_0-c\alpha|^{1+\eta}}&\text{if }|k_0-c\alpha|\geq d/2\ell
\\ 0&\text{otherwise}\end{cases}\right\}+d^{-\eta}\sum_{v=1}^\infty v^{-1-\eta}
\ll_\eta\frac d{1+(d/{\ell})^{1+\eta}}
\ll\frac{d\ell}{d+\ell}.
\end{align*}
Hence the left hand side of \eqref{PROOF1LEM1TEST1RES} is (using $m\geq3$):
\begin{align*}
\ll_\eta\sum_{d=1}^\infty\sum_{\ell=1}^{\infty}\frac{(X+\ell d)^{1-m}}{d+\ell}
\ll\sum_{d=1}^\infty\sum_{\ell=d}^{\infty}\frac{(X+\ell d)^{1-m}}\ell
\ll\begin{cases} 
{\displaystyle X^{\frac32-m}}
&\text{if }\: X\geq1
\\
1&\text{if }\:X<1.
\end{cases}
\end{align*}
\end{proof}

By following the steps leading to \eqref{PROOF1STEP3new}, with $m=3$ and $\ell=6$,
and also using 
Lemma \ref{PROOF1LEM1TEST1} and Lemma \ref{SIGMASUMLEM},
it follows that the contribution from the
last two lines of \eqref{PROOF1STEP0new} to the expression in \eqref{CONTRFROMFIXEDN} is:
\begin{align*}
\ll\|\nu\|_{\W^{1,1}}^{1-\eta}\|\nu\|_{\W^{2,1}}^\eta\|f\|_{\C_{\b}^8}\,y^{\frac14}\log(1+y^{-1}).
\end{align*}

\begin{remark}
The simple bound in Lemma \ref{PROOF1LEM1TEST1} is wasteful in the $X$-aspect for any $\alpha$ of
Diophantine type $K<4$; cf.\ Lemma \ref{PROOF1LEM1} and Lemma \ref{MXIBOUNDLEM};
however this does not matter for us, since %
the end result is anyway subsumed by the $y^{\frac14}\log(1+y^{-1})$ bound
coming from the last line of \eqref{PROOF1STEP0new}.
The fact that, in this paper, we are not aiming to get below the exponent $\frac14$,
will also be convenient at certain steps later in our discussion; cf.\ pp.\ \pageref{AAA4}--\pageref{LASTBOUND}.
\end{remark}

It remains to bound the contribution from the first line in the right hand side of \eqref{PROOF1STEP0new} 
to the expression in \eqref{CONTRFROMFIXEDN}.
This contribution equals:
\begin{align}\notag
=\sum_{n=1}^\infty\sum_{c=1}^\infty\int_{-\pi}^\pi
\sum_{k\in N_c^{(n\xi_1)}}\biggl(\int_\R\nu\bigl(-x+y\cot\theta\bigr)
e\bigl((cn\xi_1-k)\omega x\bigr)\,dx
\int_{\R/\Z}\tf_n\Bigl(u,\frac{\sin^2\theta}{c^2y},\theta\Bigr)\,du\biggr)
\hspace{20pt}
\\\label{PROOF1STEP4new}
\times\mu\Bigl(\frac c{(c,k)}\Bigr)
\frac{\varphi(c)}{\varphi(c/(c,k))}
e\bigl(-n\omega c\xi_2\bigr)
\frac{y\,d\theta}{\sin^2\theta}.
\end{align}
In order to bound this, we first fix $n\in\Z^+$ and $\theta\in(-\pi,\pi)$
(assuming $\theta\neq0$), %
and write $a=y\cot\theta$, $\alpha_1:=n\xi_1$ and $\alpha_2:=-n\omega\xi_2$
($\omega=\sgn(\theta)$ as before).
Now
\begin{align}\notag
&\sum_{c=1}^\infty\sum_{k\in N_c^{(\alpha_1)}}\biggl(\int_\R\nu\bigl(a-x\bigr)
e\bigl((c\alpha_1-k)\omega x\bigr)\,dx
\int_{\R/\Z}\tf_n\Bigl(u,\frac{\sin^2\theta}{c^2y},\theta\Bigr)\,du
\biggr)
\mu\Bigl(\frac c{(c,k)}\Bigr)
\frac{\varphi(c)}{\varphi(c/(c,k))}e(c\alpha_2)
\\\label{PROOF1STEP4new_part}
&=\sum_{j=0}^\infty\frac{\mu(q_j)}{\varphi(q_j)}\sum_{k=1}^\infty
\varphi(kq_j)e(kq_j\alpha_2)
\biggl(\int_\R\nu\bigl(a-x\bigr)e\bigl(k(q_j\alpha_1-p_j)\omega x\bigr)\,dx
\int_{\R/\Z}\tf_n\Bigl(u,\frac{\sin^2\theta}{(kq_j)^2y},\theta\Bigr)\,du\biggr),
\end{align}
where $\frac{p_j}{q_j}$ are the convergents of $\alpha_1$.
We will treat this double sum using integration by parts (cf.\ \eqref{PROOF1STEP7newpre} below), and the key task then is to
bound the following sum:
\begin{align}\label{BNTHDEF}
&B_{n,\theta}(X):=\sum_{j=0}^\infty\frac{\mu(q_j)}{\varphi(q_j)}
\sum_{1\leq k\leq X/q_j}
\varphi(kq_j)e(kq_j\alpha_2)
\int_\R\nu\bigl(a-x\bigr)e\bigl(k(q_j\alpha_1-p_j)\omega x\bigr)\,dx.
\end{align}

\begin{lem}\label{EULERPHISUMLEM}
For any $q\in\Z^+$, $\alpha\in\R$ and $X\geq1$,
\begin{align}\label{EULERPHISUMLEMRES}
\sum_{1\leq k\leq X}\varphi(kq)e(k\alpha)
\ll \sigma_1(q)X^2\sum_{1\leq j\leq X}
\min\Bigl(\frac 1{j^2},\frac1{Xj\langle j\alpha\rangle }\Bigr).
\end{align}
\end{lem}
\begin{proof}
For any $k\in\Z^+$ we have
\begin{align*}
\varphi(kq)=\sum_{d\mid kq}\mu\Bigl(\frac{kq}d\Bigr)d
=\sum_{d_1\mid k}\sum_{\substack{d_2\mid q\\(d_2,k/d_1)=1}}
\mu\Bigl(\frac{kq}{d_1d_2}\Bigr)d_1d_2
=\sum_{\substack{d_1\mid k\\(q,k/d_1)=1}}\sum_{d_2\mid q}
\mu\Bigl(\frac{kq}{d_1d_2}\Bigr)d_1d_2,
\end{align*}
where the third equality holds since the conditions $d_1\mid k$, $d_2\mid q$, $(d_2,k/d_1)=1$ and $\mu(\frac{kq}{d_1d_2})\neq0$
together imply $(q,k/d_1)=1$.
Using this formula and then substituting $k=jd_1$, we get
\begin{align}\label{EULERPHISUMLEMPF1}
\sum_{1\leq k\leq X}\varphi(kq)e(k\alpha)
&=\sum_{d_2\mid q}d_2\sum_{\substack{1\leq j\leq X\\(q,j)=1}}
\mu\Bigl(\frac{jq}{d_2}\Bigr)\sum_{1\leq d_1\leq X/j}d_1e(d_1j\alpha).
\end{align}
But for any $j,n\in\Z^+$ with $j\alpha\notin\Z$ we have
\begin{align*}
\sum_{d_1=1}^n  d_1e(d_1j\alpha) 
= \frac{n e((n+2)j \alpha) - (n+1) e((n+1)j \alpha)+ e(j \alpha)}{(e(j \alpha)-1)^2}
\ll \min\Bigl(n^2, \frac{n}{\langle j \alpha\rangle}+\frac{1}{\langle j \alpha\rangle^2}\Bigr)
\\
\ll\min\Bigl(n^2, \frac{n}{\langle j \alpha\rangle}\Bigr),
\end{align*}
and the last %
bound is valid also when $j\alpha\in\Z$.
Using this bound in \eqref{EULERPHISUMLEMPF1} we obtain \eqref{EULERPHISUMLEMRES}.
\end{proof}

\begin{lem}\label{EULERPHISUMAPPLLEM}
Let $q\in\Z^+$, $\alpha,\beta\in\R$, $Y\geq1$ and $g\in\C_c(\R)$.
Let $L$ and $L'$ be positive real numbers such that
$\supp(g)\subset[A,A+L]\subset[-L',L']$, for some $A\in\R$.
Then
\begin{align}\notag
&\sum_{1\leq k\leq Y}\varphi(kq)e(k\alpha)\int_\R g(x)e(k\beta x)\,dx
\\\label{EULERPHISUMAPPLLEMRES}
&\ll
\|g\|_{\L^\infty}\,\sigma_1(q)
\begin{cases}Y\min\bigl(LY,|\beta|^{-1}\bigr)
\bigl(1+\log^+\bigl(LY|\beta|\bigr)\bigr)^2\,
\fM_{\alpha}\bigl(\frac{\min(LY,|\beta|^{-1})}{L'}\bigr)
&\text{if } \:L|\beta|\leq10;
\\
LY(1+\log Y)^2&\text{if } \:L|\beta|\geq\frac1{10}.
\end{cases}
\end{align}
(Thus both bounds are valid when $\frac1{10}\leq L|\beta|\leq10$.)
\end{lem}

\begin{proof}
Changing order of summation and integration and applying
Lemma \ref{EULERPHISUMLEM}, we have
\begin{align}\label{EULERPHISUMAPPLLEMPF5}
\sum_{1\leq k\leq Y}\varphi(kq)e(k\alpha)\int_\R g(x)e(k\beta x)\,dx
\ll\sigma_1(q)Y^2\sum_{1\leq j\leq Y}
\int_\R\bigl|g(x)\bigr|
\min\Bigl(\frac1{j^2},\frac1{Yj\langle j(\alpha+\beta x)\rangle}\Bigr)\,dx.
\end{align}
Here
\begin{align}\notag
\int_\R\bigl|g(x)\bigr|
\min\Bigl(\frac1{j^2},\frac1{Yj\langle j(\alpha+\beta x)\rangle}\Bigr)\,dx
\leq\|g\|_{\L^\infty}\int_A^{A+L}
\min\Bigl(\frac1{j^2},\frac1{Yj\langle j(\alpha+\beta x)\rangle}\Bigr)\,dx
\\\label{EULERPHISUMAPPLLEMPF1}
\leq\|g\|_{\L^\infty}\int_{-L}^{L}
\min\Bigl(\frac1{j^2},\frac1{Yj\langle j\beta x\rangle}\Bigr)\,dx,
\end{align}
where the last inequality is holds since %
the function $x\mapsto\min(\frac1{j^2},\frac1{Yj\langle j\beta x\rangle})$ is even,
periodic with period $\frac1{j|\beta|}$,
and decreasing in $[0,\frac1{2j|\beta|}]$ (if $\beta\neq0$).
If $Lj|\beta|\geq\frac1{10}$ then \eqref{EULERPHISUMAPPLLEMPF1} is (using $1\leq j\leq Y$):
\begin{align}\label{EULERPHISUMAPPLLEMPF4}
\ll\|g\|_{\L^\infty}\, L\int_{\R/\Z}\min\Bigl(\frac1{j^2},\frac1{Yj\langle y\rangle}\Bigr)
\,dy
\ll\frac{\|g\|_{\L^\infty}\, L}{jY}\Bigl(1+\log(Y/j)\Bigr).
\end{align}
On the other hand if $Lj|\beta|\leq1$ then \eqref{EULERPHISUMAPPLLEMPF1} is
\begin{align}\label{EULERPHISUMAPPLLEMPF3}
\ll\|g\|_{\L^\infty} \frac1{j|\beta|}\int_0^{Lj|\beta|}
\min\Bigl(\frac1{j^2},\frac1{Yjy}\Bigr)\,dy
\asymp
\frac{\|g\|_{\L^\infty}}{j^2}\min\Bigl(L,\frac1{Y|\beta|}\Bigr)
\Bigl(1+\log^+\bigl(LY|\beta|\bigr)\Bigr).
\end{align}

We also note an alternative bound in a special case:
If $\langle j\alpha\rangle\geq 2jL'|\beta|$ then
for all $x$ in the support of $g$ we have
$|j\beta x|\leq\frac12\langle j\alpha\rangle$ and thus
$\langle j(\alpha+\beta x)\rangle\geq\frac12\langle j\alpha\rangle$; therefore
\begin{align}\label{EULERPHISUMAPPLLEMPF2}
\int_\R\bigl|g(x)\bigr|
\min\Bigl(\frac1{j^2},\frac1{Yj\langle j(\alpha+\beta x)\rangle}\Bigr)\,dx
\ll\|g\|_{\L^\infty}\,L\min\Bigl(\frac1{j^2},\frac1{Yj\langle j\alpha\rangle}\Bigr).
\end{align}

If $L|\beta|\geq\frac1{10}$ then \eqref{EULERPHISUMAPPLLEMPF5} and \eqref{EULERPHISUMAPPLLEMPF4}
immediately imply that the second bound in \eqref{EULERPHISUMAPPLLEMRES} holds.
Hence from now on we assume $L|\beta|\leq10$, and our task is to prove the first bound in \eqref{EULERPHISUMAPPLLEMRES}.

\vspace{3pt}

\textbf{Case I: $LY|\beta|\geq1$.}
In this case %
we split our sum over $1\leq j\leq Y$ into two parts
corresponding to $j<(L|\beta|)^{-1}$ and $j\geq(L|\beta|)^{-1}$, respectively.
For $1\leq j<(L|\beta|)^{-1}$ we have, by 
\eqref{EULERPHISUMAPPLLEMPF3} and \eqref{EULERPHISUMAPPLLEMPF2}:
\begin{align*}
\int_\R\bigl|g(x)\bigr|
\min\Bigl(\frac1{j^2},\frac1{Yj\langle j(\alpha+\beta x)\rangle}\Bigr)\,dx
\ll
\|g\|_{\L^\infty}  %
\left.\begin{cases}
\frac1{Y|\beta|j^2}(1+\log(LY|\beta|))\quad\:\:\text{(in general)}
\\
\frac L{Yj\langle j\alpha\rangle}\qquad\qquad\quad
(\text{if }\: \langle j\alpha\rangle\geq2jL'|\beta|)
\end{cases}\right\}
\\
\ll\|g\|_{\L^\infty}\frac1{Y|\beta|}
\min\Bigl(\frac1{j^2},\frac{L'|\beta|}{j\langle j\alpha\rangle}\Bigr)
\Bigl(1+\log(LY|\beta|)\Bigr),
\end{align*}
where we used $L\leq2L'$.
Hence
\begin{align*}
\sum_{1\leq j<(L|\beta|)^{-1}}
\int_\R\bigl|g(x)\bigr|
\min\Bigl(\frac1{j^2},\frac1{Yj\langle j(\alpha+\beta x)\rangle}\Bigr)\,dx
\ll
\|g\|_{\L^\infty}\frac{1+\log(LY|\beta|)}{Y|\beta|}
\sum_{j=1}^\infty
\min\Bigl(\frac1{j^2},\frac{L'|\beta|}{j\langle j\alpha\rangle}\Bigr).
\end{align*}
For the remaining sum we have, by \eqref{EULERPHISUMAPPLLEMPF4},
\begin{align*}
\sum_{(L|\beta|)^{-1}\leq j\leq Y}
\int_\R\bigl|g(x)\bigr|
\min\Bigl(\frac1{j^2},\frac1{Yj\langle j(\alpha+\beta x)\rangle}\Bigr)\,dx
\ll\frac{\|g\|_{\L^\infty}L}Y\sum_{(L|\beta|)^{-1}\leq j\leq Y}\frac1j\Bigl(1+\log(Y/j)\Bigr)
\\
\ll\|g\|_{\L^\infty}\frac LY\Bigl(1+\log(LY|\beta|)\Bigr)^2.
\end{align*}
Now \eqref{EULERPHISUMAPPLLEMRES} follows from \eqref{EULERPHISUMAPPLLEMPF5}
and the last two bounds, since we are assuming $Y^{-1}\leq L|\beta|\leq10$.

\vspace{3pt}

\textbf{Case II: $LY|\beta|<1$.}
Then %
for all $1\leq j\leq Y$,
by \eqref{EULERPHISUMAPPLLEMPF3} and \eqref{EULERPHISUMAPPLLEMPF2},
\begin{align*}
\int_\R\bigl|g(x)\bigr|
\min\Bigl(\frac1{j^2},\frac1{Yj\langle j(\alpha+\beta x)\rangle}\Bigr)\,dx
\ll
\|g\|_{\L^\infty}  %
\left.\begin{cases}
\frac L{j^2}&\text{(in general)}
\\
\frac L{Yj\langle j\alpha\rangle}&
(\text{if }\: \langle j\alpha\rangle\geq2jL'|\beta|)
\end{cases}\right\}
\\
\ll\|g\|_{\L^\infty}\,L
\min\Bigl(\frac1{j^2},\frac{L'/L}{Yj\langle j\alpha\rangle}\Bigr).
\end{align*}
Hence again \eqref{EULERPHISUMAPPLLEMRES} follows from 
\eqref{EULERPHISUMAPPLLEMPF5}.
\end{proof}

Recall that we wish to bound $B_{n,\theta}(X)$ for $X\geq1$ (cf.\ \eqref{BNTHDEF}).
For each $j\geq0$ such that $q_j\leq X$, we apply Lemma \ref{EULERPHISUMAPPLLEM} with 
$g(x)=\nu(a-x)$,
$\beta=(q_j\alpha_1-p_j)\omega$,
$\alpha=q_j\alpha_2$,
$q=q_j$ and $Y=X/q_j$.
Note that $\supp(g)\subset[a-L,a+L]$, since $\supp(\nu)\subset[-L,L]$ by assumption.
Also $(2q_{j+1})^{-1}<|\beta|<q_{j+1}^{-1}$ (cf., e.g., \cite[Ch.\ X]{HW}).
Hence Lemma \ref{EULERPHISUMAPPLLEM} implies
\begin{align*}
&\sum_{1\leq k\leq X/q_j}
\varphi(kq_j)e(kq_j\alpha_2)
\Bigl(\int_\R\nu\bigl(a-x\bigr)
e\bigl(k(q_j\alpha_1-p_j)\omega x\bigr)\,dx\Bigr)
\\
&\ll\|\nu\|_{\L^\infty}\sigma_1(q_j)
\begin{cases} \frac X{q_j}\min(\frac{LX}{q_j},q_{j+1})\bigl(1+\log^+(\frac{LX}{q_jq_{j+1}})\bigr)^2
\fM_{q_j\alpha_2}\bigl(\frac{\min(LX/q_j\, ,\, q_{j+1})}{L+|a|}\bigr)
&\text{if }\: q_{j+1}\geq L
\\[5pt]
\frac{LX}{q_j}(1+\log(\frac{X}{q_j}))^2&\text{if }\: q_{j+1}<L.
\end{cases}
\end{align*}

In order to get a bound on $B_{n,\theta}(X)$, we multiply the last bound with $|\mu(q_j)|\varphi(q_j)^{-1}$,
and then add over all $j\geq0$ for which $q_j\leq X$.
We split the set of these $j$ into three disjoint parts, according to the following conditions:
\begin{align}\label{CASESPLITTING1}
\begin{cases}
\text{(i) $q_{j+1}<L$;}
\\
\text{(ii) $q_{j+1}\geq L$ and $q_jq_{j+1}\leq LX$;}
\\
\text{(iii) $q_{j+1}\geq L$ and $q_jq_{j+1}>LX$.}
\end{cases}
\end{align}
We thus obtain
\begin{align}\notag
B_{n,\theta}(X)\ll\|\nu\|_{\L^\infty}\biggl\{
&\sum^{(i)}_j\frac{|\mu(q_j)|\sigma_1(q_j)}{\varphi(q_j)}\frac{LX}{q_j}\Bigl(1+\log\Bigl(\frac X{q_j}\Bigr)\Bigr)^2
\\\label{PROOF1STEP6new}
&+\sum^{(ii)}_j\frac{|\mu(q_j)|\sigma_1(q_j)}{\varphi(q_j)}
\frac{Xq_{j+1}}{q_j}\Bigl(1+\log\Bigl(\frac{LX}{q_jq_{j+1}}\Bigr)\Bigr)^2
\fM_{q_j\alpha_2}\Bigl(\frac{q_{j+1}}{L+|a|}\Bigr)
\\\notag
&+\sum^{(iii)}_j\frac{|\mu(q_j)|\sigma_1(q_j)}{\varphi(q_j)}
\frac{LX^2}{q_j^2}\fM_{q_j\alpha_2}\Bigl(\frac{LX}{q_j(L+|a|)}\Bigr)
\biggr\}.
\end{align}

Let us first record a trivial bound on $B_{n,\theta}(X)$.

\begin{lem}\label{SIGMA1PHILEM}
For any positive integer $q$, $\frac{|\mu(q)|\sigma_1(q)}{\varphi(q)}\ll(\log\log(q+2))^4.$
\label{LOGLOGQJ4BOUND}   %
\end{lem}
\begin{proof}
The bound is trivial unless $\mu(q)\neq0$, i.e.\ $q$ is squarefree. For squarefree $q$,
\begin{align*}
\frac{|\mu(q)|\sigma_1(q)}{\varphi(q)}=\prod_{p\mid q}\frac{p+1}{p-1}
\leq\prod_{p\mid q}(1+4p^{-1})
\leq\prod_{p\mid q}(1+p^{-1})^4
\ll(\log\log(q+2))^4,
\end{align*}
e.g. by \cite[Theorems 4 and 7]{Ingham}.
\end{proof}

\begin{lem}\label{TRIVBNTHXBOUNDLEM}
For all $X\geq1$, $B_{n,\theta}(X)\ll\|\nu\|_{\L^\infty}\, LX^2$.
\end{lem}
\begin{proof}
Using Lemma \ref{SIGMA1PHILEM} and
(for (ii)) the fact that $u(1+\log(LX/u))^2\ll LX$ for all $u\in[1,LX]$,
we obtain (for any fixed $\ve>0$)
\begin{align*}
B_{n,\theta}(X)\ll_\ve\|\nu\|_{\L^\infty}\biggl\{LX\bigl(1+\log X\bigr)^2\sum_j^{(i)}q_j^{-1+\ve}
+LX^2\sum_j^{(ii)}q_j^{\ve-2}+LX^2\sum_j^{(iii)}q_j^{\ve-2}\biggr\}.
\end{align*}
Here the first sum is bounded using the fact that the sequence $\{q_j\}$ grows geometrically
in the precise sense that $q_{j+2}\geq q_{j+1}+q_j\geq2q_j$ for all $j$ (cf.\ \cite[(10.2.2)]{HW});
the remaining two sums are bounded trivially. This gives the stated bound.
\end{proof}

Using \eqref{PROOF1STEP4new_part}, \eqref{BNTHDEF} and \eqref{DECAYTFNFROMCMNORMLEMAPPL1} with $m=3$,
together with the fact that $B_{n,\theta}(X)\ll_{\nu,L} X^2$ as $X\to\infty$, by Lemma \ref{TRIVBNTHXBOUNDLEM},
we see that the expression in \eqref{PROOF1STEP4new} can be rewritten as:
\begin{align}\label{PROOF1STEP7newpre}
-\sum_{n=1}^\infty\int_{-\pi}^{\pi}
\int_1^\infty\biggl(\frac{\partial}{\partial X}\int_{\R/\Z}
\tf_n\Bigl(u,\frac{\sin^2\theta}{X^2y},\theta\Bigr)\,du
\biggr)\,B_{n,\theta}(X)\,dX\,\frac{y\,d\theta}{\sin^2\theta}.
\end{align}
By Lemma \ref{DERDECAYTFNFROMCMNORMLEM2} we have, for any fixed $m\in\Z_{\geq0}$,
\begin{align*}
\frac{\partial}{\partial X}\int_{\R/\Z}\tf_n
\Bigl(u,\frac{\sin^2\theta}{X^2y},\theta\Bigr)\,du
\ll_m\|f\|_{\C_{\b}^{m+1}}\,X^{-1}
\min\Bigl(1,\Bigl(\frac{|\sin\theta|}{nX\sqrt y}\Bigr)^m\Bigr).
\end{align*}
Hence the expression in \eqref{PROOF1STEP4new} is
\begin{align}\label{PROOF1STEP7new}
\ll_m\|f\|_{\C_{\b}^{m+1}}\,y\sum_{n=1}^\infty\int_{-\pi}^\pi\int_1^\infty
\bigl|B_{n,\theta}(X)\bigr|\,
\min\Bigl(1,\Bigl(\frac{|\sin\theta|}{nX\sqrt y}\Bigr)^m\Bigr)\,
\frac{dX\,d\theta}{X\sin^2\theta}.
\end{align}

The following three lemmas will allow us to further simplify the bound.
\begin{lem}\label{TRIVCONTRHANDLE1LEM}
For any $m\geq 2$, $0<\sigma<1$ and $0<y<1$ we have
\begin{align*}
\sum_{n=1}^\infty\int_{-\pi}^\pi\int_1^\infty X^{1+\sigma}\,
\min\Bigl(1,\Bigl(\frac{|\sin\theta|}{nX\sqrt y}\Bigr)^m\Bigr)\,
\frac{dX\,d\theta}{X\sin^2\theta}
\ll_{m,\sigma} y^{-\frac12(1+\sigma)}.
\end{align*}
\end{lem}
\begin{proof}
Using \eqref{PROOF1STEP2new} we see that the given expression is
\begin{align}\notag
&\ll_m\sum_{n=1}^\infty\int_1^\infty
(nX\sqrt y)^{-1}(1+nX\sqrt y)^{1-m}\, X^\sigma\,dX
\ll_m\sum_{n=1}^\infty
\begin{cases} 
(n\sqrt y)^{-m}&\text{if }\:n\sqrt y\geq1
\\
(n\sqrt y)^{-1-\sigma}&\text{if }\:n\sqrt y\leq1,
\end{cases}%
\end{align}
and this gives the stated bound.
\end{proof}

The next lemma generalizes \eqref{PROOF1STEP2new}.
\begin{lem}\label{THETAINTEGRALBOUNDLEM}
For $m\geq2$, $a>0$, and $0\leq\delta\leq1$, we have
\begin{align*}%
\int_{\substack{\theta\in(-\pi,\pi)\\|\sin\theta|<\delta}}
\min\Bigl(1,\bigl(a^{-1}|\sin\theta|\bigr)^m\Bigr)\,
\frac{d\theta}{\sin^2\theta}\asymp_m 
\left.\begin{cases} a^{-m}\delta^{m-1}&\text{if }\: \delta\leq a
\\ a^{-1}&\text{if }\:\delta\geq a\end{cases}\right\}
=a^{-1}\min(1,(\delta/a)^{m-1}).
\end{align*}
\end{lem}
\begin{proof}
This is seen by a direct computation.
\end{proof}

\begin{lem}\label{BOUNDFORTHETANOTSMALLLEM}
Let $\sigma\in[0,\frac12]$, $m\geq3$ and $0<y<1$.
If the integral over $\theta$ in \eqref{PROOF1STEP7new} is restricted by the condition $|\sin\theta|\leq y^{\sigma}$,
the resulting expression is
$\ll_m\|f\|_{\C_{\b}^{m+1}}\|\nu\|_{\L^\infty}\,Ly^\sigma.$
\end{lem}
\begin{proof}
Using Lemma \ref{THETAINTEGRALBOUNDLEM} and Lemma \ref{TRIVBNTHXBOUNDLEM} we see that the expression in question is
\begin{align*}
&\ll\|f\|_{\C_{\b}^{m+1}}\|\nu\|_{\L^\infty}\,Ly
\sum_{n=1}^\infty\frac1{n\sqrt y}\int_1^\infty
\min\Bigl(1,\Bigl(\frac{y^\sigma}{nX\sqrt y}\Bigr)^{m-1}\Bigr)\,dX
\\
&\ll\|f\|_{\C_{\b}^{m+1}}\|\nu\|_{\L^\infty}\,Ly
\sum_{n=1}^\infty\frac1{n\sqrt y}\left.\begin{cases}
y^{\sigma-\frac12}/n & \text{if }\: n\leq y^{\sigma-\frac12}
\\
(y^{\sigma-\frac12}/n)^{m-1} &
\text{if }\: n\geq y^{\sigma-\frac12}
\end{cases}\right\}
\ll\|f\|_{\C_{\b}^{m+1}}\|\nu\|_{\L^\infty}\,Ly^\sigma.
\end{align*}
\end{proof}

We saw in the proof of Lemma \ref{TRIVBNTHXBOUNDLEM} that the
$\sum^{(i)}$-sum in \eqref{PROOF1STEP6new} is $\ll_\ve LX^{1+\ve}$;
hence by Lemma \ref{TRIVCONTRHANDLE1LEM} the total contribution from the
$\sum^{(i)}$-sum in \eqref{PROOF1STEP6new} to the expression in \eqref{PROOF1STEP7new} is
$\ll_{m,\ve}\|f\|_{\C_{\b}^{m+1}}\|\nu\|_{\L^\infty} Ly^{\frac12-\ve}$.
We also use Lemma \ref{BOUNDFORTHETANOTSMALLLEM} with $\sigma=\frac12$,
and note that for $|\sin\theta|>y^{\frac12}$ we have $|a|=y|\cot\theta|<y^{\frac12}<1$ and $L+|a|\leq2L$
in \eqref{PROOF1STEP6new};
hence we conclude that the whole expression in \eqref{PROOF1STEP7new} is
\begin{align}\label{PROOF1STEP7newA}
\ll_{m,\ve} 
\|f\|_{\C_{\b}^{m+1}}\|\nu\|_{\L^\infty}\biggl\{Ly^{\frac12-\ve}+y\sum_{n=1}^\infty\int_1^\infty
B_n(X)\int_{-\pi}^\pi\min\Bigl(1,\Bigl(\frac{|\sin\theta|}{nX\sqrt y}\Bigr)^m\Bigr)\,
\frac{d\theta}{\sin^2\theta}\,\frac{dX}X\biggr\},
\end{align}
where (keeping from now on $\alpha_2:=n\xi_2$, and  using the fact that $\fM_\alpha(X)\equiv\fM_{-\alpha}(X)$)
\begin{align}\label{BNDEF}
B_n(X)=\sum^{(ii)}_j\frac{|\mu(q_j)|\sigma_1(q_j)}{\varphi(q_j)}
\frac{Xq_{j+1}}{q_j}\Bigl(1+\log\Bigl(\frac{LX}{q_jq_{j+1}}\Bigr)\Bigr)^2
\fM_{q_j\alpha_2}\Bigl(\frac{q_{j+1}}{2L}\Bigr)
\\\notag
+\sum^{(iii)}_j\frac{|\mu(q_j)|\sigma_1(q_j)}{\varphi(q_j)}
\frac{LX^2}{q_j^2}\fM_{q_j\alpha_2}\Bigl(\frac{X}{2q_j}\Bigr).
\end{align}
Using \eqref{PROOF1STEP2new} to bound the integral over $\theta$ %
we conclude that the expression in \eqref{PROOF1STEP7newA} is
\begin{align}\label{PROOF1STEP7newB}
\ll_{m,\ve}\|f\|_{\C_{\b}^{m+1}}\|\nu\|_{\L^\infty}\biggl\{Ly^{\frac12-\ve}+y^{\frac12}\sum_{n=1}^\infty\frac1n
\int_1^\infty\frac{B_n(X)}{(1+n\sqrt yX)^{m-1}X^2}\,dX\biggr\}.
\end{align}

Now note that 
\begin{align}\notag
\int_1^\infty\frac{B_n(X)}{(1+n\sqrt yX)^{m-1}X^2}\,dX
=\sum_{\substack{j\geq0\\(q_{j+1}\geq L)}}\frac{|\mu(q_j)|\sigma_1(q_j)}{\varphi(q_j)}
\Biggl\{\frac L{q_j^2}\int_{1}^{q_jq_{j+1}/L}\frac{\fM_{q_j\alpha_2}(\frac{X}{2q_j})}{(1+n\sqrt yX)^{m-1}}\,dX
\\\label{KEYBNINTEGRAL1}
+\frac{q_{j+1}}{q_j}\fM_{q_j\alpha_2}\Bigl(\frac{q_{j+1}}{2L}\Bigr)
\int_{q_jq_{j+1}/L}^\infty\frac{(1+\log(\frac{LX}{q_jq_{j+1}}))^2}{(1+n\sqrt yX)^{m-1}X}\,dX
\Biggr\}.
\end{align}
Let us write $\fM_\alpha^1(X)$ for the integral of $\fM_\alpha(X)$:
\begin{align}\label{FM1DEF}
\fM_\alpha^1(X):=\int_0^X\fM_\alpha(Y)\,dY=\sum_{n=1}^\infty\min\Bigl(\frac X{n^2},\frac1{n\langle n\alpha\rangle }\Bigr)
\biggl(1+\log^+\Bigl(\frac{X\langle n\alpha\rangle }n\Bigr)\biggr).
\end{align}
We have the bound %
\begin{align*}
\int_{1}^{q_jq_{j+1}/L}\frac{\fM_{q_j\alpha_2}(\frac{X}{2q_j})}{(1+n\sqrt yX)^{m-1}}\,dX
\ll\int_0^{q_jq_{j+1}/L}\fM_{q_j\alpha_2}\Bigl(\frac{X}{2q_j}\Bigr)\,dX
=2q_j\fM_{q_j\alpha_2}^1\Bigl(\frac{q_{j+1}}{2L}\Bigr).
\end{align*}
If $q_jq_{j+1}/L>(n\sqrt y)^{-1}$ then the same integral can also be bounded more sharply as
(assuming $m\geq3$, and using the fact that $\fM_\alpha(X)\leq X^{-1}\fM_\alpha^1(X)$ for all $X$)
\begin{align*}
&\ll\int_0^{(n\sqrt y)^{-1}}\fM_{q_j\alpha_2}\Bigl(\frac{X}{2q_j}\Bigr)\,dX
+\int_{(n\sqrt y)^{-1}}^\infty \frac{\fM_{q_j\alpha_2}(\frac1{2q_jn\sqrt y})}{(n\sqrt yX)^{m-1}}\,dX
\ll 2q_j\fM_{q_j\alpha_2}^1\Bigl(\frac{1}{2q_jn\sqrt y}\Bigr).
\end{align*}
Also by an easy computation, 
\begin{align}\label{KEYBNINTEGRAL3}
\int_{q_jq_{j+1}/L}^\infty\frac{(1+\log(\frac{LX}{q_jq_{j+1}}))^2}{(1+n\sqrt yX)^{m-1}X}\,dX
\ll\begin{cases}
(1+\log(\frac L{q_jq_{j+1}n\sqrt y}))^3
&\text{if }\: q_jq_{j+1}/L\leq (n\sqrt y)^{-1}
\\[10pt]
\bigl(\frac L{q_jq_{j+1}n\sqrt y}\bigr)^{m-1}
&\text{if }\:(n\sqrt y)^{-1}\leq q_jq_{j+1}/L.
\end{cases}
\end{align}
Adding up the bounds
(using the fact that if $(n\sqrt y)^{-1}\leq q_jq_{j+1}/L$ then
$\fM_{q_j\alpha_2}(\frac{q_{j+1}}{2L})\leq2q_jn\sqrt y\,\,\fM_{q_j\alpha_2}^1(\frac1{2q_jn\sqrt y})$), we conclude
\begin{align}\label{KEYBNINTEGRAL4}
&\int_1^\infty\frac{B_n(X)}{(1+n\sqrt yX)^{m-1}X^2}\,dX
\\\notag
&\ll L\sum_{\substack{j\geq0\\(q_{j+1}\geq L)}}\frac{|\mu(q_j)|\sigma_1(q_j)}{q_j\varphi(q_j)}
\Bigl(1+\log^+\Bigl(\frac L{q_jq_{j+1}n\sqrt y}\Bigr)\Bigr)^3
\fM^1_{q_j\alpha_2}\Bigl(\min\Bigl(\frac{q_{j+1}}{L},\frac{1}{q_jn\sqrt y}\Bigr)\Bigr).
\end{align}
Using also $q_{j+1}<\langle q_j\alpha_1\rangle^{-1}<2q_{j+1}$ we get
\begin{align}\label{NCHOICEFLEXKEY}
\ll L\sum_{q=1}^\infty
\frac{|\mu(q)|\sigma_1(q)}{q\varphi(q)}
\Bigl(1+\log^+\Bigl(\frac {L\langle q\alpha_1\rangle}{qn\sqrt y}\Bigr)\Bigr)^3
\fM^1_{q\alpha_2}\Bigl(\min\Bigl(\frac{1}{L\langle q\alpha_1\rangle},\frac1{qn\sqrt y}\Bigr)\Bigr).
\end{align}
Adding now over $n$ (recalling $\alpha_1=n\xi_1$, $\alpha_2=n\xi_2$) we get, after substituting $k=qn$ and using
$\sum_{q\mid k}\frac{|\mu(q)|\sigma_1(q)}{\varphi(q)}\ll_\ve k^\ve$,
\begin{align}\notag
y^{\frac12}\sum_{n=1}^\infty\frac1n\int_1^\infty\frac{B_n(X)}{(1+n\sqrt yX)^{m-1}X^2}\,dX
\hspace{180pt}
\\\label{AAA2}
\ll_{\ve}
Ly^{\frac12}
\sum_{k=1}^\infty k^{\ve-1}
\Bigl(1+\log^+\Bigl(\frac {L\langle k\xi_1\rangle}{k\sqrt y}\Bigr)\Bigr)^3
\fM^1_{k\xi_2}\Bigl(\min\Bigl(\frac{1}{L\langle k\xi_1\rangle},\frac1{k\sqrt y}\Bigr)\Bigr).
\end{align}

\begin{remark}\label{NCHOICEREM}
In \eqref{NCHOICEFLEXKEY} we overestimate a sum over $\{q_j\}$ by a sum over \textit{all} $q\in\Z^+$;
the (simple) reason why this is not very wasteful will be seen in the next paragraph.
It is clear from this %
that there is some flexibility in choosing the set $N$ in \eqref{PROOF1STEP0new} so as to make our proof work.
\end{remark}

As we will see, the vast majority of the terms in the sum in \eqref{AAA2} can be bounded trivially.
First of all, using $\fM_{k\xi_2}^1(X)\ll X$ and the fact that for any $k,y>0$
the function $f(\delta)=(1+\log^+(\frac{\delta}{k\sqrt y}))^3\min(\delta^{-1},\frac1{k\sqrt y})$
satisfies $f(\delta_2)\ll %
f(\delta_1)$ for all $0<\delta_1\leq\delta_2$,
it follows that the contribution from all $k$ with $\langle k\xi_1\rangle\geq(2k)^{-1}$ in the right hand side of \eqref{AAA2} is
\begin{align*}
\ll Ly^{\frac12}\sum_{k=1}^\infty k^{\ve-1}
\Bigl(1+\log^+\Bigl(\frac {L}{k^2\sqrt y}\Bigr)\Bigr)^3
\min\Bigl(\frac{k}{L},\frac1{k\sqrt y}\Bigr)
\ll_\ve L^{\frac12+\ve}y^{\frac14-\ve}.
\end{align*}

It remains to consider the contribution from all $k$ with $\langle k\xi_1\rangle<(2k)^{-1}$.
From now on, let $\frac{p_j}{q_j}$ for $j\in\Z_{\geq0}$ be the $j$th convergent of the
continued fraction expansion of $\xi_1$.
Then each $k\geq1$ satisfying $\langle k\xi_1\rangle<(2k)^{-1}$ is known to be of the form $k=\ell q_j$ 
for some $j\in\Z_{\geq0}$ and some $\ell\in\Z^+$
which is so small that $\langle k\xi_1\rangle=\ell\langle q_j\xi_1\rangle$ and thus
$\frac{\ell}{2q_{j+1}}<\langle k\xi_1\rangle<\frac\ell{q_{j+1}}$ (cf., e.g., \cite[Thm.\ 184]{HW}).
Hence the total contribution from all such $k$ to the right hand side of \eqref{AAA2} is
\begin{align}\label{AAA5}
\ll Ly^{\frac12}\sum_{j=0}^\infty\sum_{\ell=1}^\infty (\ell q_j)^{\ve-1}
\Bigl(1+\log^+\Bigl(\frac{L}{q_jq_{j+1}\sqrt y}\Bigr)\Bigr)^3
\fM^1_{\ell q_j\xi_2}\Bigl(\ell^{-1}\min\Bigl(\frac{q_{j+1}}L,\frac1{q_j\sqrt y}\Bigr)\Bigr).
\end{align}
Recall the summation formula for $\fM_\alpha^1(X)$; cf.\ \eqref{FM1DEF}. Let us write $\fM_\alpha^{1,\ve}(X)$ for the
analogous sum with an extra factor $n^\ve$ in each term:
\begin{align}\label{FM1VEDEF}
\fM_\alpha^{1,\ve}(X):=\sum_{n=1}^\infty n^{\ve}
\min\Bigl(\frac X{n^2},\frac{1}{n\langle n\alpha\rangle}\Bigr)\Bigl(1+\log^+\Bigl(\frac{X\langle n\alpha\rangle}{n}\Bigr)\Bigr).
\end{align}
\begin{lem}\label{FM1SUMBOUNDLEM}
For any $X>0$, $\alpha\in\R$, $\ve>0$,
\begin{align}\label{FM1SUMBOUNDLEMRES}
\sum_{\ell=1}^\infty \ell^{\ve-1}\fM^1_{\ell\alpha}(\ell^{-1}X)\ll_\ve\fM_\alpha^{1,2\ve}(X),
\end{align}
\end{lem}
\begin{proof}
Using \eqref{FM1DEF} and substituting $k=\ell n$ we obtain
\begin{align*}
\sum_{\ell=1}^\infty \ell^{\ve-1}\fM^1_{\ell\alpha}(\ell^{-1}X)
=\sum_{k=1}^\infty\biggl(\sum_{\ell\mid k}\ell^{\ve}\biggr)
\min\Bigl(\frac X{k^2},\frac{1}{k\langle k\alpha\rangle}\Bigr)
\Bigl(1+\log^+\Bigl(\frac{X\langle k\alpha\rangle}{k}\Bigr)\Bigr).
\end{align*}
Now the desired bound follows using $\sum_{\ell\mid k}\ell^{\ve}\ll_\ve k^{2\ve}$.
\end{proof}
\begin{lem}\label{FM1VEBOUNDLEM}
For any $X>0$, $\alpha\in\R$, $0<\ve\leq\frac12$ we have $\fM_\alpha^{1,\ve}(X)\ll X$.
\end{lem}
\begin{proof}
Simply note $\fM_\alpha^{1,\ve}(X)\leq(\sum_{n=1}^\infty n^{\ve-2})X$.
\end{proof}

Using Lemma \ref{FM1SUMBOUNDLEM} we see that \eqref{AAA5} is
\begin{align}\label{AAA3}
\ll_\ve Ly^{\frac12}\sum_{j=0}^\infty q_j^{\ve-1}
\Bigl(1+\log^+\Bigl(\frac{L}{q_jq_{j+1}\sqrt y}\Bigr)\Bigr)^3
\fM^{1,2\ve}_{q_j\xi_2}\Bigl(\min\Bigl(\frac{q_{j+1}}L,\frac1{q_j\sqrt y}\Bigr)\Bigr).
\end{align}
Recall that $0<y<1$.
Now let $j_0\geq0$ be the unique index satisfying
\begin{align*}
q_{j_0}^4<y^{-\frac12}\leq q_{j_0+1}^4.
\end{align*}
Then the contribution from all $j<j_0$ in \eqref{AAA3} is (using Lemma \ref{FM1VEBOUNDLEM}, keeping $\ve\leq\frac14$)
\begin{align*}
\ll Ly^{\frac12}\sum_{j<j_0}(1+\log(Ly^{-1}))^3\frac{q_{j+1}}L\ll y^{\frac12}q_{j_0}(1+\log(Ly^{-1}))^3
<y^{\frac38}(1+\log(Ly^{-1}))^3.
\end{align*}
Also if $\frac{q_{j_0+1}}{Lq_{j_0}}\leq y^{-\frac14}$ then the contribution from $j=j_0$ in \eqref{AAA3} is
\begin{align*}
\ll Ly^{\frac12}q_{j_0}^{\ve}
\biggl(1+\log(y^{-1})+\log^+\Bigl(\frac{L}{q_{j_0+1}}\Bigr)\biggr)^3\frac{q_{j_0+1}}{Lq_{j_0}}
\ll_\ve Ly^{\frac14-\ve}.
\end{align*}
Finally the contribution from all $j>j_0$ to \eqref{AAA3} is
\begin{align*}
\ll Ly^{\frac12}(1+\log(y^{-1}))^3\sum_{j=j_0+1}^\infty q_j^{\ve-1}
\left.\begin{cases}
\frac{q_{j+1}}L(1+\log(\frac L{q_{j+1}}))^3&\text{if }\: q_{j+1}<L
\\[5pt]
\frac1{q_j\sqrt y}&\text{if }\: q_{j+1}\geq L
\end{cases}\right\}
\\
\ll Ly^{\frac12}(1+\log(y^{-1}))^3\biggl(1+y^{-\frac12}\sum_{j=j_0+1}^\infty q_j^{\ve-2}\biggr)
\ll_\ve Ly^{\frac14-\ve}.
\end{align*}
Now there remains \textit{at most one} $j$ to consider in the bound in \eqref{AAA3}:
namely that $j\geq0$, if any, which satisfies $q_j^4<y^{-\frac12}<(\frac{q_{j+1}}{Lq_j})^2$.
In the case when such a $j$ exists, let us write $q:=q_j$ and $q':=q_{j+1}$.
Collecting our bounds and recalling the definition of $\fM^{1,2\ve}_{q\xi_2}(X)$,
we have now proved that the right hand side of \eqref{AAA2} is
$\ll_\ve Ly^{\frac14-\ve}$ if the special denominator $q$ does not exist, and otherwise it is
\begin{align}\label{AAA4}
\ll_\ve Ly^{\frac14-\ve}+Ly^{\frac12}q^{\ve-1}\Bigl(1+\log^+\Bigl(\frac{L}{q'q\sqrt y}\Bigr)\Bigr)^3 
\sum_{n=1}^\infty n^{2\ve-1}
\Bigl(1+\log^+\Bigl(\frac{X\langle nq\xi_2\rangle}{n}\Bigr)\Bigr)\min\Bigl(\frac X{n},\frac{1}{\langle nq\xi_2\rangle}\Bigr)
\end{align}
where
\begin{align*}
X:=\min\Bigl(\frac1{q\sqrt y},\frac{q'}L\Bigr)>1.
\end{align*}

For the rest of this discussion we will assume that the special denominator $q$ exists.
In close analogy to what we have shown for \eqref{AAA2}, we will see that 
the vast majority of the terms in the sum in \eqref{AAA4} can be bounded trivially.
First, note that the total contribution to \eqref{AAA4} from all $n\in\Z^+$ with
$\langle nq\xi_2\rangle\geq(2n)^{-1}$ is 
\begin{align}\label{AAA4a}
\ll_\ve Ly^{\frac12}q^{\ve-1}\Bigl(1+\log^+\Bigl(\frac{L}{q'q\sqrt y}\Bigr)\Bigr)^3 (1+\log X)X^{\frac12+\ve}
\ll_\ve Ly^{\frac14-\ve},
\end{align}
where the last relation holds since $\log^+(\frac{L}{q'q\sqrt y})\leq\log(y^{-1})$ and $X\leq y^{-\frac12}$.
From now on we assume, without loss of generality, that \textit{$\xi_2$ is irrational}
(just as we did for $\xi_1$ on p.\ \pageref{XI2IRRASS}).
Let $\frac{r_j}{s_j}$ for $j\in\Z_{\geq0}$ be the $j$th convergent of the continued fraction expansion of $q\xi_2$.
Then each $n\in\Z^+$ satisfying $\langle nq\xi_2\rangle<(2n)^{-1}$ is of the form $n=\ell s_j$
for some $j\in\Z_{\geq0}$ and some $\ell\in\Z^+$
which is so small that $\langle nq\xi_2\rangle=\ell\langle s_jq\xi_2\rangle$ and thus
$\frac{\ell}{2s_{j+1}}<\langle nq\xi_2\rangle<\frac\ell{s_{j+1}}$.
Hence the total contribution from all these $n$ to the sum in \eqref{AAA4} is
\begin{align}\notag
\ll\sum_{j=0}^\infty\sum_{\ell=1}^\infty (\ell s_j)^{2\ve-1}\Bigl(1+\log^+\Bigl(\frac{X}{s_js_{j+1}}\Bigr)\Bigr)
\ell^{-1}\min\Bigl(\frac X{s_j},s_{j+1}\Bigr)
\\\label{AAA6}
\ll\sum_{j=0}^\infty s_j^{2\ve-1}\Bigl(1+\log^+\Bigl(\frac{X}{s_js_{j+1}}\Bigr)\Bigr)
\min\Bigl(\frac X{s_j},s_{j+1}\Bigr).
\end{align}
Next let $j_0\geq0$ be the unique index satisfying $s_{j_0}^4<X\leq s_{j_0+1}^4$.
Then the contribution from all $j<j_0$ in \eqref{AAA6} is
$\leq(1+\log X)\sum_{j<j_0}s_{j+1}\ll_\ve X^{\frac14+\ve}$.
Also if $\frac{s_{j_0+1}}{s_{j_0}}\leq\sqrt X$ then the contribution from $j=j_0$ in \eqref{AAA6} is
$\leq(1+\log X)s_{j_0}^{2\ve}\frac{s_{j_0+1}}{s_{j_0}}\ll_\ve X^{\frac12+\ve}$.
Finally the contribution from all $j>j_0$ in \eqref{AAA6} is
\begin{align*}
\leq X(1+\log X)\sum_{j=j_0+1}^\infty s_j^{2\ve-2}
\ll X(1+\log X)s_{j_0+1}^{2\ve-2}
\ll_\ve X^{\frac12+\ve}.
\end{align*}
Now there remains \textit{at most one} $j$ to consider in \eqref{AAA6}:
namely that $j\geq0$, if any, which satisfies $s_j^4<X<(\frac{s_{j+1}}{s_j})^2$.
In the case when such a $j$ exists, let us write $s:=s_j$ and $s':=s_{j+1}$.
Collecting our bounds (recalling also the last relation in \eqref{AAA4a}),
we have now proved that \eqref{AAA4} is
$\ll_\ve Ly^{\frac14-\ve}$ if the special denominator $s$ does not exist, and otherwise it is
\begin{align*}
&\ll_\ve Ly^{\frac14-\ve}+Ly^{\frac12}q^{\ve-1}\Bigl(1+\log^+\Bigl(\frac{L}{q'q\sqrt y}\Bigr)\Bigr)^3 
s^{2\ve-1}\Bigl(1+\log^+\Bigl(\frac{X}{ss'}\Bigr)\Bigr)\min\Bigl(\frac X{s},s'\Bigr).
\end{align*}
Recalling the definition of $X$, and writing 
$U=\frac{L}{q'q\sqrt y}$, $V=\frac 1{qss'\sqrt y}$, the above is seen to be
\begin{align}\notag
\leq Ly^{\frac14-\ve}+L(sq)^{2\ve-2}\frac{(1+\log(\max(1,U,V)))^4}{\max(1,U,V)}
\ll_\ve Ly^{\frac14-\ve}+L\bigl((sq)^{-2}\min(1,U^{-1},V^{-1})\bigr)^{1-\ve}
\\\label{LASTBOUND}
\leq L\bigl(y^{\frac14-\ve}+\fb_{\vecxi,L}(y)^{1-\ve}\bigr),
\end{align}
Here the last relation follows from \eqref{MYXILDEF},
since $q'<\langle q\xi_1\rangle^{-1}\leq s\langle sq\xi_1\rangle^{-1}$ and $s'<\langle sq\xi_2\rangle^{-1}$
and therefore, writing $q_0:=qs$, we have
$(qs)^{-2}U^{-1}<\frac{\sqrt y}{Lq_0\langle q_0\xi_1\rangle}$ and
$(qs)^{-2}V^{-1}<\frac{\sqrt y}{q_0\langle q_0\xi_2\rangle}$.
Taking $m=3$ it follows that \eqref{PROOF1STEP7newB} is
$\ll L\|\nu\|_{\L^\infty}\|f\|_{\C_{\b}^4}\bigl(y^{\frac14}+\fb_{\vecxi,L}(y)\bigr)^{1-4\ve}$.
Hence, replacing $\ve$ by $\frac14\ve$, we have now completed the proof of Theorem \ref{MAINTHM}.
\hfill $\square$ $\square$ $\square$ 

\begin{remark}\label{TRUEMAXREMARK}
Let us note that the last step in \eqref{LASTBOUND} is essentially sharp.
Indeed, if $\fb_{\vecxi,L}(y)>2y^{\frac14}$ (and $y<1$, $L\geq1$) then the ``special denominators'' $q,s$ introduced 
above do in fact exist, and
\begin{align}\label{TRUEMAXREMARKRES}
\fb_{\vecxi,L}(y)
=\min\Bigl(\frac1{(sq)^2},\frac{\sqrt y}{Ls^2q\langle q\xi_1\rangle},\frac{\sqrt y}{sq\langle sq\xi_2\rangle}\Bigr)
\end{align}
(and here $\frac12<q'\langle q\xi_1\rangle<1$ and $\frac12<s'\langle sq\xi_2\rangle<1$,
so that the expression in \eqref{TRUEMAXREMARKRES} is comparable with
$(sq)^{-2}\min(1,U^{-1},V^{-1})$ in the notation used in \eqref{LASTBOUND}).

To prove this claim, assume $\fb_{\vecxi,L}(y)> 2y^{\frac14}$ and let $q_0$ be a positive integer for which the maximum in
\eqref{MYXILDEF} is attained.
Then $\langle q_0\xi_1\rangle<\frac{y^{1/4}}{2Lq_0}<\frac1{2q_0}$; therefore there exist
$q,s\in\Z^+$ such that $q_0=sq$,
$q$ is a denominator of a convergent of the continuous fraction expansion of $\xi_1$ 
and $\langle q_0\xi_1\rangle=s\langle q\xi_1\rangle$.
It follows that \eqref{TRUEMAXREMARKRES} holds for these $q,s$.
Next note that also
$\langle sq\xi_2\rangle<\frac{y^{1/4}}{2sq}<\frac1{2s}$;
therefore there exist $\widetilde{s},k\in\Z^+$ such that 
$s=k\widetilde{s}$, $\widetilde{s}$ is a denominator of a convergent of 
$q\xi_2$,
and $\langle sq\xi_2\rangle=k\langle \widetilde{s}q\xi_2\rangle$;
and now the assumption that the maximum in \eqref{MYXILDEF} is attained at $q_0$ forces $k=1$,
i.e.\ $s$ itself is a denominator of a convergent of $q\xi_2$.
Finally note that since %
\eqref{TRUEMAXREMARKRES} is larger than $2y^{\frac14}$, we have $(sq)^4<\frac14y^{-\frac12}$,
$(\frac1{Lq\langle q\xi_1\rangle})^2>4\frac{s^4}{\sqrt y}$ and
$(\frac1{s\langle sq\xi_2\rangle})^2>4\frac{q^2}{\sqrt y}$,
and these inequalities imply that $q,s$ are in fact the ``special denominators'' introduced above.
\end{remark}

\section{Basic properties of the majorant $\fb_{\vecxi,L}(y)$}
\label{MAJORANTPROPSEC}

In this section we will note some basic properties of the majorant $\fb_{\vecxi,L}(y)$ appearing in the
bound in our Theorems \ref{MAINABTHM} and \ref{MAINTHM}.
This is helpful for clarifying the content of those theorems in certain parameter regimes;
we will also make use of the facts proved here %
in our treatment of general $U^\R$-orbits in the next section. %

\begin{lem}\label{EORDERCHANGELEM}
For any $\vecxi\in\R^2$ and any $L_1,L_2,y_1,y_2>0$ we have
\begin{align*}
\min\bigl(\tfrac{L_2}{L_1},1\bigr)\min\bigl(\tfrac{y_1}{y_2},1\bigr)^{\frac12}\fb_{\vecxi,L_2}(y_2)\leq \fb_{\vecxi,L_1}(y_1)\leq
\max\bigl(\tfrac{L_2}{L_1},1\bigr)\max\bigl(\tfrac{y_1}{y_2},1\bigr)^{\frac12}\fb_{\vecxi,L_2}(y_2).
\end{align*}
\end{lem}
\begin{proof}
This is immediate from the definition, \eqref{MYXILDEF}.
\end{proof}
In particular replacing $L$ and/or $y$ by numbers of the same order of magnitude
does not change the order of magnitude of $\fb_{\vecxi,L}(y)$; we will use this fact several times in 
Section \ref{NONCLOSEDSECTION} %
without explicit mention.

\begin{lem}\label{SHIFTXIFREEDOMLEM}
For any $0<y<1$, $\vecxi\in\R^2$, $L\geq1$, and any integer $n\in[-L,L]$, we have
$\fb_{\vecxi U^n,L}(y)\asymp\fb_{\vecxi,L}(y)$.
\end{lem}
\begin{proof}
For every $q\in\Z^+$,
$\langle q(n\xi_1+\xi_2)\rangle\leq L\langle q\xi_1\rangle+\langle q\xi_2\rangle$; thus
$\min(\frac1{q^2},\frac{\sqrt y}{Lq\langle q\xi_1\rangle},\frac{\sqrt y}{q\langle q(n\xi_1+\xi_2)\rangle})
\geq\frac12\min(\frac1{q^2},\frac{\sqrt y}{Lq\langle q\xi_1\rangle},\frac{\sqrt y}{q\langle q\xi_2\rangle})$.
Therefore $\fb_{\vecxi U^n,L}(y)\geq\frac12\fb_{\vecxi,L}(y)$.
Similarly $\fb_{\vecxi,L}(y)\geq\frac12\fb_{\vecxi U^n,L}(y)$.
\end{proof}

The following lemma is in principle contained in the discussion on the last pages of Section~\ref{NEWBOUNDSEC}.
For clarity, we write out the short proof here.

\begin{lem}\label{EXILBASICLEM1}
Let $0<y<1$, $\vecxi\in\R^2$ and $L\geq1$, and assume that $\fb_{\vecxi,L}(y)>2y^{\frac14}$.
Let $q_0$ be a positive integer where the maximum in \eqref{MYXILDEF} is attained.
Then every positive integer $q$ such that 
$\min(\frac1{q^2},\frac{\sqrt y}{Lq\langle q\xi_1\rangle},\frac{\sqrt y}{q\langle q\xi_2\rangle})>2y^{\frac14}$
must be of the form $q=mq_0$ for some $m\in\Z^+$ which is so small that
$\langle q\xi_1\rangle=m\langle q_0\xi_1\rangle$ and $\langle q\xi_2\rangle=m\langle q_0\xi_2\rangle$
(in particular $q_0$ is uniquely determined).
\end{lem}
\begin{proof}
We assume $\xi_1,\xi_2\notin\Q$; the cases when $\xi_1\in\Q$ or $\xi_2\in\Q$ can then be treated by
a limit argument.   %
Let $q$ be a positive integer satisfying
$\min(\frac1{q^2},\frac{\sqrt y}{Lq\langle q\xi_1\rangle},\frac{\sqrt y}{q\langle q\xi_2\rangle})>2y^{\frac14}$.
Then $\langle q\xi_1\rangle<\frac{y^{1/4}}{2Lq}<\frac1{2q}$ and $\langle q\xi_2\rangle<\frac{y^{1/4}}{2q}<\frac1{2q}$.
Therefore for both $j=1,2$, we have $q=m_jq_j$ for some $m_j,q_j\in\Z^+$ such that $q_j$ is
a denominator of a convergent of the continuous fraction expansion of $\xi_j$,
and $\langle q\xi_j\rangle=m_j\langle q_j\xi_j\rangle$.
Note that $q_j\leq q<y^{-\frac18}$ and also,
if we denote by $q_j'$ the denominator of the ``next'' convergent of $\xi_j$, then
$\frac1{2q_j'}<\langle q_j\xi_j\rangle<\frac{y^{1/4}}2$ and thus $q_j'>y^{-\frac14}>y^{-\frac18}$.
Hence $q_1$ and $q_2$ are uniquely determined for our given $\vecxi,L,y$;
namely, $q_j$ equals the largest number $<y^{-\frac18}$ among 
all the denominators of the convergents of $\xi_j$.
Let $q_0$ be the least common multiple of these two numbers $q_1,q_2$.
It then follows that any number $q$ as above must be of the form $q=mq_0$ for some $m\in\Z^+$
so small that $\langle q\xi_1\rangle=m\langle q_0\xi_1\rangle$ and $\langle q\xi_2\rangle=m\langle q_0\xi_2\rangle$;
thus also $q_0$ is the unique number at which the maximum in \eqref{MYXILDEF} is attained.
\end{proof}

Let us define
\begin{align}\label{BTILDEDEF}
\widetilde \fb_{\vecxi,L}(y)=\fb_{\vecxi,L}(y)+y^{\frac14}.
\end{align}
Note that the obvious analogues of Lemmata \ref{EORDERCHANGELEM} and \ref{SHIFTXIFREEDOMLEM} also hold for $\widetilde\fb$.

\begin{lem}\label{GOODMSUMLEMNEW}
Let $0<\eta<1$. For any $0<y<1$, $\vecxi\in\R^2$, $L\geq1$ we have
\begin{align}\label{GOODMSUMLEMNEWRES}
\sum_{|n|\leq L}\widetilde \fb_{\vecxi U^n,1}(y)^\eta\asymp_\eta L\widetilde \fb_{\vecxi,L}(y)^\eta.
\end{align}
\end{lem}
\begin{proof}
\textbf{Case 1.} Assume $\fb_{\vecxi,L}(y)\geq8y^{\frac14}$.
Let $q$ be the unique positive integer at which the maximum in \eqref{MYXILDEF} %
is attained (cf.\ Lemma \ref{EXILBASICLEM1}).
In particular then $q^2\leq\frac1{8y^{1/4}}$, $\langle q\xi_1\rangle\leq\frac{y^{1/4}}{8Lq}$ and
$\langle q\xi_2\rangle\leq\frac{y^{1/4}}{8q}$.
We claim that, with the same $q$, for every integer $n$ with $|n|\leq L$,
\begin{align}\label{GOODMSUMLEMNEWPF1}
\fb_{\vecxi U^n,1}(y)=\min\Bigl(\frac1{q^2},\frac{\sqrt y}{q\langle q\xi_1\rangle},
\frac{\sqrt y}{q\langle q(n\xi_1+\xi_2)\rangle}\Bigr).
\end{align}
To prove this, note that
\begin{align*}
\langle q(n\xi_1+\xi_2)\rangle\leq |n|\langle q\xi_1\rangle+\langle q\xi_2\rangle
\leq L\langle q\xi_1\rangle+\langle q\xi_2\rangle\leq\frac{y^{1/4}}{4q}.
\end{align*}
Hence the right hand side of \eqref{GOODMSUMLEMNEWPF1} is $\geq4y^{\frac14}$,
and by Lemma \ref{EXILBASICLEM1}, we have $q=mq_0$ for some $m\in\Z^+$,
where $q_0$ is the positive integer for which the maximum for $\fb_{\vecxi U^n,1}(y)$ is attained,
and $\langle q\xi_1\rangle=m\langle q_0\xi_1\rangle$
and $\langle q(n\xi_1+\xi_2)\rangle=m\langle q_0(n\xi_1+\xi_2)\rangle$.
Now
\begin{align*}
\langle q_0\xi_2\rangle=\langle q_0(n\xi_1+\xi_2-n\xi_1)\rangle
\leq \frac1m\langle q(n\xi_1+\xi_2)\rangle+\frac{|n|}m\langle q\xi_1\rangle
<\frac1{4m}+\frac L{8Lm}<\frac1{2m},
\end{align*}
so that $\langle q\xi_2\rangle=m\langle q_0\xi_2\rangle$.
Therefore
$\min(\frac1{q^2},\frac{\sqrt y}{Lq\langle q\xi_1\rangle},\frac{\sqrt y}{q\langle q\xi_2\rangle})
=m^{-2}\min(\frac1{q_0^2},\frac{\sqrt y}{Lq_0\langle q_0\xi_1\rangle},\frac{\sqrt y}{q_0\langle q_0\xi_2\rangle})$,
so that by our choice of $q$, $m=1$ must hold.
Now \eqref{GOODMSUMLEMNEWPF1} is proved.

\textbf{Case 1a: $\langle q\xi_2\rangle\geq2L\langle q\xi_1\rangle$.}
Then $\fb_{\vecxi,L}(y)=\min(\frac1{q^2},\frac{\sqrt y}{q\langle q\xi_2\rangle})$,
and also for each $|n|\leq L$ we have
$\frac12\langle q\xi_2\rangle\leq\langle q(n\xi_1+\xi_2)\rangle\leq\frac32\langle q\xi_2\rangle$
and $\frac23\fb_{\vecxi,L}(y)\leq \fb_{\vecxi U^n,1}(y)\leq2\fb_{\vecxi,L}(y)$.
Hence \eqref{GOODMSUMLEMNEWRES} holds.

\textbf{Case 1b: $\langle q\xi_2\rangle<2L\langle q\xi_1\rangle$.}
Then $\frac12\min(\frac1{q^2},\frac{\sqrt y}{Lq\langle q\xi_1\rangle})\leq
\fb_{\vecxi,L}(y)\leq\min(\frac1{q^2},\frac{\sqrt y}{Lq\langle q\xi_1\rangle})$.
Note also that there is $\omega\in\{\pm1\}$ such that
$\langle q(n\xi_1+\xi_2)\rangle=\bigl|\omega n\langle q\xi_1\rangle+\langle q\xi_2\rangle\bigr|$ for every integer $n$
with $|n|\leq L$; hence our task is to prove:
\begin{align}\label{GOODMSUMLEMNEWPF2}
\sum_{|n|\leq L}\min\Bigl(\frac1{q^2},\frac{\sqrt y}{q\langle q\xi_1\rangle},
\frac{\sqrt y}{q\bigl|n\langle q\xi_1\rangle+\langle q\xi_2\rangle\bigr|}\Bigr)^\eta
\asymp_\eta L\min\Bigl(\frac1{q^2},\frac{\sqrt y}{Lq\langle q\xi_1\rangle}\Bigr)^\eta.
\end{align}
Set 
$A=\max(0,\frac{\langle q\xi_2\rangle}{\langle q\xi_1\rangle}-L)$.
Then $0\leq A<L$, and \eqref{GOODMSUMLEMNEWPF2} is
\begin{align*}
\asymp\sum_{n=1}^L\min\Bigl(\frac1{q^2},\frac{\sqrt y}{(A+n)q\langle q\xi_1\rangle}\Bigr)^\eta
=q^{-2\eta}\sum_{n=1}^L\min\Bigl(1,\frac B{A+n}\Bigr)^\eta,
\end{align*}
with $B=\frac{\sqrt y q}{\langle q\xi_1\rangle}$.
However the last expression is $\asymp_\eta q^{-2\eta} L\min\bigl(1,(B/L)^\eta\bigr)$,
uniformly over all $A\in[0,L]$.
Hence \eqref{GOODMSUMLEMNEWPF2} holds.

\textbf{Case 2:} Assume $\fb_{\vecxi,L}(y)<8y^{\frac14}$.
Then the right hand side of \eqref{GOODMSUMLEMNEWRES} is $\asymp_\eta Ly^{\eta/4}$,
and it now suffices to prove that $\sum_{|n|\leq L}\fb_{\vecxi U^n,1}(y)^\eta\ll_\eta Ly^{\eta/4}$.
For any integer $n$ satisfying $|n|\leq L$ and
$\fb_{\vecxi U^n,1}(y)>16y^{\frac14}$, we can do the following:
Let $L'\geq1$ be the largest number for which $\fb_{\vecxi U^n,L'}(y)\geq16y^{\frac14}$.
Then in fact $L'<L$ and $\fb_{\vecxi U^n,L'}(y)=16y^{\frac14}$, 
since $\fb_{\vecxi U^n,L}(y)\leq2\fb_{\vecxi,L}(y)<16y^{\frac14}$,
by the proof of Lemma \ref{SHIFTXIFREEDOMLEM}.
Furthermore, by what we proved in Case 1,
\begin{align*}
\sum_{n-L'\leq m\leq n+L'}\fb_{\vecxi U^m,1}(y)^\eta\ll_\eta L'\fb_{\vecxi U^n,L'}(y)^\eta\asymp L'y^{\eta/4}.
\end{align*}

It follows that for \textit{every} integer $n$ with $|n|\leq L$, there exist integers
$a(n)\leq n\leq b(n)$ satisfying $b(n)-n=n-a(n)<L$, such that
$\sum_{m=a(n)}^{b(n)}\fb_{\vecxi U^m,1}(y)^\eta\ll_\eta (b(n)-a(n)+1)y^{\eta/4}$.
(Indeed, if $\fb_{\vecxi U^n,1}(y)\leq16y^{\frac14}$ then take $a(n)=b(n)=n$.)

Now fix $F$ to be any subset of $\Z\cap[-L,L]$ which is minimal with the property that
$\cup_{n\in F} [a(n),b(n)]$ contains all $\Z\cap[-L,L]$.
Let us write $F=\{n_1,n_2,\ldots,n_r\}$ where $n_1<n_2<\ldots<n_r$.
Then $a(n_1)<a(n_2)<\ldots<a(n_r)$,
for otherwise, if $a(n_j)\geq a(n_{j+1})$ for some $1\leq j<r$,
then $[a(n_j),b(n_j)]\subset[a(n_{j+1}),b(n_{j+1})]$ and so
$\Z\cap[-L,L]\subset\cup_{n\in F\setminus\{n_j\}}[a(n),b(n)]$, contradicting the minimality of $F$.
Similarly $b(n_1)<b(n_2)<\ldots<b(n_r)$.
Next note that if $1\leq j\leq r-2$ and $b(n_j)\geq a(n_{j+2})$, then
$[a({n_{j+1}}),b({n_{j+1}})]\subset[a(n_j),b(n_j)]\cup[a({n_{j+2}}),b({n_{j+2}})]$,
so that $n_{j+1}$ could be removed from $F$, a contradiction.
Hence, for every $j\in\{1,\ldots,r-2\}$ we have $b(n_j)<a(n_{j+2})$.
It follows that $\sum_{\substack{1\leq j\leq r\\ j\text{ odd}}}(b(n_j)-a(n_j)+1)\leq b(n_r)-a(n_1)+1\ll L$,
and the same bound holds for the sum over all even $j$.
Hence
\begin{align*}
\sum_{|n|\leq L}\fb_{\vecxi U^n,1}(y)^\eta\leq
\sum_{j=1}^r\sum_{m=a(n_j)}^{b(n_j)}\fb_{\vecxi U^m,1}(y)^\eta
\ll_\eta \sum_{j=1}^r(b(n_j)-a(n_j)+1)y^{\eta/4}\ll Ly^{\eta/4}.
\end{align*}
\end{proof}

\begin{remark}\label{NEEDALFBETANEARZEROREM}
For any integer $n$ we have
\begin{align*}
\int_{\alpha}^{\beta}f\bigl(\Gamma(1_2,\vecxi)U^x a(y))\,dx
=\int_{\alpha-n}^{\beta-n}f\bigl(\Gamma(1_2,\vecxi U^n)U^x a(y))\,dx,
\end{align*}
since $U^{-n}\in\Gamma$.
Hence in Theorem \ref{MAINABTHM}, %
\eqref{MAINABTHMRES} holds more generally with the right hand side replaced by
\begin{align*}
C\|f\|_{\C_{\b}^8}\frac{L_n}{\beta-\alpha}\,\widetilde \fb_{\vecxi U^n,L_n}(y)^{1-\ve},
\qquad\text{with }\:
L_n:=\max(1,|\alpha-n|,|\beta-n|),
\end{align*}
where $n$ is an arbitrary integer.   %
It follows from Lemmata \ref{EORDERCHANGELEM} and \ref{SHIFTXIFREEDOMLEM} 
that among the choices of $n$, the best %
bound (to within an absolute constant)
is obtained for any $n$ such that the point $0$ lies within distance $\ll 1+|\beta-\alpha|$
from the interval $[\alpha-n,\beta-n]$.
\end{remark}

\begin{remark}\label{MAYSPLITREMARK}
Assume now that $|\beta-\alpha|$ is large, and that $[\alpha,\beta]$ has distance $\ll|\beta-\alpha|$ from $0$, 
in line with Remark \ref{NEEDALFBETANEARZEROREM}. %
One may then consider partitioning $[\alpha,\beta]$ into subintervals, 
applying Theorem \ref{MAINABTHM} to each of these individually, and then adding the results. %
It follows from Lemma~\ref{GOODMSUMLEMNEW} (and Lemma \ref{EORDERCHANGELEM})
that the resulting error bound is never better (to within an absolute constant) than the original one, %
\eqref{MAINABTHMRES}; 
and if each subinterval has length $\gg1$ then the resulting error bound is in fact 
\textit{equally good} as \eqref{MAINABTHMRES}.
\end{remark}

\section{General orbits}
\label{NONCLOSEDSECTION}

We will now prove the effective equidistribution result for arbitrary $U^\R$-orbits in $X$, Theorem~\ref{NONCLOSEDTHM2}.
The proof uses the technique of approximating nonclosed horocycles in $X'$ by pieces of closed horocycles,
in the precise form which was worked out in Sarnak and Ubis \cite[Sec.~2]{pSaU2011}.
We fix a left invariant Riemannian metric $d_G$ on $G$.
Recall that $G'=\SL(2,\R)\subset G$ and $\Gamma'=\SL(2,\Z)$.
\begin{prop}\label{SUAPPRLEM} 
[Sarnak and Ubis \cite{pSaU2011}.]
There is an absolute constant $C_1>0$ such that the following holds.
For every $M\in G'$ and $T\geq2$ there is some $\gamma=\gamma_{M,T}\in\Gamma'$ and numbers
$\alpha=\alpha_{M,T}\in\R$, $y=y_{M,T}>0$, $W=W_{M,T}\in\R$ and $\omega=\omega_{M,T}\in\{1,-1\}$ such that
\begin{align}\label{SUAPPRLEMRES1} 
\frac1{C_1y}\leq T\leq C_1|W|
\end{align}
and such that, writing $\ell(t)\equiv t$ if $\omega=1$ and $\ell(t)\equiv T-t$ if $\omega=-1$:
\begin{align}\label{SUAPPRLEMRES2} 
d_G\Bigl(\gamma^{-1} MU^{\ell(t)}\: , \: U^{\alpha+\frac{yW}{1- \omega  t/W}} %
a\bigl({\textstyle \frac{y}{(1- \omega  t/W)^2}}\bigr)\Bigr)
\ll\frac{|W|^{-1}}{|1- \omega  t/W|},\qquad\forall t\in[0,T],
\end{align}
and
\begin{align}\label{SUAPPRLEMRES3} 
-\tfrac12<\Re\bigl(\gamma^{-1} MU^{\ell(0)}(i)\bigr)\leq\tfrac12.
\end{align}
\end{prop}
\begin{proof}
This is proved in \cite[Sec.\ 2]{pSaU2011}.
(Note that the restriction of $d_G$ to $G'$ is a left invariant Riemannian metric on $G'$.
Note also that once \eqref{SUAPPRLEMRES1} and \eqref{SUAPPRLEMRES2} hold,
we can make also \eqref{SUAPPRLEMRES3} hold by replacing $\gamma$ by $\gamma U^n$ and $\alpha$ by $\alpha-n$,
for an appropriate $n\in\Z$.)
\end{proof}

Now for any $M\in G'$ and $T\geq2$ we have defined both $y_{M,T}$ (in Proposition \ref{SUAPPRLEM})
and $y_M(T)$ (in \eqref{YGTDEF}). These are in fact of the same order of magnitude:
\begin{lem}\label{TWOYSAMELEM}
$y_{M,T}\asymp y_M(T)$, uniformly over all $M\in G'$ and $T\geq2$.
\end{lem}
\begin{proof}
Let $\gamma=\gamma_{M,T}$, $\alpha=\alpha_{M,T}$, $y=y_{M,T}$ and $W=W_{M,T}$ be as in Proposition \ref{SUAPPRLEM}.
By \eqref{SUAPPRLEMRES2}, $d_G(\gamma^{-1} MU^{\ell(0)},U^{\alpha+yW}a(y))\ll|W|^{-1}$,
and therefore $d_G(\gamma^{-1} MU^{\ell(0)}a(T),U^{\alpha+yW}a(y)a(T))\ll|W|^{-1}T\leq C_1$.
Note here that $U^{\ell(0)}a(T)$ equals either $a(T)$ or $U^Ta(T)=a(T)U^1$.
It follows that, if we set $g=\gamma^{-1} Ma(T)$ and consider the standard action of 
$G'$ on the Poincar\'e upper half plane model of the hyperbolic plane,
then $\Im g(i)\asymp\Im U^{\alpha+yW}a(y)a(T)(i)=yT\geq C_1^{-1}$.
Hence the invariant height function used in \cite{iha} satisfies
$\scrY_{\Gamma'}(g)\ll\Im(g(i))\leq\scrY_{\Gamma'}(g)$ and therefore $\scrY_{\Gamma'}(g)\asymp yT$.
The lemma follows from this, since $y_M(T)=T^{-1}\svl(g)^{-2}=T^{-1}\scrY_{\Gamma'}(g)$.
\end{proof}

Using Theorem \ref{MAINABTHM} and Proposition \ref{SUAPPRLEM}  we will now prove:
\begin{thm}\label{NONCLOSEDTHM}
Let $\ve>0$ be fixed.
For any $\vecxi\in\R^2$, $M\in G'$, $T\geq2$, $f\in\C_{\b}^8(\GaG)$,
and for any $y=y_{M,T}$ and $\gamma=\gamma_{M,T}$ as in Proposition \ref{SUAPPRLEM}, we have
\begin{align}\label{NONCLOSEDTHMRES}
T^{-1}\int_0^T f(\Gamma(1_2,\vecxi)MU^t)\,dt=\int_{\GaG} f\,d\mu
+O_\ve\Bigl(\|f\|_{\C_{\b}^8}\,\widetilde \fb_{\vecxi\gamma,yT}(y)^{\frac12-\ve}\Bigr).
\end{align}
Here $\widetilde \fb_{\vecxi\gamma,yT}(y):=\fb_{\vecxi\gamma,yT}(y)+y^{\frac14}$ 
as in \eqref{BTILDEDEF}.
\end{thm}
We will see below
that Theorem~\ref{NONCLOSEDTHM} implies Theorem \ref{NONCLOSEDTHM2}.

\begin{proof}[Proof of Theorem \ref{NONCLOSEDTHM}]
Let $\vecxi,M,T,y,\gamma,f$ be as in the statement of the theorem;
also fix corresponding numbers $\alpha=\alpha_{M,T}$, $W=W_{M,T}$, $\omega=\omega_{M,T}$ as in Proposition \ref{SUAPPRLEM},
and set $\ell(t)\equiv t$ if $\omega=1$, $\ell(t)\equiv T-t$ if $\omega=-1$.
Note that \eqref{NONCLOSEDTHMRES} is trivial when $y\geq1$ (since then $\widetilde \fb_{\vecxi\gamma,yT}(y)>1$);
hence from now on we will assume $y<1$.
We will partition the interval $[0,T]$ into smaller intervals $I_0,I_1,\ldots,I_m$,
in a way which we make precise below.
Using $\gamma^{-1}(1_2,\vecxi)M=(1_2,\vecxi\gamma)\gamma^{-1} M$ we have
\begin{align}\label{NONCLOSEDSTEP1}
\int_0^T f\bigl(\Gamma (1_2,\vecxi)MU^t\bigr)\,dt
=\sum_{j=0}^{m}\int_{I_j}f\bigl(\Gamma (1_2,\vecxi\gamma)\gamma^{-1} MU^{\ell(t)}\bigr)\,dt.
\end{align}
For each $j$ we set
$\rho_j^{\max}=\sup_{t\in I_j}|1-\omega t/W|$,
$\rho_j^{\min}=\inf_{t\in I_j}|1-\omega t/W|$.
We also set $\tau_j=|I_j|$, the length of the interval $I_j$.
Our partition will be such that %
$I_0$ contains those $t$ for which $1-\omega t/W$ are closest to $0$;
in particular we will have $\rho_j^{\min}>0$ for all $j\geq1$.
Using \eqref{SUAPPRLEMRES2} together with 
$|f(\Gamma g_1)-f(\Gamma g_2)|\ll\|f\|_{\C_{\b}^1}d_G(g_1,g_2)$ ($\forall g_1,g_2\in G$)
and the fact that $d_G$ is left invariant,
we have, for each $j\geq1$,
\begin{align}\notag
\int_{I_j}f\bigl(\Gamma (1_2,\vecxi\gamma)\gamma^{-1} MU^{\ell(t)}\bigr)\,dt
=\int_{I_j}f\Bigl(\Gamma(1_2,\vecxi\gamma)U^{\alpha+\frac{yW}{1-\omega  t/W}}
a\bigl({\textstyle \frac{y}{(1-\omega  t/W)^2}}\bigr)\Bigr)\,dt
\hspace{40pt}
\\\label{NONCLOSEDSTEP2}
+O\biggl(\|f\|_{\C_{\b}^1}\frac{\tau_j}{\rho_j^{\min}|W|}\biggr).
\end{align}
We set $y_j^*=y/(\rho_j^{\min})^2$.
Note that $a(\frac{y}{(1-\omega  t/W)^2})=a(y_j^*)a\bigl(\frac{(1-\omega  t/W)^2}{(\rho_j^{\min})^2}\bigr)^{-1}$ and
$1\leq \frac{(1-\omega  t/W)^2}{(\rho_j^{\min})^2}\leq1+\frac{2\rho_j^{\max}\tau_j}{(\rho_j^{\min})^2|W|}$
for all $t\in I_j$; therefore
$d_G\bigl(a(\frac{y}{(1-\omega  t/W)^2}),a(y_j^*)\bigr)\ll\frac{\rho_j^{\max}\tau_j}{(\rho_j^{\min})^2|W|}$
for all $t\in I_j$.
We will choose the intervals $I_0,\ldots,I_m$ so that $\rho_j^{\max}\leq2\rho_j^{\min}$ for each $j\geq1$.
Hence we may replace $a(\frac{y}{(1-\omega  t/W)^2})$ by $a(y_j^*)$ in the integral in
\eqref {NONCLOSEDSTEP2}, without changing the error term.
Next we take $s=\alpha+\frac{yW}{1-\omega  t/W}$ as a new variable of integration.
Let $S_j\subset\R$ be the $s$-interval which corresponds to $I_j$.
Note that $|S_j|=\frac{y\tau_j}{\rho_j^{\min}\rho_j^{\max}}$,
and 
$\frac{dt}{ds}=\omega y^{-1}(1-\omega t/W)^2=\omega\frac{\rho_j^{\min}\rho_j^{\max}}y+O(\frac{\rho_j^{\max}\tau_j}{y|W|})$
for $t\in I_j$. 
Hence \eqref{NONCLOSEDSTEP2} equals
\begin{align}\label{NONCLOSEDSTEP2a}
\frac{\rho_j^{\min}\rho_j^{\max}}y\int_{S_j}f\bigl(\Gamma(1_2,\vecxi\gamma)U^s\, a(y_j^*)\bigr)\,ds
+O\biggl(\|f\|_{\C_{\b}^1}\frac{1+\tau_j^2}{\rho_j^{\min}|W|}\biggr).
\end{align}
We will choose the intervals $I_0,\ldots,I_m$ so that $y_j^*<1$ for each $j\geq1$.
Take $n_j\in\Z$ so that $S_j-n_j$ intersects the interval $[0,1)$, and set $\gamma_j:=\gamma U^{n_j}$.
Applying Theorem \ref{MAINABTHM} together with Remark \ref{NEEDALFBETANEARZEROREM} (with $n=n_j$), we conclude that 
\begin{align}\label{NONCLOSEDSTEP5}
\int_{I_j}f\bigl(\Gamma (1_2,\vecxi\gamma)\gamma^{-1} MU^{\ell(t)}\bigr)\,dt
=\tau_j\int_{\GaG}f\,d\mu+O_\ve\biggl(\|f\|_{\C_{\b}^8}\frac{L_j}{y_j^*}\widetilde \fb_{\vecxi\gamma_j,L_j}(y_j^*)^{1-\ve}
+\|f\|_{\C_{\b}^1}\frac{1+\tau_j^2}{\rho_j^{\min}|W|}\biggr),
\end{align}
where $L_j=1+|S_j|$.
We have $|S_j|=\frac{y\tau_j}{\rho_j^{\min}\rho_j^{\max}}\asymp y_j^*\tau_j$, 
and we will choose $I_0,\ldots,I_m$ in such a way that 
$y_j^*\tau_j\ll1$ for all $j\geq1$; hence \eqref{NONCLOSEDSTEP5} holds with $L_j$ replaced by $1$.
We now wish to choose $I_0,\ldots,I_m$ in such a way that for each $j\geq1$, $\tau_j$ takes a value 
which essentially minimizes $\tau_j^{-1}$ times the error term in \eqref{NONCLOSEDSTEP5},
but subject to $y_j^*\tau_j\ll1$.

The precise choice of $I_0,\ldots,I_m$ is made according to the following algorithm.
Let the absolute constant $C_1>0$ be as in Proposition \ref{SUAPPRLEM}, and set $C_2=\frac1{2(1+C_1)^3}$.
\\[5pt]
1. Set $j=1$ and $T_1=0$.
\\[3pt]
2. If $1-\omega T_j/W>2y^{\frac14}$ then set $\rho_j=|1-\omega T_j/W|$,
$n_j=\lfloor\alpha+\frac{yW}{1-\omega T_j/W}\rfloor\in\Z$ and $\gamma_j=\gamma U^{n_j}$, and go to Step 3;
otherwise change the value of $T_j$ to $T_j=T$ and go to Step 4.
\\[3pt]
3. Set
${\displaystyle\tau_j=\min\Bigl(\rho_j^{\frac32}y^{-\frac12}|W|^{\frac12}\widetilde \fb_{\vecxi\gamma_j,1}(y/\rho_j^2)^{\frac12}
,C_2\rho_j^2y^{-1},T-T_j\Bigr)}$, 
$T_{j+1}=T_j+\tau_j$ and $I_j=[T_j,T_{j+1}]$. 
If $T_{j+1}=T$, set $m=j$ and $I_0=\emptyset$, and we are \textit{done}; otherwise replace $j$ by $j+1$ and go back to Step 2.
\\[3pt]
4. If $1-\omega T_j/W<-2y^{\frac14}$ then set
$\rho_j=|1-\omega T_j/W|$, $n_j=\lfloor\alpha+\frac{yW}{1-\omega T_j/W}\rfloor\in\Z$
and $\gamma_j=\gamma U^{n_j}$, and go to Step 5; otherwise 
set $m=j-1$ and $I_0=[0,T]\setminus\cup_{i=1}^{m}I_i$ (this is an interval), and we are \textit{done.}
\\[3pt]
5. Set
${\displaystyle\tau_j=\min\Bigl(\rho_j^{\frac32}y^{-\frac12}|W|^{\frac12}\widetilde \fb_{\vecxi\gamma_j,1}(y/\rho_j^2)^{\frac12}
,C_2\rho_j^2y^{-1}\Bigr)}$, 
$T_{j+1}=T_j-\tau_j$ and $I_j=[T_{j+1},T_j]$. 
Then replace $j$ by $j+1$ and go back to Step 4.

\vspace{10pt}

Note that $|1-\omega t/W|\leq 1+T/|W|\leq1+C_1$ for all $t\in[0,T]$;
hence we always get $2y^{\frac14}<\rho_j\leq1+C_1$ in Steps 2 and 4.
Using this and \eqref{SUAPPRLEMRES1}  we see that each time we set $\tau_j$ and $I_j$ in Steps 3 and 5, we get
$\tau_j\leq C_2\rho_j^2y^{-1}\leq C_2C_1^2(1+C_1)\rho_j|W|<\frac12\rho_j|W|$, and therefore
$|1-\omega t/W-\rho_j|<\frac12\rho_j$ for all $t\in I_j$.
Hence for each such interval $I_j$ we have $\rho_j^{\min}>0$ and $\rho_j^{\max}<2\rho_j^{\min}$,
and also $y_j^*=y/(\rho_j^{\min})^2<4y/\rho_j^2<y^{\frac12}<1$.
It also follows that for any interval $I_j$ obtained in Step 3 (resp. Step 5)
we have $1-\omega t/W>y^{\frac14}$ (resp.\ $1-\omega t/W<-y^{\frac14}$) for all $t\in I_j$;
therefore the intervals constructed in Steps 2--3 do not overlap with those constructed in Steps 4--5.
Hence the resulting $I_0,I_1,\ldots,I_m$ indeed form a partition of $[0,T]$
(after possibly removing one or both endpoints from some of the $I_j$'s), %
satisfying all the conditions specified earlier.

For each $j\geq1$ we have, because of the choice of $\tau_j$ in Steps 3 and 5,
\begin{align*}
\frac{1+\tau_j^2}{\rho_j^{\min}|W|}
\ll(\rho_j^{\min})^{-1}+(\rho_j^{\min})^2y^{-1}\widetilde\fb_{\vecxi\gamma_j,1}(y_j^*)
\ll {y_j^*}^{-\frac12}+{y_j^*}^{-1}\widetilde\fb_{\vecxi\gamma_j,1}(y_j^*)
\ll{y_j^*}^{-1}\widetilde\fb_{\vecxi\gamma_j,1}(y_j^*).
\end{align*}
Hence by \eqref{NONCLOSEDSTEP1} and \eqref{NONCLOSEDSTEP5} (with $L_j=1$),
we have (possibly with $m=0$):
\begin{align*}
\int_0^T f\bigl(\Gamma (1_2,\vecxi)MU^t\bigr)\,dt
=\int_{I_0} f(\cdots)\,dt
+\Bigl(\sum_{j=1}^m\tau_j\Bigr)\int_{\GaG}f\,d\mu
+O_\ve\biggl(\|f\|_{\C_{\b}^8}\sum_{j=1}^m {y_j^*}^{-1}\,\widetilde \fb_{\vecxi\gamma_j,1}(y_j^*)^{1-\ve}\biggr).
\end{align*}
Next we note that $\int_{I_0} f\,dt=\tau_0\int_{\GaG}f\,d\mu+O(\|f\|_{\C_{\b}^0}\tau_0)$
and $\tau_0\ll y^{\frac14}T$.
(Indeed, $\tau_0\leq T$; also
if $y^{\frac14}<\frac14$, say, and $I_0\neq\emptyset$, then it follows from our construction that
$|1-\omega t/W|\leq2y^{\frac14}<\frac12$ for all $t\in I_0$;
hence $|W|<2T$ and $\tau_0=|I_0|\leq4y^{\frac14}|W|<8y^{\frac14}T$.)
Therefore,
\begin{align}\label{NONCLOSEDSTEP6}
\int_0^T f\bigl(\Gamma (1_2,\vecxi)MU^t\bigr)\,dt
=T\int_{\GaG}f\,d\mu
+O_\ve\Bigl(\|f\|_{\C_{\b}^8}\Bigr)\Bigl\{Ty^{\frac14}+\sum_{j=1}^m {y_j^*}^{-1}\,\widetilde \fb_{\vecxi\gamma_j,1}(y_j^*)^{1-\ve}
\Bigr\}.
\end{align}

Now for each $n\in\Z$, let $J_n$ be the set of those $j\in\{1,\ldots,m\}$ for which $n_j=n$.
By our choice of $n_j$, for each $j\in J_n$ there is some $t\in I_j$ such that
$\frac{yW}{1-\omega t/W}\in[n-\alpha,n-\alpha+1)$.
If $|n-\alpha|\geq2$ then this forces $|1-\omega t/W|\asymp\frac{y|W|}{|n-\alpha|}$;
on the other hand if $|n-\alpha|<2$ then both $|1-\omega t/W|\asymp1$ and $y|W|\asymp1$,
since $y|W|\geq C_1^{-2}$ and $|1-\omega t/W|\leq1+C_1$ by Prop.\ \ref{SUAPPRLEM}.
It follows that $\rho_j^{\min}\asymp\rho_j^{\max}\asymp\rho(n):=\frac{y|W|}{1+|n-\alpha|}$ for each $j\in J_n$,
and thus also $y_j^*\asymp y/\rho(n)^2$ and $|S_j|\asymp y\tau_j/\rho(n)^2$.
But $\sum_{j\in J_n}|S_j|=|\cup_{j\in J_n}S_j|\ll1$, 
since for each $j\in J_n$ we have $|S_j|\ll1$ (by our choice of $\tau_j$) and 
$S_j\cap [n,n+1)\neq\emptyset$.
Hence $\sum_{j\in J_n}\tau_j\ll\rho(n)^2y^{-1}$.
However for all except at most one $j\in J_n$ (the possible exception being $j=m$)
we have
$\tau_j\gg\min(\rho(n)^{\frac32}y^{-\frac12}|W|^{\frac12}\widetilde \fb_{\vecxi\gamma U^n,1}(y/\rho(n)^2)^{\frac12},\rho(n)^2/y)$.
Hence
\begin{align}\label{NONCLOSEDSTEP8}
\# J_n\ll 1+\rho(n)^{\frac12}y^{-\frac12}|W|^{-\frac12}\,\widetilde \fb_{\vecxi\gamma U^n,1}\bigl(y/\rho(n)^2\bigr)^{-\frac12}
\ll \widetilde \fb_{\vecxi\gamma U^n,1}\bigl(y/\rho(n)^2\bigr)^{-\frac12},
\end{align}
and thus in \eqref{NONCLOSEDSTEP6} we have
\begin{align}\notag
\sum_{j=1}^m {y_j^*}^{-1}\,\widetilde \fb_{\vecxi\gamma_j,1}(y_j^*)^{1-\ve}
\ll y^{-1}\sum_{\substack{n\in\Z\\(J_n\neq\emptyset)}}
\rho(n)^2\,\widetilde \fb_{\vecxi\gamma U^n,1}\bigl(y/\rho(n)^2\bigr)^{\frac12-\ve}
\hspace{50pt}
\\\label{NONCLOSEDSTEP7}
\ll y^{-1}\sum_{\substack{n\in\Z\\(J_n\neq\emptyset)}}
\rho(n)^{\frac32}\,\widetilde \fb_{\vecxi\gamma U^n,1}(y)^{\frac12-\ve},
\end{align}
where we used Lemma \ref{EORDERCHANGELEM} and the fact that $\rho(n)\ll1$ for all $n$.

Let us first assume $|W|\geq2T$. 
Then for every $n$ with $J_n\neq\emptyset$ we have $\rho(n)\asymp1$, and there is some $t\in[0,T]$ such that
$n=\alpha+\frac{yW}{1-\omega t/W}+O(1)$; thus $n=\alpha+yW+O(yT)=O(yT)$,
since $|\alpha+yW|\ll1$ by \eqref{SUAPPRLEMRES3} and \eqref{SUAPPRLEMRES2}.
Hence by Lemma \ref{GOODMSUMLEMNEW}, \eqref{NONCLOSEDSTEP7} is
$\ll T\widetilde \fb_{\vecxi\gamma,yT}(y)^{\frac12-\ve}$.
Next assume instead $|W|<2T$. Then $|W|\asymp T$, by \eqref{SUAPPRLEMRES1}.
Given $\widetilde\rho\in(0,1]$, note that for every $n$ with $\rho(n)=\frac{y|W|}{1+|n-\alpha|}\geq\widetilde\rho$ 
we have $|n-\alpha|<\widetilde\rho^{-1}y|W|\asymp\widetilde\rho^{-1}yT$,
and since $|\alpha+yW|\ll1$ this implies $|n|\ll \widetilde\rho^{-1}yT$.
By Lemma \ref{GOODMSUMLEMNEW},
the sum of $\widetilde \fb_{\vecxi\gamma U^n,1}(y)^{\frac12-\ve}$ over all these $n$ is
$\ll \widetilde\rho^{-1}yT\widetilde \fb_{\vecxi\gamma,yT}(y)^{\frac12-\ve}$.
Hence the contribution from all $n$ with $\rho(n)\geq\frac12$ in \eqref{NONCLOSEDSTEP7}
is $\ll T\widetilde \fb_{\vecxi\gamma,yT}(y)^{\frac12-\ve}$,
and for each $k\in\Z^+$, the contribution from all $n$ with 
$\rho(n)\in[2^{-k-1},2^{-k})$ in \eqref{NONCLOSEDSTEP7}
is $\ll 2^{-\frac12k}T\widetilde \fb_{\vecxi\gamma,yT}(y)^{\frac12-\ve}$.
Adding over $k$ we again conclude that \eqref{NONCLOSEDSTEP7} is
$\ll T\widetilde \fb_{\vecxi\gamma,yT}(y)^{\frac12-\ve}$.
In view of \eqref{NONCLOSEDSTEP6}, this completes the proof of Theorem \ref{NONCLOSEDTHM}.
\end{proof}

We remark that the last step in \eqref{NONCLOSEDSTEP8} is in general wasteful, but leads to a simple result.
Working instead with the first bound in \eqref{NONCLOSEDSTEP8} one obtains a variant of Theorem \ref{NONCLOSEDTHM} with a
more complicated but generally better error term:
\begin{thmbis}{NONCLOSEDTHM}
Let $\ve>0$ be fixed.
For any $\vecxi\in\R^2$, $M\in G'$, $T\geq2$, $f\in\C_{\b}^8(\GaG)$,
and for any $y=y_{M,T}$ and $\gamma=\gamma_{M,T}$ as in Proposition \ref{SUAPPRLEM}, we have
\begin{align*}
T^{-1}\int_0^T f(\Gamma(1_2,\vecxi)MU^t)\,dt=\int_{\GaG} f\,d\mu
+O_\ve\Bigl(\|f\|_{\C_{\b}^8}\widetilde \fb_{\vecxi\gamma,yT}(y)^{\frac12-\ve}\Bigl(
\widetilde \fb_{\vecxi\gamma,yT}(y)^{\frac12}+\bigl(y|W|\bigr)^{-\frac12}\Bigr)\Bigr).
\end{align*}
\end{thmbis}
It is worth noticing that in the special case $M=U^{\alpha_0} a(y_0)$ 
with $-\frac12<\alpha_0\leq\frac12$, $y_0$ small
and $T>C_1^{-1}y_0^{-1}$, we may take $\gamma_{M,T}=1_2$ in Proposition \ref{SUAPPRLEM},
as well as $\alpha_{M,T}=\alpha_0-y_0W_{M,T}$ and
$W_{M,T}$ ``very large'' (that is, let $W_{M,T}\to\infty$ for our fixed $M,T$, so that 
\eqref{SUAPPRLEMRES2} turns into an equality between two points in $G$).
In this case, one may expect from the method of proof that Theorem \ref{NONCLOSEDTHM} should recover the statement
of Theorem \ref{MAINABTHM}, with $y=y_0$, $\alpha=\alpha_0$, $\beta=\beta_0+y_0T$.
This is indeed seen to be the case when we use the more precise error term of Theorem \ref*{NONCLOSEDTHM}$'$.
In this vein recall also Remarks \ref{NEEDALFBETANEARZEROREM}, \ref{MAYSPLITREMARK}.
In the case of $\beta-\alpha$ becoming \textit{small} as $y\to0$ in Theorem \ref{MAINABTHM},
we expect that Theorem \ref*{NONCLOSEDTHM}$'$ should typically result in a \textit{better} error term than that of
Theorem \ref{MAINABTHM}.

\vspace{5pt}

Next we will reinterprete the error term in Theorem~\ref{NONCLOSEDTHM}
and thereby deduce Theorem \ref{NONCLOSEDTHM2}.
Recall $\fR_L=[-L^{-1},L^{-1}]\times[-1,1]\subset\R^2$.
Let us first note that, for any $\vecxi\in\R^2$, $L>0$, $y>0$, 
\begin{align}\label{EXILALTFORMULA}
\fb_{\vecxi,L}(y)=\inf\Bigl\{\delta>0\col
\Bigl[\forall q\in\Z_{\leq\delta^{-1/2}}^+:\:
(q^{-1}\Z^2+\vecxi)\cap\frac{\sqrt y}{\delta q^2} \fR_L %
=\emptyset\Bigr]\Bigr\}.
\end{align}
Indeed, from the definition \eqref{MYXILDEF} we see that, 
given any $\delta>0$ we have $\fb_{\vecxi,L}(y)\geq\delta$ if and only if
there is some $q\in\Z^+$ such that $q\leq\delta^{-1/2}$,
$\langle q\xi_1\rangle\leq\frac{\sqrt y}{\delta Lq}$ and
$\langle q\xi_2\rangle\leq\frac{\sqrt y}{\delta q}$;
and the last two conditions hold if and only if 
$(\Z^2+q\vecxi)\cap\frac{\sqrt y}{\delta q}\fR_L\neq\emptyset$.

Using $\frac{\sqrt y}{\delta q^2}\fR_L=(\frac{1}{\delta q^2}\fR_{L/y})a(y)^{-1}$,
the formula \eqref{EXILALTFORMULA} may also be expressed as
\begin{align}\label{EXILALTFORMULA2}
\fb_{\vecxi,L}(y)=\fb_{(1_2,\vecxi)a(y)}(L/y),
\end{align}
where in the right hand side we use the notation introduced in \eqref{BGTDEF}.

\begin{proof}[Proof of Theorem \ref{NONCLOSEDTHM2}]
Let $g,T,f$ be as in the statement of Theorem \ref{NONCLOSEDTHM2}.
Write $g=(1_2,\vecxi)M$; 
fix corresponding numbers 
$y=y_{M,T}$, $\alpha=\alpha_{M,T}$, $W=W_{M,T}$, $\omega=\omega_{M,T}$ 
and $\gamma=\gamma_{M,T}\in\Gamma'$  as in Proposition \ref{SUAPPRLEM},
and set $\ell(t)\equiv t$ if $\omega=1$, $\ell(t)\equiv T-t$ if $\omega=-1$.
By \eqref{SUAPPRLEMRES2} we have $\gamma^{-1} MU^{\ell(0)}=U^{\alpha+yW}a(y)\eta$
for some $\eta\in G'$ in a $O(|W|^{-1})$-neighbourhood of $1_2$.
Hence for any $q\in\Z^+$,
\begin{align*}
(q^{-1}\Z^2)g
=(q^{-1}\Z^2+\vecxi)\gamma\gamma^{-1} M
=(q^{-1}\Z^2+\vecxi\gamma)U^{\alpha+yW}a(y)\eta U^{-\ell(0)}.
\end{align*}
Now assume that, for some $q\in\Z^+$ and $\delta>0$,
the lattice translate $q^{-1}\Z^2+\vecxi\gamma$ contains a point $(x_1,x_2)\in\frac{\sqrt y}{\delta q^2} \fR_{yT}$.
Then $(q^{-1}\Z^2)g$ contains the point
\begin{align*}
(x_1,x_2)U^{\alpha+yW}a(y)\eta U^{-\ell(0)}
=\Bigl(y^{1/2}x_1,y^{-1/2}((\alpha+yW)x_1+x_2)\Bigr)\eta U^{-\ell(0)}.
\end{align*}
But here $|\alpha+yW|\ll1$ by \eqref{SUAPPRLEMRES2}, \eqref{SUAPPRLEMRES3},
and $yT\gg1$ by \eqref{SUAPPRLEMRES1};
hence $|x_1|\leq\frac{1}{\delta q^2\sqrt yT}\ll\frac{\sqrt y}{\delta q^2}$, and the above point is
\begin{align*}
=\Bigl(O\Bigl(\frac1{\delta q^2T}\Bigr),O\Bigl(\frac1{\delta q^2}\Bigr)\Bigr)
\matr{1+O(|W|^{-1})}{O(|W|^{-1})}{O(|W|^{-1})}{1+O(|W|^{-1})}
\matr1{O(T)}01 
=\Bigl(O\Bigl(\frac1{\delta q^2T}\Bigr),O\Bigl(\frac1{\delta q^2}\Bigr)\Bigr),
\end{align*}
where we also used the fact that $|W|\gg T$.
We have thus proved that 
there is an absolute constant $C_2>1$ such that, %
for any $q\in\Z^+$ and $\delta>0$ for which
$(q^{-1}\Z^2+\vecxi\gamma)\cap\frac{\sqrt y}{\delta q^2} \fR_{yT}\neq\emptyset$,
we have
$(q^{-1}\Z^2)g\cap \frac{C_2}{\delta q^2}\fR_T\neq\emptyset$.
Hence by \eqref{EXILALTFORMULA},
\begin{align*}
\fb_{\vecxi\gamma,yT}(y) &=
\inf\Bigl\{\delta>0\col\Bigl[\forall q\in\Z_{\leq\delta^{-1/2}}^+:\:
(q^{-1}\Z^2+\vecxi\gamma)\cap\frac{\sqrt y}{\delta q^2} \fR_{yT}=\emptyset\Bigr]\Bigr\}
\\
&\leq
\inf\Bigl\{\delta>0\col\Bigl[\forall q\in\Z_{\leq\delta^{-1/2}}^+:\:
(q^{-1}\Z^2)g\cap \frac{C_2}{\delta q^2}\fR_T=\emptyset\Bigr]\Bigr\}
\\
&\leq C_2\inf\Bigl\{\delta>0\col\Bigl[\forall q\in\Z_{\leq\delta^{-1/2}}^+:\:
(q^{-1}\Z^2)g\cap \frac{1}{\delta q^2}\fR_T=\emptyset\Bigr]\Bigr\}
=C_2\fb_g(T).
\end{align*}
Using this bound together with
$\widetilde \fb_{\vecxi\gamma,yT}(y)=y^{\frac14}+\fb_{\vecxi\gamma,yT}(y)$ and
$y=y_{M,T}\ll y_M(T)=y_g(T)$ (cf.\ Lemma \ref{TWOYSAMELEM}),
we see that $\widetilde \fb_{\vecxi\gamma,yT}(y)\ll y_g(T)^{\frac14}+\fb_g(T)$,
so that Theorem \ref{NONCLOSEDTHM2} follows from Theorem~\ref{NONCLOSEDTHM}.
\end{proof}

Finally let us prove that, generically, the error term in Theorem \ref{NONCLOSEDTHM2}
decays like $T^{-\frac18+\ve}$ as $T\to\infty$.

\begin{prop}\label{BGTGENDECAYPROP}
Let $0<\alpha<\frac12$ and $M\in G'$ be given.
Then for Lebesgue almost all $\vecxi\in\R^2$, there is some $C>0$ such that
$b_{(M,\vecxi)}(T)<CT^{-\alpha}$ for all $T\geq1$.
\end{prop}
\begin{proof}
It follows from the definition, \eqref{BGTDEF}, that for given $(M,\vecxi)\in G$ and $C\geq1$, the inequality
$b_{(M,\vecxi)}(T)<CT^{-\alpha}$ holds for all $T\geq1$ if and only if,
for every $q\in\Z^+$, the lattice translate $\vecxi+q^{-1}\Z^2M$ is disjoint from the set
\begin{align*}
B_{Cq^2}=\bigcup_{T\geq(Cq^2)^{1/\alpha}}\frac{T^\alpha}{Cq^2}\scrR_T
=\Bigl\{(x_1,x_2)\in\R^2\col |x_1|\leq(Cq^2)^{-\frac1\alpha}\min\bigl(1,|x_2|^{1-\frac1\alpha}\bigr)\Bigr\}.
\end{align*}
For given $M\in G'$ we write $L=\Z^2M$, so that the lattice translate in question is $\vecxi+q^{-1}L$.
Note that this point set %
only depends on the congruence class of $\vecxi\mod q^{-1}L$.
Now %
\begin{align*}
&\int_{\R^2/L} I\bigl((\vecxi+q^{-1}L)\cap B_{Cq^2}\neq\emptyset\bigr)\,d\vecxi
=q^{-2}\int_{\R^2/qL} I\bigl((\veceta+L)\cap qB_{Cq^2}\neq\emptyset\bigr)\,d\veceta
\\
&=\int_{\R^2/L} I\bigl((\veceta+L)\cap qB_{Cq^2}\neq\emptyset\bigr)\,d\veceta
\leq\int_{\R^2/L}\sum_{\vecm\in L} I\bigl(\veceta+\vecm\in qB_{Cq^2}\bigr)\,d\veceta
=\bigl|qB_{Cq^2}\bigr|=q^2\bigl|B_{Cq^2}\bigr|,
\end{align*}
where we substituted $\vecxi=q^{-1}\veceta$, and where $|\cdot|$ denotes Lebesgue measure on $\R^2$.
Next note that, since $1-\frac1\alpha<-1$, we have $\bigl|B_{Cq^2}\bigr|=K(Cq^2)^{-\frac1\alpha}$ where $K>0$ is a
constant which only depends on $\alpha$.
It follows that 
\begin{align*}
\int_{\R^2/L}I\Bigl(\exists T\geq1\col b_{(M,\vecxi)}(T)\geq CT^{-\alpha}\Bigr)\,d\vecxi
\leq\sum_{q=1}^\infty q^2\bigl|B_{Cq^2}\bigr|
=K\Bigl(\sum_{q=1}^\infty q^{2(1-\frac1\alpha)}\Bigr)C^{-\frac1\alpha}.
\end{align*}
The sum converges for our $\alpha$, %
and the proposition follows since the last expression tends to zero as $C\to\infty$.
\end{proof}

\begin{remark}\label{NONCLOSEDOPTEXPREM}
As we noted in the introduction, Proposition \ref{BGTGENDECAYPROP} implies that for 
$\mu$-almost all $g\in G$, the right hand side in \eqref{NONCLOSEDTHM2RES} in Theorem \ref{NONCLOSEDTHM2} 
decays more rapidly than $T^{\ve-\frac18}$ as $T\to\infty$ ($\forall\ve>0$). 
On the other hand, using the fact that the flow $\{U^t\}$ is mixing on smooth vectors in $L^2(X)$ with a rate
$t^{\ve-1}$ as $t\to\infty$
(as follows from %
\cite{edwards} combined with an argument as in 
\cite[Lemma~2.3]{Ratner87}\footnote{Note that \cite[p.\ 282, line -8]{Ratner87} should be corrected to 
``$u(t)=r(\theta_2)a(\arsinh(t/2))r(\theta_1)$''. Here Ratner's ``$a(t)$'' equals $\Phi^{2t}$ in our notation.}),
one can prove that for sufficiently nice test functions $f$ on $\GaG$,
and for $\mu$-almost all $\Gamma g\in X$, the deviation of the ergodic average in the left hand side of
\eqref{NONCLOSEDTHM2RES} decays like $T^{\ve-\frac12}$ as $T\to\infty$; cf.\ \cite{gaposhkin}.
In this last statement the $\mu$-null set of exceptional points $\Gamma g$ is
\textit{non-explicit} and depends on $f$; furthermore the implied constant in the bound depends on both $f$ and $\Gamma g$
in a non-explicit way; %
the strength of Theorem \ref{NONCLOSEDTHM2} lies of course in the fact that 
it gives a bound where all these dependencies are explicit.
Nevertheless, the discussion %
suggests that it might be possible to improve Theorem \ref{NONCLOSEDTHM2} 
so as to yield a rate of decay $T^{\ve-\frac12}$ for any $\Gamma g\in X$ satisfying an appropriate Diophantine condition.

In this vein, we note that there are \textit{two} steps in our proof 
of Theorem \ref{NONCLOSEDTHM2} which are clearly non-optimal,
each of which causes 
a \textit{halving} of the expected optimal exponent.
The first is when we bound the $d$-sums in \eqref{CONTRFROMFIXEDN} individually for each $c$ using the Weil bound,
and the second is in \eqref{NONCLOSEDSTEP2}, where we replace the integral over the given orbit
with an integral over a nearby orbit which is a lift of a piece of a closed horocycle.
We discussed the first of these in Remark \ref{MAINTHMWEAKERPROPREM}.
Regarding the second step, %
we note that a possible approach for an improved treatment might be to
\textit{rework} the proof of Theorem \ref{MAINABTHM} for the case of an arbitrary $U^t$-orbit,
choosing coordinates in a similar way as in the proofs of 
\cite[Propositions 5.1 and 5.3]{iha}.
\end{remark}

\end{document}